\documentclass[12pt, reqno]{amsart}

\usepackage{
amsmath,
amssymb,
amsthm,
mathrsfs,
amsfonts,
ytableau,
enumitem,
comment,
upgreek,
soul,
yhmath,
mathtools,
shuffle,
kotex,
array,
caption,
tabularx
}
\usepackage{adjustbox}
\usepackage{blindtext}
\usepackage[bb=libus]{mathalpha}

\usepackage[linktocpage]{hyperref}

\hypersetup{
    colorlinks=true,
    linkcolor = blue,
    citecolor=magenta,
    urlcolor=cyan,
}

\usepackage[capitalise, nameinlink]{cleveref}

\usepackage{nicematrix}

\crefname{equation}{}{}
\crefname{figure}{{\sc Figure}}{{\sc Figure}}
\crefname{subsection}{Subsection}{Subsections}

\usepackage[euler]{textgreek}

\usepackage{tikz,tikz-cd}

\usepackage[colorinlistoftodos, textwidth = 2.3cm]{todonotes}

\ytableausetup{mathmode, boxsize=1.2em}

\newtheorem{theorem}{Theorem}[section]
\newtheorem{proposition}[theorem]{Proposition}
\newtheorem{lemma}[theorem]{Lemma}

\newtheorem*{claim*}{Claim}

\theoremstyle{definition}
\newtheorem{algorithm}[theorem]{Algorithm}
\newtheorem{procedure}[theorem]{Procedure}
\newtheorem{example}[theorem]{Example}
\newtheorem{definition}[theorem]{Definition}
\newtheorem{remark}[theorem]{Remark}

\numberwithin{equation}{section} \numberwithin{figure}{section}
\numberwithin{table}{section}

\def\Z{\mathbb Z}
\def\C{\mathbb C}

\newcommand{\nc}{\newcommand}

\nc{\smallsearrow}{\scalebox{0.4}{$\boldsymbol{\searrow}$}}
\nc{\ttau}{\tau}
\nc{\bbd}{\mathbb{d}}
\nc{\bbe}{\mathbb{e}}
\nc{\sfIES}{\mathsf{IES}}
\nc{\sfent}{\mathsf{ent}}
\nc{\sfmIES}{\mathsf{\underline{IES}}}
\nc{\sfRSP}{\mathsf{RSP}}
\nc{\sfdg}{\mathsf{dg}}
\nc{\sfrev}{\mathsf{r}}
\nc{\sfrevmap}{\mathsf{rev}}
\nc{\hsfdg}{\widehat{\tcd}}
\nc{\ttd}{\mathtt{d}}
\nc{\sfP}{\mathsf{P}}
\nc{\bfY}{\mathbf{Y}}
\nc{\sfS}{\mathsf{S}}
\nc{\sfseq}{\mathsf{seq}}
\nc{\sfDS}{\mathsf{DS}}
\nc{\sfst}{\mathsf{st}}
\nc{\tDC}{\mathtt{DC}}
\nc{\scrC}{\mathscr{C}}
\nc{\SR}{\mathcal{SR}}
\nc{\sfposet}
{\mathsf{poset}}
\nc{\rmconj}{\mathrm{conj}}
\nc{\fraki}{\mathfrak{i}}
\nc{\rmset}{\mathrm{set}}
\nc{\rmlex}{\mathrm{lex}}
\nc{\sfCell}{\mathsf{Cell}}
\nc{\ocard}{\overline{\mathrm{card}}}
\nc{\omin}{\overline{\mathrm{min}}}
\nc{\omax}{\overline{\mathrm{max}}}
\nc{\calL}{\mathcal{L}}
\nc{\calK}{\mathcal{K}}
\nc{\calI}{\mathcal{I}}
\nc{\calFo}{\mathcal{Fo}}
\nc{\sfM}{\mathsf{M}}
\nc{\sfm}{\mathsf{m}}
\nc{\Ptab}{\mathsf{P}}
\nc{\hatPtab}{\hat{\mathsf{P}}}
\nc{\checkPtab}{\check{\mathsf{P}}}
\nc{\hatQtab}{\hat{\mathsf{Q}}}
\nc{\checkQtab}{\check{\mathsf{Q}}}
\nc{\Qtab}{\mathsf{Q}}
\nc{\pt}{T^{\scalebox{0.5}{$\nearrow$}}}
\nc{\rmIm}{\mathrm{Im}}
\nc{\rmInvC}{\mathrm{InvCell}}
\nc{\rmInv}{\mathrm{Inv}}
\nc{\rmInt}{\mathrm{Int}}
\nc{\Deq}{\overset{D}{\simeq}}
\nc{\Dleq}{\preceq_D}
\nc{\bdP}{\boldsymbol{P}}
\nc{\bdI}{\boldsymbol{I}}
\nc{\soc}{\mathrm{soc}}
\nc{\balproj}{\bal_{\mathrm{proj}}}
\nc{\balinj}{\bal_{\mathrm{inj}}}

\nc{\SG}{\mathfrak{S}}
\nc{\frakR}{\mathfrak{R}}
\nc{\frakL}{\mathfrak{L}}
\nc{\PCT}{\mathrm{PCT}}
\nc{\SPCT}{\mathrm{SPCT}}
\nc{\RT}{\mathrm{RT}}
\nc{\SRT}{\mathrm{SRT}}
\nc{\SYT}{\mathrm{SYT}}
\nc{\SAF}{\mathrm{SAF}}
\nc{\SSYT}{\mathrm{SSYT}}
\nc{\LR}{\mathrm{LR}}
\nc{\sgn}{\mathrm{sgn}}
\nc{\RCT}{\mathrm{RCT}}
\nc{\SRCT}{\mathrm{SRCT}}
\nc{\SYRT}{\mathrm{SYRT}}
\nc{\nSYCT}{\mathrm{SET}_{\mathbf{sed}}}
\nc{\SYCT}{\mathrm{SYCT}}
\nc{\YCT}{\mathrm{YCT}}
\nc{\CT}{\mathrm{CT}}
\nc{\SCT}{\mathrm{SCT}}
\nc{\SPYCT}{\mathrm{SPYCT}}
\nc{\tst}{\mathtt{st}}
\nc{\Span}{\mathrm{span}}
\nc{\comp}{\mathrm{comp}}
\nc{\rmst}{\mathrm{st}}
\nc{\Des}{\mathrm{Des}}
\nc{\set}{\mathrm{set}}
\nc{\wt}{\mathrm{wt}}
\nc{\ch}{\mathrm{ch}}
\nc{\id}{\mathrm{id}}
\nc{\Sym}{\mathrm{Sym}}
\nc{\Qsym}{\mathrm{QSym}}
\nc{\Nsym}{\mathrm{NSym}}
\nc{\hatsh}{\hat{\mathrm{sh}}}
\nc{\CTra}{\stackrel{{\rm CT}}{\ra}}
\nc{\sh}{\mathrm{sh}}
\nc{\bfS}{\mathbf{S}}
\nc{\bfm}{\mathbf{m}}
\nc{\hbfS}{\hat{\mathbf{S}}}
\nc{\bfF}{\mathbf{F}}

\nc{\tilal}{\widetilde{\alpha}}
\nc{\tilcalW}{\widetilde{\mathcal{W}}}
\nc{\calA}{\mathcal{A}}
\nc{\calB}{\mathcal{B}}
\nc{\calG}{\mathcal{G}}
\nc{\calR}{\mathcal{R}}
\nc{\calS}{\mathcal{S}}
\nc{\hcalS}{\hat{\mathcal{S}}}
\nc{\scrS}{\mathscr{S}}
\nc{\hscrS}{\hat{\mathscr{S}}}
\nc{\calV}{\mathcal{V}}

\nc{\sfR}{\mathsf{R}}

\nc{\tal}{\lambda(\alpha)}
\nc{\tbe}{\widetilde{\beta}}
\nc{\opi}{\overline{\pi}}
\nc{\calP}{\mathcal{P}}
\nc{\rmtop}{\mathrm{top}}
\nc{\rad}{\mathrm{rad}}
\nc{\bfP}{\mathbf{P}}
\nc{\bfw}{\mathbf{w}}
\nc{\SET}{\mathrm{SET}}
\nc{\SIT}{\mathrm{SIT}}
\nc{\SSAF}{\mathrm{SSAF}}
\nc{\SSYCT}{\mathrm{SSYCT}}

\nc{\tcd}{\mathtt{cd}}
\nc{\trcd}{\mathtt{rcd}}
\nc{\tyd}{\mathtt{yd}}
\nc{\trd}{\mathtt{rd}}

\nc{\rmr}{\mathrm{r}}
\nc{\rmc}{\mathrm{c}}
\nc{\rmt}{\mathrm{t}}

\nc{\len}{\mathsf{len}}
\nc{\col}{\mathrm{col}}
\nc{\row}{\mathrm{row}}
\nc{\calE}{\mathcal{E}}
\nc{\calT}{\mathscr{T}}
\nc{\sfT}{\mathsf{T}}
\nc{\calEsa}{\mathcal{E}^\upsig(\alpha)}
\nc{\tauC}{\tau_{\scalebox{0.5}{$C$}}}
\nc{\sytabC}{\sytab_{\scalebox{0.5}{$C$}}}
\nc{\pr}{\mathsf{pr}}
\nc{\Ups}{\Upsilon}
\nc{\pact}{\diamond}
\nc{\tauE}{\tau_{E}^{~}}
\nc{\tauF}{\tau_{\scalebox{0.5}{$F$}}}
\nc{\tauG}{\tau_{\scalebox{0.5}{$G$}}}
\nc{\rtE}{T_{\scalebox{0.5}{$E$}}}
\nc{\rtF}{T_{\scalebox{0.5}{$F$}}}
\nc{\rtG}{T_{\scalebox{0.5}{$G$}}}
\nc{\sytab}{\widehat{\tau}}
\nc{\hatE}{\widehat{E}}
\nc{\hati}{\hat{i}}
\nc{\hcalE}{\widehat{\calE}}
\nc{\hatC}{\widehat{C}}
\nc{\bal}{{\boldsymbol{\upalpha}}}
\nc{\SPYRT}{\mathrm{SPYRT}}

\nc{\bgam}{{\boldsymbol{\upgamma}}}
\nc{\bdel}{{\boldsymbol{\updelta}}}
\nc{\weakcon}{\odot}
\nc{\basisI}{I}

\nc{\ldalpha}{\lambda(\alpha)}
\nc{\SRIT}{\mathrm{SRIT}}
\nc{\re}{\mathrm{rev}}
\nc{\otau}{\overline{\tau}}
\nc{\rtop}{{\rm top}}
\nc{\sfc}{\mathsf{c}}
\nc{\sfC}{\mathsf{C}}
\nc{\sfr}{\mathsf{r}}
\nc{\sfrow}{\mathsf{w_r}}
\nc{\tH}{\mathtt{H}}
\nc{\tV}{\mathtt{V}}
\nc{\rpi}{\mathring{\pi}}
\nc{\cpi}{\check{\pi}}
\nc{\frakm}{\mathfrak{m}}
\nc{\sfem}{\mathsf{em}}
\nc{\Hom}{\mathrm{Hom}}
\nc{\module}{\mathrm{mod} \, }
\nc{\sfw}{\mathsf{w}}
\nc{\sfcol}{\mathsf{w_c}}
\nc{\tread}{\underline{\mathsf{read}}}
\nc{\ocalE}{\overline{\calE}}
\nc{\oE}{\overline{E}}

\nc{\SPCTsa}{\SPCT^\upsig(\alpha)}
\nc{\bfSsa}{\bfS_\alpha^\upsig}
\nc{\bfSsaC}{{\bfS}^\upsig_{\alpha,C}}
\nc{\hbfSsa}{\widehat{\bfS}_\alpha^\upsig}
\nc{\upineq}{\rotatebox{90}{$<$}}
\nc{\downineq}{\rotatebox{270}{$<$}}
\nc{\diagineq}{\rotatebox{135}{$<$}}
\nc{\sfB}{\mathsf{B}}
\nc{\hxi}{\widehat{\xi}}
\nc{\hxidwJ}{\hxi_{\scalebox{0.55}{$J$}}}
\nc{\hxiupJ}{\hxi^{\scalebox{0.55}{$J$}}}
\nc{\bfT}{\mathbf{T}}
\nc{\tshuffle}{\,\widetilde{\shuffle}\,}
\nc{\sJ}{\scalebox{0.55}{$J$}}
\nc{\sJo}{\scalebox{0.55}{$J_1$}}
\nc{\sJt}{\scalebox{0.55}{$J_2$}}
\nc{\Keq}{\overset{K}{\cong}}
\nc{\dKeq}{\overset{K^*}{\cong}}
\nc{\Rect}{\mathrm{Rect}}

\nc{\ra}{\rightarrow}

\nc{\matr}[2]{\left( \hspace{-1ex} \begin{array}{c} #1 \\ #2 \end{array} \hspace{-1ex} \right)}

\definecolor{wsgreen}{rgb}{0,0.5,0}

\nc{\DRIT}{\mathrm{DRIT}}
\nc{\hpi}{\pi}
\nc{\rmInc}{\mathrm{Inc}}
\nc{\hfrakI}{\widehat{\mathfrak{I}}}
\nc{\orho}{\overline{\rho}}
\nc{\autotheta}{\uptheta}
\nc{\calW}{\mathcal{W}}
\nc{\Endo}{\mathrm{End}}

\nc{\autophi}{\upphi}
\nc{\autochi}{\upchi}
\nc{\autoomega}{\upomega}
\nc{\hIM}{\widehat{\sfB}}
\nc{\bfpi}{\boldsymbol{\uppi}}
\nc{\bfopi}{\overline{\boldsymbol{\uppi}}}
\nc{\osfB}{\overline{\sfB}}
\nc{\rmw}{\mathrm{w}}
\nc{\conc}{\; {\bullet} \;}
\nc{\ostar}{\; \overline{\bullet} \;}
\nc{\rank}{\mathrm{rank}}
\nc{\fkp}{\mathfrak{p}}
\nc{\bfR}{\mathbf{R}}
\nc{\calD}{\overline{\mathrm{Des}}}

\nc{\upsig}{{\boldsymbol{\upsigma}}}
\nc{\bfSsaE}{{\bfS}^\upsig_{\alpha,E}}

\nc{\hfkp}{\widehat{\mathfrak{p}}}

\nc{\hautophi}{{\widehat{\autophi}}}
\nc{\hautotheta}{{\widehat{\autotheta}}}
\nc{\hautoomega}{{\widehat{\autoomega}}}
\nc{\rmperm}{\mathrm{perm}}
\nc{\Hnmod}{\text{$H_n(0)$-$\mathbf{mod}$}}
\nc{\modHn}{\text{$\mathbf{mod}$-$H_n(0)$}}

\nc{\bfsigJ}{\sigma_{\scalebox{0.55}{$J$}}}
\nc{\bfrhoJ}{\rho^{\scalebox{0.55}{$J$}}}

\nc{\teta}{\widetilde{\eta}}
\nc{\urmw}{\underline{\mathrm{w}}}
\nc{\sfcnt}{\mathsf{cnt}}

\nc{\pistar}[1]{\pi_{#1}^*}
\nc{\wfkp}{\widetilde{\mathfrak{p}}}
\nc{\bfpsi}{\boldsymbol{\uppsi}}

\nc{\yt}[1]{\todo[size=\tiny,color=blue!10]{#1 \\ \hfill --- Young-Tak}}
\nc{\YT}[1]{\todo[size=\tiny,inline,color=blue!10]{#1
		\\ \hfill --- Young-Tak}}
		
\nc{\yh}[1]{\todo[size=\tiny,color=cyan!10]{#1 \\ \hfill --- Young-Hun}}
\nc{\YH}[1]{\todo[size=\tiny,inline,color=cyan!10]{#1
		\\ \hfill --- Young-Hun}}
		
\nc{\sy}[1]{\todo[size=\tiny,color=magenta!10]{#1 \\ \hfill --- So-Yeon}}
\nc{\SY}[1]{\todo[size=\tiny,inline,color=magenta!10]{#1
		\\ \hfill --- So-Yeon}}

\nc{\nt}[1]{\todo[size=\tiny,color=green!10]{#1 \\ \hfill --- Note}}
\nc{\NT}[1]{\todo[size=\tiny,inline,color=green!10]{#1
		\\ \hfill --- Note}}

\definecolor{purple}{rgb}{0.44, 0.0, 1.0}

\definecolor{yhblue}{rgb}{0,0,0.6}

\newenvironment{red}{\relax\color{red}}{\hspace*{.5ex}\relax}
\newenvironment{blue}{\relax\color{yhblue}}{\hspace*{.5ex}\relax}
\newenvironment{green}{\relax\color{wsgreen}}{\hspace*{.5ex}\relax}
\newenvironment{magenta}{\relax\color{magenta}}{\hspace*{.5ex}\relax}
\newenvironment{purple}{\relax\color{purple}}{\hspace*{.5ex}\relax}

\nc{\ber}{\begin{red}}
\nc{\er}{\end{red}}
\nc{\beb}{\begin{blue}}
\nc{\eb}{\end{blue}}
\nc{\bema}{\begin{magenta}}
\nc{\ema}{\end{magenta}}
\nc{\begr}{\begin{green}}
\nc{\egr}{\end{green}}

\nc{\bepu}{\begin{purple}}
\nc{\epu}{\end{purple}}

\title[Distinguished filtrations]{Distinguished filtrations of the $0$-Hecke modules  
for dual immaculate quasisymmetric functions  
}

\author[S.-Y. Lee]{So-Yeon Lee}
\address[S.-Y. Lee]{Department of Mathematics, Sogang University, Seoul 04107, Republic of Korea}
\email{sylee0814@sogang.ac.kr}

\author[Y.-T. Oh]{Young-Tak Oh}
\address[Y.-T. Oh]{Department of Mathematics, Sogang University, Seoul 04107, Republic of Korea}
\email{ytoh@sogang.ac.kr}


\date{\today}

\begin{document}

\maketitle

\begin{abstract}
Let \( \alpha \) range over the set of compositions. 
Dual immaculate quasisymmetric functions \( \SG_\alpha^* \), introduced by Berg, Bergeron, Saliola, Serrano, and Zabrocki, provide a quasisymmetric analogue of Schur functions. They also constructed an indecomposable \( 0 \)-Hecke module \( \calV_\alpha \) whose image under the quasisymmetric characteristic is \( \SG_\alpha^* \). 
In this paper, we prove that \( \mathcal{V}_\alpha \) admits a distinguished filtration with respect to the basis of Young quasisymmetric Schur functions. 
This result offers a novel representation-theoretic interpretation of  the positive expansion of  \( \SG_\alpha^* \) in the basis of Young quasisymmetric Schur functions.
A key tool in our proof is Mason’s analogue of the Robinson-Schensted-Knuth algorithm, for which we establish a version of Green's theorem. 
As an unexpected byproduct of our investigation, we construct an indecomposable \( 0 \)-Hecke module \( \mathbf{Y}_\alpha \) whose image under the quasisymmetric characteristic is the Young quasisymmetric Schur function \( \hscrS_\alpha \).
Further properties of this module are also investigated.
And, by applying a suitable automorphism twist to this module, we obtain an indecomposable \( 0 \)-Hecke module whose image under the quasisymmetric characteristic is the quasisymmetric Schur function \( \scrS_\alpha \).
\end{abstract}

\tableofcontents

\section{Introduction}

Let \( n \) be a nonnegative integer. 
The \( 0 \)-Hecke algebra \( H_n(0) \) is defined by setting \( q = 0 \) in the generic Hecke algebra \( H_n(q) \). However, the representation theory of \( H_n(0) \) differs significantly from that of the generic Hecke algebra. 
In particular, \( H_n(0) \) is neither semi-simple nor representation-finite for \( n \ge 4 \); it is of tame type when \( n = 4 \) and of wild type for \( n > 4 \) (see~\cite{11DY, 02DHT}).

The representation theory of \(0\)-Hecke algebras has deep connections with quasisymmetric functions. Let \( \calG_0(\Hnmod) \) denote the Grothendieck group of the category \( \Hnmod \) of finitely generated left \( H_n(0) \)-modules and let \( \Qsym \) be the ring of quasisymmetric functions. 
In \cite{96DKLT, 09BL}, it has been shown that the direct sum \( \bigoplus_{n \ge 0} \calG_0(\Hnmod) \), equipped with a Hopf algebra structure defined by the induction product and restriction, is isomorphic to \( \Qsym \) under the map
\[
 \ch : \bigoplus_{n \ge 0} \calG_0(\Hnmod) \to \Qsym, \quad [\bfF_\alpha] \mapsto F_{\alpha},
\] known as \emph{quasisymmetric characteristic}. Here, \( \alpha \) is a composition, \( \bfF_\alpha \) is the irreducible \( H_{|\alpha|}(0) \)-module indexed by \( \alpha \), and \( F_\alpha \) is the fundamental quasisymmetric function indexed by \( \alpha \) (see Section \ref{subsec: 0-Hecke alg and QSym}).

Given an index set \( I \), let \( \mathcal{B} = \{\mathcal{B}_\alpha \mid \alpha \in I\} \) be a linearly independent subset of the \( n \)th homogeneous component \( \Qsym_n \) of \( \Qsym \), where each \( \mathcal{B}_\alpha \) has a positive expansion in the basis \( \{F_\alpha \mid \alpha \models n\} \). 
For an  \( H_n(0) \)-module \( M \), a \emph{distinguished filtration of \( M \) with respect to \( \mathcal{B} \)} is defined to be an \( H_n(0) \)-submodule series  
\[
0 =: M_0 \subsetneq M_1 \subsetneq M_2 \subsetneq \cdots \subsetneq M_l := M,
\]
such that for each \( 1 \leq k \leq l \), there exists \( \alpha \in I \) satisfying \( \ch([M_k / M_{k-1}]) = \mathcal{B}_\alpha \). 
It is important to observe that such a distinguished filtration of \( M \) with respect to \( \mathcal{B} \) may not exist, even if \( \ch([M]) \) admits a positive expansion in terms of \( \mathcal{B} \).
The concept of distinguished filtrations was introduced in \cite[Section 6]{23KLO} in the context of studying \( 0 \)-Hecke modules associated with regular Schur labeled posets. 
Roughly speaking, a regular Schur labeled skew shape poset \( P \) is a poset where the set of linear extensions forms a weak Bruhat interval and the \( P \)-partition generating function is symmetric. 
It was shown therein that for any regular Schur labeled poset \( P \) with underlying set $[n]:=\{1,2,\ldots, n\}$, the associated \( H_n(0)\)-module \( \sfM_P \) admits a distinguished filtration with respect to the Schur basis \( \{s_\lambda \mid \lambda \vdash n\} \).

In this paper, we consider two bases for \(\Qsym_n\): the basis \(\calS := \{\scrS_\alpha \mid \alpha \models n\}\) consisting of quasisymmetric Schur functions, and the basis \(\hcalS := \{\hscrS_\alpha \mid \alpha \models n\}\)  consisting of Young quasisymmetric Schur functions. The quasisymmetric Schur functions are widely viewed as a natural quasisymmetric analogue of the Schur functions (see \cite{11HLMW, 11BLW, 11HLMW2, 13LMvW}).
Meanwhile, the Young quasisymmetric Schur functions are derived from the quasisymmetric Schur functions through the automorphism \(\uprho: \Qsym \to \Qsym\), defined by \(F_\alpha \mapsto F_{\alpha^\rmr}\), where \(\alpha^\rmr\) denotes the reverse of the composition \(\alpha\). 
With these bases, we study distinguished filtrations of the \(H_n(0)\)-modules for the dual immaculate quasisymmetric functions and the extended Schur functions.

The dual immaculate quasisymmetric functions \( \SG_\alpha^* \), introduced by Berg, Bergeron, Saliola, Serrano, and Zabrocki \cite{14BBSSZ}, serve as quasisymmetric analogues of the immaculate noncommutative symmetric functions. The set \( \{\SG_\alpha^* \mid \alpha \models n\} \) forms a basis for \( \Qsym_n \), which has been extensively studied in the context of quasisymmetric functions and the representation theory of \( H_n(0) \) (for instance, see \cite{15BBSSZ, 16BSZ, 17BBSSZ, 22CKNO2, 22CKNO}). 
In particular, in \cite{15BBSSZ}, the authors constructed an indecomposable \( H_n(0) \)-module \( \calV_\alpha \) and showed that its image under the quasisymmetric characteristic  is \( \SG_\alpha^* \).
Meanwhile, Allen, Hallam, and Mason \cite{18AHM} demonstrated that the dual immaculate quasisymmetric functions expand positively in  \( \hcalS \). Therefore, it is natural to ask whether the module \( \calV_\alpha \) admits a distinguished filtration with respect to  \( \hcalS \).

The extended Schur functions \( \calE_\alpha \), introduced by Assaf and Searles \cite{19AS} as stable limits of lock polynomials, also constitute a basis for \( \Qsym \). In \cite[Theorems 3.10 and 3.13]{19Searles}, Searles constructed an indecomposable \( H_n(0) \)-module \( X_\alpha \) and showed that its image under the quasisymmetric characteristic is \( \calE_\alpha \). 
Recently, Marcum and Niese \cite[Theorem 3.5]{24MN} showed that if \( \alpha \) is a composition of \( n \) obtained by shuffling a partition and \( (1^k) \) for some \( k \geq 0 \), then \(\calE_\alpha  \) expands positively in  \( \hcalS \). 
Analogous to the case of \( \calV_\alpha \), this naturally leads to the question of whether \( X_\alpha \) admits a distinguished filtration with respect to  \( \hcalS \).

The primary objective of this paper is to address the aforementioned questions. The main results obtained in this work are outlined below.

In \cref{An analogue of RSK and its Greene Theorem}, we deal with Mason's analogue of the Robinson-Schensted-Knuth algorithm, which is the most important tool in our study.
In \cref{subsec: an analog of RSK}, following \cite[Procedure 4.2]{06Mason}, we introduce an algorithm that, given a two-line array \( w \), produces a pair of Young composition tableaux \( (\hatPtab(w), \hatQtab(w)) \) (\cref{alg: modified mason algotithm}). We then investigate the similarities and differences between this algorithm and the Robinson-Schensted-Knuth algorithm.
In \cref{subsec: the corresponding Greene's theorem}, we present a theorem that determines the shapes of the insertion tableaux \( \hatPtab(w) \) (\cref{thm: an analoge of Greene}).  
In the case of the Robinson-Schensted-Knuth algorithm, this problem was resolved by Greene \cite{74Gre}.  
For our purpose, we define the concept of the \emph{initial entries set} of a longest \( k \)-increasing subsequence of \( w \) (see \cref{def: initial entry set}).

In \cref{Distinguished filtrations of V and X}, we present answers to the aforementioned questions. 
It was shown in \cite[Theorem 4.4]{22JKLO} that both \( \calV_\alpha \) and \( X_\alpha \) are equipped with the structures of weak Bruhat interval modules. Specifically,  
\[
\calV_\alpha \cong \sfB(\sfrow(\calT_\alpha), \sfrow(\calT'_\alpha)) \quad \text{and} \quad X_\alpha \cong \sfB(\sfrow(\sfT_\alpha), \sfrow(\sfT'_\alpha))
\]
(for details, see \cref{Weak Bruhat interval modules}).
In \cref{Distinguished filtrations of V}, we first define an equivalence relation \( \simeq_M \) on \( \SG_n \), which refines the dual Knuth equivalence relation. 
We then show that \( [\sfrow(\calT_\alpha), \sfrow(\calT'_{\alpha})]_L \cdot w_0 \) is closed under \( \simeq_M \), and a suitable order relation exists among the shapes of \( \hatPtab(\sigma) \) for all \( \sigma \in [\sfrow(\calT_\alpha), \sfrow(\calT'_{\alpha})]_L \cdot w_0 \). 
Finally, using these results, we prove that for any composition \( \alpha \) of \( n \), the module \( \calV_\alpha \) admits a distinguished filtration with respect to  \( \hcalS \) (\cref{thm: dist filt for Va}).
In \cref{Distinguished filtrations of X}, we restrict \( \alpha \) to a composition of \( n \) obtained by shuffling a partition and \( (1^k) \) for some \( k \geq 0 \). 
We first show that \( [\sfrow(\sfT_\alpha), \sfrow(\sfT'_{\alpha})]_L \cdot w_0 \) is closed under \( \simeq_M \). 
Then, following the approach used in the proof of \cref{thm: dist filt for Va}, we prove that \( X_\alpha \) admits a distinguished filtration with respect to \( \hcalS \) (\cref{thm: dist filt for Xa}).
It is worth noting that if \(\lambda\) is a partition of \(n\), then \(\ch([X_\lambda]) = s_\lambda\). However, \(X_\lambda\) may not necessarily admit a distinguished filtration with respect to \(\calS\)
(see \cref{no filtration example} and \cref{fig: appendix example}).

\cref{Indecomposable modules for (Young) quasisymmetric Schur functions} concerns a significant open question in the representation theory of \(0\)-Hecke algebras: whether there exists an indecomposable $H_n(0)$-module \(M\) whose image under the quasisymmetric characteristic equals \(\scrS_{\alpha}\) or \(\hscrS_{\alpha}\). In 2015, Tewari and van Willigenburg \cite{15TW} constructed an \(H_n(0)\)-module \(\bfS_\alpha\) that satisfies \(\ch([\bfS_\alpha]) = \scrS_{\alpha}\). 
However, \(\bfS_\alpha\) is not generally indecomposable.
In \cref{subsec: indecomp mod for YSal and Sal}, we provide an indecomposable \( H_n(0) \)-module \( \bfY_\alpha \) whose image under the quasisymmetric characteristic  is \( \hscrS_{\alpha} \) (\cref{thm: new modules for quasisymmetric Schur and Young quasi Schur}).  
To be precise, for any distinguished filtration  
\[
0 = M_0 \subsetneq M_1 \subsetneq \cdots \subsetneq  M_{l-1} \subsetneq M_l = \sfB(\sfrow(\calT_\alpha), \sfrow(\calT'_\alpha))
\]  
of \( \sfB(\sfrow(\calT_\alpha), \sfrow(\calT'_\alpha)) \), as constructed in the proof of \cref{thm: dist filt for Va},
we define \( \bfY_\alpha \) as the quotient 
$M_l/M_{l-1}$. 
It turns out that this module is independent of the choice of the filtration. 
We then provide a basis \( \calK_\alpha \subseteq [\sfrow(\calT_\alpha), \sfrow(\calT'_\alpha)]_L \) for \( \bfY_\alpha \) (see \eqref{eq: def of K alpha}).
In \cref{subsec: A surjection series containing Yal}, we construct a surjection series  
\[
\overline{\bfP}_{\alpha^{\rmc}} \stackrel{}{\twoheadrightarrow} \calV_\alpha \twoheadrightarrow \bfY_\alpha \twoheadrightarrow \hbfS_{\alpha, C}
\]  
(\cref{prop: existence of Gamma map}).  
Additionally, we show that this series can be extended when \( \alpha \) is obtained by shuffling a partition with \( (1^k) \) for some \( k \geq 0 \):  
\[
\overline{\bfP}_{\alpha^{\rmc}} \twoheadrightarrow \calV_\alpha \twoheadrightarrow X_\alpha \twoheadrightarrow \bfY_\alpha \twoheadrightarrow \hbfS_{\alpha, C}
\]  
(\cref{prop: exitence of tildel map}).  
In \cref{subsec: Weak Bruhat interval module structure for Yal}, we show that \(\calK_\alpha\) is a weak Bruhat interval when \( \alpha \) is a composition obtained by shuffling a partition and \( (1^k) \) for some \( k \geq 0 \).  
To achieve this, we first introduce the concept of \emph{southeast decreasing fillings} (see \cref{def: SE-decreasing}).  
Using this concept, we characterize  the elements in $\calK_\alpha$. 
We then construct a specific filling \( \tau'_\alpha \) and show that  
$\calK_\alpha = [\sfrow(\sfT_\alpha), \sfrow(\tau'_\alpha)]_L$
(\cref{prop: description for Gamma wrt nSYCT} and \cref{second main result of sec 5.3}).

\section{Preliminaries}

For integers $m$ and $n$, we define $[m,n]$ and $[n]$ to be the intervals $\{t\in \mathbb Z \mid m\le t \le n\}$ and $\{t\in \mathbb Z \mid 1\le t \le n\}$, respectively.
Throughout this paper, $n$ will denote a nonnegative integer unless otherwise stated.

\subsection{Compositions and their diagrams}\label{subsec: comp and diag}

A \emph{composition} $\alpha$ of $n$, denoted by $\alpha \models n$, is a finite ordered list of positive integers $(\alpha_1,\alpha_2,\ldots, \alpha_k)$ satisfying $\sum_{i=1}^k \alpha_i = n$.
We call $\alpha_i$ ($1 \le i \le k$) a \emph{part} of $\alpha$, $k =: \ell(\alpha)$ the \emph{length} of $\alpha$, and $n =:|\alpha|$ the \emph{size} of $\alpha$. 
And, we define the empty composition $\varnothing$ to be the unique composition of size and length $0$.
Whenever necessary, we set $\alpha_i = 0$ for all $i > \ell(\alpha)$.
We represent each composition visually using its composition diagram.
The \emph{composition diagram} $\tcd(\alpha)$ of $\alpha$ is the array of left-justified cells, with 
$\alpha_i$ cells in the $i$th row from the bottom for $1 \leq i \leq k$. 
If $\alpha=(3,1,2)$, then 
\[
\tcd(\alpha) =
\begin{array}{c}
\begin{ytableau}
    {\color{white}1} & \text{}\\
    {\color{white}1}\\
    {\color{white}1} & {\color{white}1} & \text{}
\end{ytableau}
\end{array}.
\]
Each cell is described by its row and column indices, with the coordinate $(i,j)$ denoting the 
cell in the 
$i$th row from  the bottom  and the 
$j$th column from the left. 
We often use the set of coordinates of the cells in $\tcd(\alpha)$ to denote $\tcd(\alpha)$ itself.
For example, 
$$\tcd((3,1,2))=\{(1,1),(1,2),(1,3),(2,1),(3,1), (3,2)\}.$$

For a composition $\alpha = (\alpha_1,\alpha_2,\ldots,\alpha_k) \models n$ and a set $I = \{i_1 < i_2 < \cdots < i_{l}\} \subseteq [n-1]$, 
let 
\begin{align*}
    \set(\alpha) &:= \{\alpha_1,\alpha_1+\alpha_2,\ldots, \alpha_1 + \alpha_2 + \cdots + \alpha_{k-1}\}, \text{ and} \\\comp(I) &:= (i_1,i_2 - i_1, i_3 - i_2, \ldots,n-i_{l}).
\end{align*}
The set of all compositions of $n$ is in bijection with the set of all subsets of $[n-1]$ under the correspondence $\alpha \mapsto \set(\alpha)$ (or $I \mapsto \comp(I)$).
The \emph{reverse composition $\alpha^\rmr$ of $\alpha$} is defined to be the composition $(\alpha_k, \alpha_{k-1}, \ldots, \alpha_1)$ and
the \emph{complement $\alpha^\rmc$ of $\alpha$} is defined to be the unique composition satisfying $\set(\alpha^c) = [n-1] \setminus \set(\alpha)$.

If a composition $\lambda = (\lambda_1, \lambda_2, \ldots, \lambda_k) \models n$ satisfies $\lambda_1 \ge \lambda_2 \ge \cdots \ge \lambda_k$, then it is called a \emph{partition} of $n$ and denoted as $\lambda \vdash n$.
Given two partitions $\lambda$ and $\mu$ with $\ell(\lambda) \ge \ell(\mu)$, we write $\lambda \supseteq \mu$ if $\lambda_i \geq  \mu_i$ for all $1 \le i \le \ell(\mu)$. 
A \emph{skew partition} $\lambda/\mu$ is a pair  $(\lambda, \mu)$ of partitions with $\lambda \supseteq \mu$.
We call $|\lambda/\mu| := |\lambda| - |\mu|$ the \emph{size} of $\lambda/\mu$.
When $\lambda$ is a partition, we typically use $\tyd(\lambda)$ to denote $\tcd(\lambda)$ and refer to it as the \emph{Young diagram} of $\lambda$ (in French notation). For a skew partition $\lambda/\mu$, the \emph{Young diagram $\tyd(\lambda/\mu)$ of $\lambda/\mu$} is defined as the Young diagram $\tyd(\lambda)$ with the cells corresponding to $\tyd(\mu)$ removed. 
We denote by $\SYT(\lambda / \mu)$ the set of all standard Young tableaux of shape $\lambda / \mu$.
And, we let $\SYT_n := \bigsqcup_{\lambda \vdash n} \SYT(\lambda)$.
Given a composition $\alpha$, let $\lambda(\alpha)$ be the partition obtained by rearranging the entries of $\alpha$ in decreasing order.

\textbf{Convention.}
Let $D$ be a composition diagram or a  Young diagram of skew shape.
In this paper, we regard a filling $T$ of $D$ with positive integers as a map 
$$T: D \to \Z_{>0}, \quad (i,j) \mapsto T(i,j),$$
where $T(i,j)$ denotes the entry in the cell $(i,j)$ of $T$.
Given a filling $T$ of $D$, we primarily utilize two reading words, 
$\sfrow(T)$ and $\sfcol(T)$, defined as follows: 
\begin{itemize}
    \item $\sfrow(T)$ is the word obtained by reading the entries of $T$ from right to left starting with the bottom-most row.
    \item $\sfcol(T)$ is the word obtained by reading the entries of $T$ from bottom to top starting with the rightmost column.
\end{itemize}
Unless otherwise stated, the \(i\)th row of \(T\) refers to the \(i\)th row from the bottom and the \(j\)th column of \(T\) refers to the \(j\)th column from the left.

\subsection{The left weak Bruhat order on the symmetric group}\label{subsec: Symmetric group}

Let $\SG_n$ denote the symmetric group on $[n]$. Every permutation $\sigma \in \SG_n$ can be expressed as a product of simple transpositions $s_i := (i,i+1)$ for $1 \leq i \leq n-1$.
A \emph{reduced expression for $\sigma$} is an expression that represents $\sigma$ in the shortest possible length and 
the \emph{length $\ell(\sigma)$ of $\sigma$} is the number of simple transpositions in any reduced expression for $\sigma$.
The \emph{left weak Bruhat order} $\preceq_L$  on $\SG_n$
is the partial order on $\SG_n$ whose covering relation $\preceq_L^{\rmc}$ is given as follows: 
\[
\sigma \preceq_L^\rmc s_i \sigma \;\text{ if }\ell(\sigma)<\ell(s_i \sigma).
\]
Let
\[
\Des_L(\sigma):= \{1 \leq i \leq n-1 \mid \ell(s_i \sigma) < \ell(\sigma)\}.
\]
It is well known that if $\sigma = w_1 w_2 \ldots w_n$ in one-line notation, then
\begin{align*}
\begin{aligned}
\Des_L(\sigma) & = \{ 1 \leq i \leq n-1 \mid \text{$i$ is right of $i+1$ in $w_1 w_2 \ldots w_n$} \}.
\end{aligned} 
\end{align*}
For each $\gamma \in \SG_n$, let
\begin{align*}
\rmInv_L(\gamma) & := \{(i,j) \mid 1 \le i < j \le n \text{ and } \gamma(i) > \gamma(j) \}.
\end{align*}
Then, for $\sigma, \rho \in \SG_n$, 
\begin{align*}
& \sigma \preceq_L \rho 
\quad \text{if and only if} \quad
\rmInv_L(\sigma) \subseteq \rmInv_L(\rho).
\end{align*}
Given $\sigma, \rho \in \SG_n$, the \emph{left weak Bruhat interval $[\sigma, \rho]_L$}  denotes the closed interval $\{\gamma \in \SG_n \mid \sigma \preceq_L \gamma \preceq_L \rho \}$ with respect to the left weak Bruhat order $\preceq_L$.

We collect notations associated with the symmetric group that will be used throughout the paper. For any subset $I \subseteq [n-1]$, let $\SG_I$ denote the parabolic subgroup of $\SG_n$ generated by $\{s_i \mid i \in I\}$ and let $w_0(I)$ denote the longest element in $\SG_I$. 
When $I = [n-1]$, we simply write $w_0$ to denote the longest element in $\SG_n$. 
For $\alpha \models n$, we define 
\[
w_0(\alpha) := w_0(\text{set}(\alpha)).
\]
Additionally, for any $\sigma \in \SG_n$, we define $\sigma^{w_0} := w_0 \sigma w_0$. Finally, for $S \subseteq \SG_n$ and $\xi \in \SG_n$, we define 
\[
S \cdot \xi := \{\gamma \xi \mid \gamma \in S\}
\quad \text{and} \quad 
\xi \cdot S := \{\xi \gamma \mid \gamma \in S\}.
\]

\subsection{The $0$-Hecke algebra and its representation theory}\label{subsec: 0-Hecke alg and QSym}
The $0$-Hecke algebra $H_n(0)$ is the associative $\C$-algebra with $1$ generated by $\pi_1,\pi_2,\ldots,\pi_{n-1}$ subject to the following relations:
\begin{align*}
\pi_i^2 &= \pi_i \quad \text{for $1\le i \le n-1$},\\
\pi_i \pi_{i+1} \pi_i &= \pi_{i+1} \pi_i \pi_{i+1}  \quad \text{for $1\le i \le n-2$},\\
\pi_i \pi_j &=\pi_j \pi_i \quad \text{if $|i-j| \ge 2$}.
\end{align*}
For any reduced expression $s_{i_1} s_{i_2} \cdots s_{i_p}$ for $\sigma \in \SG_n$, let 
$\pi_{\sigma} := \pi_{i_1} \pi_{i_2 } \cdots \pi_{i_p}$. 
It is well known that these elements are independent of the choices of reduced expressions and  $\{\pi_\sigma \mid \sigma \in \SG_n\}$ 
is a $\mathbb C$-bases for $H_n(0)$.

Let us briefly review the representation theory of $0$-Hecke algebras. Unless stated otherwise, all $H_n(0)$-modules are finitely generated and carry a left action.
By \cite{79Norton}, there are $2^{n-1}$ pairwise nonisomorphic irreducible $H_n(0)$-modules, naturally indexed by compositions of $n$.
To be precise, for each $\alpha \models n$, 
there exists an irreducible $H_n(0)$-module $\bfF_{\alpha}:=\C v_{\alpha}$ endowed with the $H_n(0)$-action defined as follows: for each $1 \le i \le n-1$,
\[
\pi_i \cdot v_\alpha = \begin{cases}
0 & i \in \set(\alpha),\\
v_\alpha & i \notin \set(\alpha).
\end{cases}
\]

Let $\Hnmod$ be the category of finitely generated left $H_n(0)$-modules and $\calR(H_n(0))$ the $\Z$-span of the set of (representatives of) isomorphism classes of modules in $\Hnmod$.
We denote by $[M]$ the isomorphism class corresponding to an $H_n(0)$-module $M$. 
The \emph{Grothendieck group $\calG_0(H_n(0))$ of $\Hnmod$} is the quotient of $\calR(H_n(0))$ modulo the relations $[M] = [M'] + [M'']$ whenever there exists a short exact sequence $0 \ra M' \ra M \ra M'' \ra 0$. 
The equivalence classes of the irreducible $H_n(0)$-modules form a $\Z$-basis for $\calG_0(H_n(0))$. Let
\[
\calG := \bigoplus_{n \ge 0} \calG_0(H_n(0)).
\]

Let us review the connection between $\calG$ and the ring $\Qsym$ of quasisymmetric functions.
For the definition of quasisymmetric functions, see~\cite[Section 7.19]{99Stanley}.
For a composition $\alpha$, the \emph{fundamental quasisymmetric function} $F_\alpha$, which was firstly introduced in~\cite{84Gessel}, is defined by
\[
F_\varnothing = 1 
\quad \text{and} \quad
F_\alpha = \sum_{\substack{1 \le i_1 \le i_2 \le \cdots \le i_n \\ i_j < i_{j+1} \text{ if } j \in \set(\alpha)}} x_{i_1}x_{i_2} \cdots x_{i_n} \quad \text{if $\alpha \neq \varnothing$}.
\]
It is known that $\{F_\alpha \mid \text{$\alpha$ is a composition}\}$ is a $\mathbb Z$-basis for $\Qsym$.
When $M$ is an $H_m(0)$-module and $N$ is
an $H_n(0)$-module, we write $M \boxtimes N$ for the induction product of $M$ and $N$, that is,
\begin{align*}
M \boxtimes N := M \otimes N \uparrow_{H_m(0) \otimes H_n(0)}^{H_{m+n}(0)}.
\end{align*}
Here, $H_m(0) \otimes H_n(0)$ is viewed as the subalgebra of $H_{m+n}(0)$ generated by $\{\pi_i \mid i \in [m+n-1] \setminus \{m\} \}$.
The induction product induces a multiplication on $\calG$.
It was shown in \cite{96DKLT} that the linear map
\begin{equation*}
\ch : \calG \ra \Qsym, \quad [\bfF_{\alpha}] \mapsto F_{\alpha},
\end{equation*}
called \emph{quasisymmetric characteristic}, is a ring isomorphism.

\subsection{Bases for $\Qsym$ and related $0$-Hecke modules}
\label{Bases for Qsym and related 0 Hecke modules}
In this paper, we deal with the following quasi-analogues of Schur functions: quasisymmetric Schur functions, Young quasisymmetric Schur functions, dual immaculate quasisymmetric functions, and extended Schur functions. We provide a brief introduction to each of these functions. Furthermore, for the latter two, we introduce indecomposable \(0\)-Hecke modules that have these functions as images under the quasisymmetric characteristic.

\subsubsection{Quasisymmetric Schur functions, Young quasisymmetric Schur functions, and related $0$-Hecke modules }\label{subsec: quasi Schur ftn}

Quasisymmetric Schur functions were first introduced in \cite[Definition 5.1]{11HLMW} and arise from the specialization of nonsymmetric Macdonald polynomials to standard bases, commonly known as Demazure atoms. Moreover, their \( F \)-expansions, where \( F \) denotes the basis of fundamental quasisymmetric functions, are also provided. Here, we introduce this expansion.

\begin{definition}
Let $\alpha$ be a composition of $n$. 
A \emph{composition tableau} of shape $\alpha$ is a filling $T$ of $\tcd(\alpha)$ with entries in $\Z_{>0}$  satisfying the following conditions:
\begin{enumerate}
  \item The entries in each row of $T$ weakly decrease  from left to right.
  \item  The entries in the leftmost column of $T$ strictly increase  from top to bottom.
  \item Triple rule: for any cells $(i, k+1), (j, k) \in \tcd(\alpha)$ with $i < j$,  
  \[
  \text{if $T(i, k+1) \leq T(j, k)$,  then
  $(j, k+1) \in \tcd(\alpha)$ and 
  $T(i, k+1) < T(j, k+1)$.}
  \]
\end{enumerate}
\end{definition}

Let \(\alpha \models n\). A \emph{standard composition tableau} (SCT) of shape \(\alpha\) is a composition tableau in which each number in $[n]$ appears exactly once.  
Let \(\SCT(\alpha)\) denote the set of all standard composition tableaux of shape \(\alpha\). For each \(T \in \SCT(\alpha)\), define  
\[
\Des_{\scrS}(T) := \{1 \leq i \leq n-1 \mid \text{\(i+1\) appears weakly to the right of \(i\) in \(T\)}\}.
\]  
It was shown in \cite[Section 5]{11HLMW} that  
\[
\scrS_\alpha = \sum_{T \in \SCT(\alpha^\rmr)} F_{\comp(\Des_{\scrS}(T))}
\]  
and \(\{\scrS_\alpha \mid \alpha \models n\}\) forms a basis for \(\Qsym_n\).
 
In 2015, Tewari and van Willigenberg \cite{15TW} introduced an \(H_n(0)\)-action on the \(\C\)-span of \(\SCT(\alpha^\rmr)\). The resulting \(H_n(0)\)-module, denoted by \(\bfS_\alpha\), satisfies  
\[
\ch([\bfS_\alpha]) = \scrS_\alpha.
\]
This module is not indecomposable in general.  
For its decomposition into indecomposable modules, see \cite{19Konig, 20CKNO}.
Define  
\[
\SCT(\alpha^\rmr; C) := \{T \in \SCT(\alpha^\rmr) \mid \text{each column of \(T\) increases from top to bottom}\}.
\]  
The \(\C\)-span of \(\SCT(\alpha^\rmr; C)\) forms an indecompsable $H_n(0)$-submodule of \(\bfS_\alpha\), which is denoted by \(\bfS_{\alpha, C}\) and referred to as the \emph{canonical submodule} of \(\bfS_\alpha\).

On the other hand, the \emph{Young quasisymmetric Schur function} $\hscrS_\alpha$ is defined as  $\uprho(\scrS_{\alpha^\rmr})$, where $\uprho: \Qsym \to \Qsym$ is the automorphism defined by $F_\alpha \mapsto F_{\alpha^\rmr}$. 
These functions exhibit many properties analogous to those of quasisymmetric Schur functions.  
Additionally, the Young quasisymmetric function $\hscrS_\alpha$ can be expressed using the combinatorial object known as Young composition tableaux.

\begin{definition}
Let $\alpha$ be a composition of $n$. 
A {\it Young composition tableau} of shape $\alpha$ is a filling $T$ of $\tcd(\alpha)$ with entries in $\Z_{>0}$  satisfying the following conditions:
\begin{enumerate}
    \item  The entries in each row weakly increase from left to right,
    \item The entries in the first column  strictly increase from bottom to top, and
    \item Young triple rule:
    for any cells $(i, k+1), (j, k) \in \tcd(\alpha)$ with $ i< j$,
    \[
    \text{if  $T(i, k+1) \geq T(j, k)$, then 
    $(j, k+1) \in \tcd(\alpha)$ and 
    $T(i, k+1) > T(j, k+1)$.}
    \]
    
\end{enumerate}
\end{definition}
A \emph{standard Young composition tableau} of shape $\alpha$ is a Young composition tableau of shape $\alpha$ in which each of the numbers in $[n]$ appears exactly once.
We denote the set of all standard Young composition tableaux of shape $\alpha$ by $\SYCT(\alpha)$.
For each \(T \in \SYCT(\alpha)\), we define 
\begin{equation}\label{eq: descent of Young comoposition tableau}
\Des_{\hscrS}(T) := \{1 \leq i \leq n-1 \mid \text{$i$ appears weakly right of $i+1$ in $T$} \}.
\end{equation}
It was shown in \cite[Proposition 5.2.2]{13LMvW} that 
\[
\hscrS_{\alpha} = \sum_{T \in \SYCT(\alpha)} F_{\comp(\Des_{\hscrS}(T))}
\]
(see \cite[Section 4 and Section 5]{13LMvW}).

In 2022, Choi, Kim, Nam, and Oh  introduced an \(H_n(0)\)-action on the \(\C\)-span of \(\SYCT(\alpha)\). The resulting \(H_n(0)\)-module, denoted by \(\hbfS_\alpha\), satisfies that  
$\ch([\hbfS_\alpha]) = \hscrS_\alpha$. 
This module is also not indecomposable in general.
Define  
\[
\SYCT(\alpha; C) := \{T \in \SYCT(\alpha) \mid \text{each column of \(T\) increases from bottom to top}\}.
\]  
The \(\C\)-span of \(\SYCT(\alpha; C)\) forms an indecomposable submodule of \(\hbfS_\alpha\), denoted by \(\hbfS_{\alpha, C}\), which is referred to as the \emph{canonical submodule} of \(\hbfS_\alpha\) (see \cite[Subsection 4.2]{22CKNO}).

\subsubsection{
Dual immaculate quasisymmetric functions and related $0$-Hecke modules}

For each composition $\alpha$, the dual immaculate quasisymmetric function \(\SG_\alpha^*\) was initially introduced as the quasisymmetric dual of the corresponding immaculate noncommutative symmetric function in \cite[Section 3.7]{14BBSSZ}. 
Similar to \(\scrS_\alpha\) and \(\hscrS_\alpha\), it has a positive \(F\)-expansion, which can be expressed in terms of tableaux.

\begin{definition}(\cite[Definition 3.9]{14BBSSZ})\label{def: SIT}
Let $\alpha$ be a composition of $n$. 
A \emph{standard immaculate tableau} of shape $\alpha$ is a filling $\calT$ of $\tcd(\alpha)$ with the  entries $1, 2, \ldots, n$ such that the entries are all distinct, the entries in each row increase from left to right, and the entries in the first column increase  from bottom to top. 
\end{definition}
Let $\SIT(\alpha)$ denote the set of all standard immaculate tableaux of shape $\alpha$. 
It was shown in \cite[Proposition 3.48]{14BBSSZ} that 
\[
\SG_\alpha^* = \sum_{T \in \SIT(\alpha)} F_{\comp(\Des_L(\sfrow(T)))^\rmc}.
\]
Then, Berg {\it et al.} \cite{15BBSSZ} constructed an \(H_n(0)\)-module structure for the dual immaculate quasisymmetric function \(\SG_\alpha^*\).
Define an $H_n(0)$-action on the $\C$-span of $\SIT(\alpha)$ by
\begin{align*}
    \pi_i \cdot \calT = \begin{cases}
      \calT  & \text{if $i$ appears weakly above of $i+1$ in $\calT$,}\\
      0  & \text{if $i$ and $i+1$ are in the first column of $\calT$,}\\
      s_i \cdot \calT  & \text{otherwise}
    \end{cases}
\end{align*}
for $1 \leq i \leq n-1$ and $\calT \in \SIT(\alpha)$.
Here,  $s_i \cdot \calT$ is obtained from $\calT$ by swapping $i$ and $i+1$.
The resulting module is denoted by $\calV_\alpha$.

\begin{theorem}{\rm (\cite[Theorem 3.5]{15BBSSZ})}
For a composition $\alpha$ of $n$, $\calV_\alpha$ is an indecomposable $H_n(0)$-module 
whose image under quasisymmetric characteristic is  $\SG_\alpha^*$.
\end{theorem}

\subsubsection{Extended quasisymmetric Schur functions and related $0$-Hecke modules}
The extended Schur functions \(\calE_\alpha\) were introduced by Assaf and Searles in \cite{19AS} as stable limits of lock polynomials. 
They form a basis for \(\Qsym\), and their \(F\)-expansions can be expressed in terms of standard extended tableaux.

\begin{definition}(\cite[Definition 6.18]{19AS})
Let $\alpha$ be a composition of $n$.
A \emph{standard extended tableau} of shape $\alpha$ is a filling $\sfT$ of $\tcd(\alpha)$  with the entries $1, 2, \ldots, n$ such that the entries are all distinct, the entries in each row increase from left to right, and the entries in each column increase  from bottom to top.
\end{definition}
We denote the set of all standard extended tableaux of shape $\alpha$ by $\SET(\alpha)$.
It was shown in \cite[Theorem 6.19]{19AS} that 
\begin{align*}
    \calE_\alpha = \sum_{T \in \SET(\alpha)} F_{\comp(\Des_L(\sfrow(T)))^\rmc}.
\end{align*}
Subsequently, Searles \cite{19Searles} constructed an $H_n(0)$-module for the extended Schur functions $\calE_\alpha$.
Define an $H_n(0)$-action on the $\C$-span of $\SET(\alpha)$ by
\begin{align*}
    \pi_i \cdot T = \begin{cases}
      T  & \text{if $i$ is strictly left of $i+1$ in $T$,}\\
      0  & \text{if $i$ and $i+1$ are in the same column of $T$,}\\
      s_i \cdot T  & \text{if $i$ is strictly right of $i+1$ in $T$}
    \end{cases}
\end{align*}
for $1 \leq i \leq n-1$ and $T \in \SET(\alpha)$.
Here,  $s_i \cdot T$ is obtained from $T$ by swapping $i$ and $i+1$.
The resulting module is denoted by $X_\alpha$.

\begin{theorem}{\rm (\cite[Theorem 3.10, 3.13]{19Searles})}
For $\alpha \models n$, $X_\alpha$ is an indecomposable $H_n(0)$-module 
whose image under quasisymmetric characteristic is $\calE_\alpha$.
\end{theorem}

\subsection{Weak Bruhat interval modules}
\label{Weak Bruhat interval modules}
Finally, we introduce \(0\)-Hecke modules known as weak Bruhat interval modules, which were introduced by Jung, Kim, Lee, and Oh in \cite{22JKLO} to provide a unified approach for studying \(0\)-Hecke modules constructed using tableaux.

\begin{definition}{\rm (\cite[Definition 1]{22JKLO})}\label{def: WBI action}
For each left weak Bruhat interval $[\sigma, \rho]_L$ in $\SG_n$,
define $\sfB([\sigma, \rho]_L)$ (simply, $\sfB(\sigma, \rho)$) to be  the $H_n(0)$-module with  $\C[\sigma,\rho]_L$ as the underlying space and with the $H_n(0)$-action defined by
\begin{align*}
\pi_i \cdot \gamma := \begin{cases}
\gamma & \text{if $i \in \Des_L(\gamma)$}, \\
0 & \text{if $i \notin \Des_L(\gamma)$ and $s_i\gamma \notin [\sigma,\rho]_L$,} \\
s_i \gamma & \text{if $i \notin \Des_L(\gamma)$ and $s_i\gamma \in [\sigma,\rho]_L$}.
\end{cases} 
\end{align*}
\end{definition}
The $H_n(0)$-module $\sfB(\sigma, \rho)$ is called the \emph{weak Bruhat interval module associated to $[\sigma,\rho]_L$}.
Considering the above $H_n(0)$-action on $\sfB(\sigma, \rho)$, one can derive that 
\begin{align}\label{eq: character of WBI}
\ch([\sfB(\sigma, \rho)]) = \sum_{\gamma \in [\sigma, \rho]_L} F_{\comp(\Des_L(\gamma))^\rmc}.
\end{align}

The authors proved that \(\calV_\alpha\) and \(X_\alpha\) are equipped with weak Bruhat interval module structures. Let us provide a brief explanation of this result.

We first define two special standard immaculate tableaux of shape $\alpha$. 
Let \(\calT_\alpha\) be the filling of \(\tcd(\alpha)\) with entries \(1, 2, \ldots, n\) from left to right starting with the bottommost row. 
On the other hand, let \(\calT'_\alpha\) be the filling of \(\tcd(\alpha)\) defined by the following steps:
\begin{enumerate}[label = {\rm (\arabic*)}]
\item Fill the first column with entries \(1, 2, \ldots, \ell(\alpha)\) from bottom to top.
\item Fill the remaining cells with entries \(\ell(\alpha) + 1, \ell(\alpha) + 2, \ldots, n\) from left to right starting with the topmost row.
\end{enumerate}
In a similar manner, we define two standard extended tableaux of shape \(\alpha\).
Let 
$\sfT_\alpha$ be the standard extended tableau obtained by filling $\tcd(\alpha)$ with the entries  $1,2,\ldots, n$  from left to right starting with the bottommost row  and $\sfT'_\alpha$ be the standard extended tableau obtained by filling $\tcd(\alpha)$ with the entries $1,2,\ldots, n$   from bottom to top starting with the leftmost column.

\begin{theorem}{\rm (\cite[Theorem 4.4]{22JKLO})}
\label{interval module structures of V and X}
Let $\alpha$ be a composition of $n$.
\begin{enumerate}[label = {\rm (\arabic*)}]

\item The $\C$-linear isomorphism 
\[
\Theta_V: \calV_\alpha \ra \sfB(\sfrow(\calT_\alpha), \sfrow(\calT'_\alpha)), \; \calT \mapsto \sfrow(\calT) \quad (\calT \in \SIT(\alpha))
\]
is an $H_n(0)$-module isomorphism.

\item The $\C$-linear isomorphism 
\[
\Theta_X:X_\alpha \ra \sfB(\sfrow(\sfT_\alpha), \sfrow(\sfT'_\alpha)), \; \sfT \mapsto \sfrow(\sfT) \quad\;\,\, (\sfT \in \SET(\alpha))
\]
is an $H_n(0)$-module isomorphism.
\end{enumerate}
\end{theorem}

\section{An analogue of the Robinson-Schensted-Knuth Algorithm and its Greene’s Theorem}
\label{An analogue of RSK and its Greene Theorem}

In this section, we introduce an analogue of the Robinson-Schensted-Knuth algorithm and 
 examine how the shapes of the insertion tableaux are determined. In the case of the  Robinson-Schensted-Knuth algorithm, this issue was resolved by Greene \cite{74Gre}. 
The result we establish here will be essential for proving 
the main theorems of \cref{Distinguished filtrations of V and X}.

\subsection{An analogue of the Robinson-Schensted-Knuth Algorithm}\label{subsec: an analog of RSK}
Schensted's insertion is a procedure for inserting a positive integer into a semistandard Young tableau, which is the fundamental operation of the Robinson-Schensted-Knuth algorithm.
Using this algorithm, for each $\sigma \in \SG_n$,
we associate a pair of standard Young tableaux $(\Ptab(\sigma), \Qtab(\sigma))$ with the same shape.
It is well known that the map
\[
\SG_n \ra \bigsqcup_{\lambda \vdash n} \SYT(\lambda) \times \SYT(\lambda), \quad \sigma \mapsto (\Ptab(\sigma), \Qtab(\sigma))
\]
is bijective.

In 2006, Mason \cite{06Mason} introduced a procedure for inserting a nonnegative integer into a semi-skyline augmented filling. 
In 2011, Haglund, Luoto, Mason, and van Willigenburg \cite[pages 480-481]{11HLMW} rewrote this as a procedure for inserting a nonnegative integer into a composition tableau
through the weight-preserving bijection, provided in \cite[Lemma 4.2]{11HLMW}.

\medskip
\textbf{Convention.}
In this paper, a semi-skyline augmented filling is identified with the corresponding composition tableau through the weight-preserving bijection in \cite[Lemma 4.2]{11HLMW}. Therefore, we use composition tableaux instead of semi-skyline augmented fillings. 
With this identification, all the results from \cite{06Mason}, originally described in terms of semi-skyline augmented fillings, are interpreted as results involving composition tableaux. 
In particular, Mason's insertion procedure is viewed as the procedure in \cite[pages 480-481]{11HLMW}, and the algorithm in \cite[Procedure 4.2]{06Mason} is understood as one that produces a pair of composition tableaux from a two-line array.
\medskip

In 2018, Allen, Hallam, and Mason \cite[Procedure 3.1]{18AHM} introduces a related procedure for inserting a nonnegative integer into a Young composition tableau.
In this paper, we primarily focus on this insertion procedure.
\begin{procedure}\label{alg: insertion algorithm} (\cite[Procedure 3.1]{18AHM})
Let $T$ be a Young  composition tableau of shape $\alpha \models n$ and $k$ a positive integer.
Let \(\tcd(\alpha) = \{(c_1, d_1), (c_2, d_2), \ldots, (c_n, d_n)\}\), where each pair \((c_i, d_i)\) satisfies the following lexicographic order:
For indices \(1 \leq i < j \leq n\),
\[
d_i > d_j \quad \text{or} \quad d_i = d_j \text{ and } c_i > c_j.
\]
Set $i := 1$ and $k_0 := k$.
\begin{enumerate}
\item[\textbf{P1.}] If $i \leq n$, then go to \textbf{P2}.
Otherwise,  go to \textbf{P3}.

\item[\textbf{P2.}] 
We proceed by considering the following three cases.

\noindent
{\it Case 1:} $(c_i, d_i + 1) \notin \tcd(\alpha)$ and $T(c_i, d_i) \leq k_0$. 
Place $k_0$ in the cell $(c_i, d_i + 1)$ and then terminate the procedure.

\noindent
{\it Case 2:} $(c_i, d_i + 1) \in \tcd(\alpha)$ and $T(c_i, d_i) \leq k_0 < T(c_i, d_i + 1)$.  
    Set $k_1 := T(c_i, d_i + 1)$, and 
update $T(c_i, d_i + 1)$ to $k_0$ and $k_0$ to $k_1$.
Then,  update $i$ to $i+1$ and
go to \textbf{P1}.

\noindent
{\it Case 3:}  If neither of the above applies, 
update $i$ to $i+1$ and go to \textbf{P1}.

\item[\textbf{P3.}] Create a new row  with  a single cell containing $k_0$  and place the row so that  
the first column  strictly increases from bottom to top. 
Then, terminate the procedure.
\end{enumerate}
\end{procedure}
We denote the Young  composition tableau obtained from 
\cref{alg: modified mason algotithm} as $k \ra T$.
The set of cells that are modified during the insertion process, including the final cell  added is referred to as the {\it insertion sequence}. The final cell added to the tableau is called the {\it new cell}. 

\begin{example}
Let 
\[
T =
\begin{array}{c}
\begin{ytableau}
5 \\
2& 2& 3\\
1& 4
\end{ytableau}
\end{array}
\quad \text{and} \quad k = 2.
\]
Then, $k \ra T$ is given by 
\[
\begin{array}{c}
    \begin{ytableau}
        5\\ 
        4\\
        2 & 2 & 2\\
        1 & 3 
    \end{ytableau}
\end{array}
\]
with the insertion sequence $\{(2,3), (1,2), (3,1)\}$ and the new cell $(3, 1)$.
\end{example}
Let \(\tilcalW_n\) be the set of all two-line arrays 
\[
w = \begin{pmatrix}
    i_1 &  i_2 & \ldots & i_n \\
    j_1 &  j_2 & \ldots & j_n
\end{pmatrix}
\]
of positive integers such that for \(1 \leq r \leq n-1\),
\[
i_r < i_{r+1} \quad \text{or} \quad i_r = i_{r+1} \text{ and } j_r \leq j_{r+1}.
\]
Each word \( w = w_1 w_2 \dots w_n \) consisting of positive integers, especially a permutation, is identified with the two-line array
\[
\left(
\begin{array}{cccc}
    1 & 2 & \dots & n \\
    w_1 & w_2 & \dots & w_n
\end{array}
\right) \in \tilcalW_n \,.
\]
By  modifying  \cite[Procedure 4.2]{06Mason}, we present an algorithm that produces a pair of fillings of composition diagrams from a two-line array. 

\begin{algorithm}\label{alg: modified mason algotithm}
Let 
\[
w = \left(
\begin{array}{cccc}
    i_1 & i_2 & \ldots & i_n  \\
    j_1 & j_2 & \ldots & j_n
\end{array}
\right) \in \tilcalW_n \,.
\]
Let $\hat{P} = \emptyset$, $\hat{Q} = \emptyset$, and $k = 1$.
\begin{enumerate}
    \item[\textbf{P1.}]
    Say the new cell  created during $j_k \ra \hat{P}$ is located in the $c$th column of $j_k \ra \hat{P}$.
    Update $\hat{P}$ to $j_k \ra \hat{P}$.
    
   \item[\textbf{P2.}]
    If $c = 1$, then create a new row with a single cell containing $i_k$  and place the row so that the first column of $\hat{Q}$ increases from bottom to top. 
    Otherwise, denote the topmost row of $\hat{Q}$ with $c-1$ cells by the $r$th row of $\hat{Q}$ and update $\hat{Q}(r, c) = i_k$.
    \item[\textbf{P3.}]
    If $k < n$, update $k$ to $k+1$ and go back to the step \textbf{P1}.
    Otherwise, terminate the process.
\end{enumerate}
\end{algorithm}

We denote the pair $(\hat{P}, \hat{Q})$ of fillings obtained by applying \cref{alg: modified mason algotithm} to $w$ as $(\hatPtab(w), \hatQtab(w))$.
It is evident that $\hatPtab(w)$ is a Young composition tableau. Furthermore, from \cref{lem: relation b2w hatP and checkP}, it follows that 
$\hatQtab(w)$ is also a Young composition tableau.

\begin{example}
Let 
\[
w = \left(
\begin{array}{ccccc}
    1 & 2 & 3 & 3 & 4  \\
    3 & 1 & 2 & 5 & 4 
\end{array}
\right) \in \tilcalW_5.
\]
Then, we see that 
\begin{align*}
&
\left(
\begin{array}{c}
\scalebox{0.9}{
\begin{ytableau}
    3
\end{ytableau}}
\end{array}
, 
\begin{array}{c}
\scalebox{0.9}{
\begin{ytableau}
    1
\end{ytableau}}
\end{array}
\right)
\rightsquigarrow 
\left(
\begin{array}{c}
\scalebox{0.9}{
\begin{ytableau}
    3\\
    1
\end{ytableau}}
\end{array},
\begin{array}{c}
\scalebox{0.9}{
\begin{ytableau}
    2\\
    1 
\end{ytableau}}
\end{array}
\right)
\rightsquigarrow 
\left(
\begin{array}{c}
\scalebox{0.9}{
\begin{ytableau}
    3\\
    1 & 2
\end{ytableau}}
\end{array}, 
\begin{array}{c}
\scalebox{0.9}{
\begin{ytableau}
    2 & 3\\
    1
\end{ytableau}}
\end{array}
\right) 
\rightsquigarrow
\left(
\begin{array}{c}
\scalebox{0.9}{
\begin{ytableau}
    3\\
    1 & 2 & 5
\end{ytableau}}
\end{array}, 
\begin{array}{c}
\scalebox{0.9}{
\begin{ytableau}
    2 & 3 & 3\\
    1 
\end{ytableau}}
\end{array}
\right)\\
&\rightsquigarrow 
\left(
\begin{array}{c}
\scalebox{0.9}{
\begin{ytableau}
    3 & 5\\
    1 & 2 & 4
\end{ytableau}}
\end{array},
\begin{array}{c}
\scalebox{0.9}{
\begin{ytableau}
    2 & 3 & 3\\
    1 & 4
\end{ytableau}}
\end{array}
\right) = (\hatPtab(w), \hatQtab(w)).
\end{align*}
\end{example}

We will derive important properties of \cref{alg: modified mason algotithm} from those of \cite[Procedure 4.2]{06Mason}. 
As a first step, we examine the relationship between \cref{alg: insertion algorithm} and Mason's insertion procedure.
For a finite subset \( A = \{a_1 < a_2 < \dots < a_k\} \subseteq \Z_{>0} \),
consider the bijection
\begin{equation*}
    \sfrev_A: A \to A, \quad a_i \mapsto a_{k+1-i}.
\end{equation*}
For any filling $T$ of a composition diagram, 
we denote by
$\sfent(T)$   the set of all entries in $T$.
With these notations, we have a bijection
\[
\sfrevmap: \{\text{composition tableaux}\} \to \{\text{Young composition tableaux}\}, \quad T \mapsto \sfrev_{\sfent(T)} \circ T.
\]
For each composition $\alpha$, this map induces a bijection
\begin{equation*}
\sfrevmap|_{\SCT(\alpha)}: \SCT(\alpha) \to  \SYCT(\alpha),
\end{equation*}
which satisfies, for any \( T \in \SCT(\alpha) \), the relation
\begin{align*}
    \Des_{\scrS}(T) = 
    \{|\alpha| - i \mid i \in \Des_{\hscrS}(\sfrevmap(T)) \}.
\end{align*}
For clarity, we denote the composition tableau obtained by inserting $k$ into the  composition tableau $T$ using 
\cite[pages 480-481]{11HLMW} as \( k \CTra T \).
The following lemma is easily derived.

\begin{lemma}{\rm(cf. \cite[page 78]{18AHM})}\label{prop: relation b2w YCT and CT}
Let $T$ be a Young composition tableau and $k$ a positive integer.
Then, the following relations hold.
\begin{enumerate}[label = {\rm (\arabic*)}]
    \item $k \ra T = \sfrevmap(\sfrev_{\sfent(k \ra T)} (k) \CTra \sfrev_{\sfent(k \ra T)}  \circ T)$.  
    \item The new cells created by inserting $k$ into $T$ using \cref{alg: insertion algorithm}  and by inserting 
    $\sfrev_{\sfent(k \ra T)}(k)$ into $\sfrev_{\sfent(k \ra T)} \circ T$ using the procedure described in \cite[pages 480-481]{11HLMW} are same.
\end{enumerate}
\end{lemma}

Now, given a two-line array \( w \in \tilcalW_n \), we denote by \( (F(w), G(w)) \) the pair of composition tableaux obtained by applying \cite[Procedure 4.2]{06Mason} to \( w \).
For \( w = \left(
\begin{array}{cccc}
    i_1 & i_2 & \dots & i_n \\
    j_1 & j_2 & \dots & j_n
\end{array}
\right) \in \tilcalW_n \),
define 
\begin{align*}
 w[s:t] & := \left(
\begin{array}{cccc}
    i_s & i_{s+1} & \dots & i_t \\
    j_s & j_{s+1} & \dots & j_t
\end{array}
\right) \quad \text{for } 1 \leq s \leq t \leq n, \\
 w^{\rmconj} & := 
\left(
\begin{array}{cccc}
    \sfrev_X(i_n) & \sfrev_X(i_{n-1}) & \dots & \sfrev_X(i_1) \\
    \sfrev_Y(j_n) & \sfrev_Y(j_{n-1}) & \dots & \sfrev_Y(j_1)
\end{array}
\right),
\end{align*}
where \( X = \{i_t \mid 1 \leq t \leq n\} \) and \( Y = \{j_t \mid 1 \leq t \leq n\} \).
The second notation is introduced because, 
if \( w \in \SG_n \), then \( w^{\rmconj} =  w^{w_0} \).
It is straightforward to see that \( w[s:t] \in \tilcalW_{t-s+1} \), \( w^{\rmconj} \in \tilcalW_n \) and the map
\[
\tilcalW_n \to \tilcalW_n, \quad w \mapsto w^\rmconj
\]
is a bijection.
Using these notations, we can derive the following lemma from \cref{prop: relation b2w YCT and CT}.

\begin{lemma}\label{lem: relation b2w hatP and checkP}
Let $w = \left(
\begin{array}{cccc}
    i_1 & i_2  & \ldots &i_n  \\
    j_1 & j_2  & \ldots &j_n
\end{array}
\right) \in \tilcalW_n$. Then, we have
\[
\hatPtab(w) = \sfrevmap(F(w^\rmconj)) \quad \text{and} \quad \hatQtab(w) = \sfrevmap(G(w^\rmconj)).
\]
\end{lemma}

\begin{proof}
We first prove that $\hatPtab(w) = \sfrevmap(F(w^\rmconj))$.
We achieve our purpose by applying mathematical induction on $n$.
Since it is trivial for $n = 1$, we may assume that $n > 1$.
Let $w' := w[1:n-1]$ and let
\begin{align*}
Y := \{j_t \mid 1 \leq t \leq n \}, \;\, 
Y' := \{j_t \mid 1 \leq t \leq n - 1\}, \;\; \text{and} \;\;
 Y'_1 := \sfrev_Y(Y').
\end{align*}
Given any two finite subsets \( A, B \subseteq \Z_{>0} \) with \( |A| = |B| \), let \( \iota_{A,B}: A \to B \) denote the order-preserving bijection such that for any \( a_1, a_2 \in A \),
\[
a_1 < a_2 \quad \text{if and only if} \quad
\iota_{A, B}(a_1) < \iota_{A, B}(a_2).
\]
Then one sees that $F(w'^{\;\rmconj}) = \iota_{Y'_1, Y'} \circ F(w^\rmconj[2:n])$.
Combining this equality with the induction hypothesis, we derive that 
\begin{align}\label{eq: hatPtab and checkPtab}
\hatPtab(w') = \sfrevmap(F(w'^{\;\rmconj}))  = \sfrev_{Y'} \circ \iota_{Y'_1, Y'} \circ F(w^\rmconj[2:n]) = \sfrev_{Y} \circ F(w^\rmconj[2:n]). 
\end{align}
Here, the last equality follows from 
\[
\sfrev_{Y'} \circ \iota_{Y'_1, Y'} (y_k) = \sfrev_{Y'}(\sfrev_{Y}(y_{m+1 -k})) = \sfrev_{Y}(y_{k}),
\]
when we write $Y'_1 = \{y_1 < y_2 < \cdots < y_m\}$.
Therefore, we have 
\begin{align*}
\hatPtab(w)
&= j_n  \ra \hatPtab(w')\\
&= \sfrevmap(\sfrev_Y(j_n) \CTra \sfrev_Y \circ \hatPtab(w')) \quad  \text{(by \cref{prop: relation b2w YCT and CT}(1))}\\
&= \sfrevmap(\sfrev_Y(j_n)  \CTra F(w^\rmconj[2:n])) \quad \text{(by \cref{eq: hatPtab and checkPtab})}\\
&= \sfrevmap(F(w^\rmconj)).
\end{align*}

Next, we prove that  $\hatQtab(w) = \sfrevmap(G(w^\rmconj))$ by applying mathematical induction on $n$. Since it is trivial for $n = 1$, we may assume that $n > 1$.
Let
\begin{align*}
X := \{i_t \mid 1 \leq t \leq n \}, \;\,
X' := \{i_t \mid 1 \leq t \leq n-1\}, \;\; \text{and} \;\; X'_1 := \sfrev_X(X').
\end{align*}
One sees that $G(w'^{\;\rmconj}) =  \iota_{X'_1, X'} \circ G(w^\rmconj[2:n])$.
Combining this equality with the induction hypothesis, we derive that 
\begin{align}\label{eq: induction on check and hat Qtab}
\hatQtab(w') = \sfrevmap(G(w'^{\;\rmconj})) = \iota_{X'_1, X'} \circ G(w^\rmconj[2:n]).
\end{align}
This implies that $\hatQtab(w')$ and $G(w^{\rmconj}[2:n])$ have the same shape, denoted  by $\alpha$. 
And, by \cref{prop: relation b2w YCT and CT}(2), the new cells created during $j_n \ra \hatPtab(w')$ and $\sfrev_Y(j_n) \ra F(w^\rmconj[2:n])$ are same. Denote this cell by $(r, c)$.
Therefore, it follows that
\begin{itemize}
    \item $\hatQtab(w)$ is obtained from $\hatQtab(w')$ by filling $(k, c)$ with $i_n$, and
    \item $G(w^{\rmconj})$ is obtained from $G(w^\rmconj [2: n])$ by filling  $(k, c)$ with $\sfrev_X(i_n)$,
\end{itemize}
where $k = \max\{1 \leq t \leq \ell(\alpha) + 1 \mid \alpha_t = c-1 \}$.
Now, the desired result can be derived by combining these properties with
\cref{eq: induction on check and hat Qtab}.
\end{proof}

Consider the correspondence 
\begin{equation}\label{eq: pseudo RSK corr} 
\SG_n \ra \bigsqcup_{\substack{\alpha, \beta \models n \\ \lambda(\alpha) = \lambda(\beta)}} \SYCT(\alpha) \times \SYCT(\beta), \quad w \mapsto (\hatPtab(w), \hatQtab(w)). 
\end{equation}
By combining the properties in \cite[Section 4]{06Mason} with \cref{lem: relation b2w hatP and checkP}, we derive the following properties:
\begin{itemize}
    \item 
    The correspondence given in \cref{eq: pseudo RSK corr} is a bijection.
    \item For $\sigma \in \SG_n$,
    \begin{align}\label{eq: due to Mason prop 1}
    \hatPtab(\sigma^{-1}) = \hatQtab(\sigma) \quad \text{and} \quad \hatQtab(\sigma^{-1}) = \hatPtab(\sigma).
    \end{align}
    \item For $\sigma, \rho \in \SG_n$, 
     \begin{align}\label{eq: due to Mason prop 2}
   \Ptab(\sigma) = \Ptab(\rho) \quad \text{if and only if} \quad \hatPtab(\sigma) = \hatPtab(\rho).
    \end{align} 
\end{itemize}
Another important property is that the mapping $\sigma \mapsto \hatPtab(\sigma)$ preserves descents, as shown in the following proposition:

\begin{proposition}\label{prop: left descent and Young descent}
For $\sigma \in \SG_n$, we have
\[
\Des_L(\sigma) = \Des_{\hscrS}(\hatPtab(\sigma)).
\]
For the definition of $\Des_{\hscrS}(\hatPtab(\sigma))$, see \cref{eq: descent of Young comoposition tableau}.
\end{proposition}

\begin{proof}
We first prove our assertion in the case where $\sigma$ is equal to $\sfrow(T) w_0$ for some $T \in \SYT_n$. 
Then, it holds that   $\Ptab(\sfrow(T) w_0) = T$.
Moreover, since $\sfrow(T) w_0$ is the immaculate reading word of $T$ defined in \cite[Definition 2.7]{18AHM},  it follows from \cite[Lemma 3.11, Lemma 3.12]{18AHM} that 
$\Des_L(\sfrow(T) w_0) = \Des_{\hscrS}(\hatPtab(\sfrow(T) w_0))$.

Now, consider the case where $\sigma$ is an arbitrary permutation in $\SG_n$. 
Let $T = \Ptab(\sigma)$.
Since $\Ptab(\sfrow(T)w_0) = \Ptab(\sigma)$, it follows that $\Des_L(\sigma) = \Des_L(\sfrow(T)w_0)$ and $\hatPtab(\sfrow(T)w_0) = \hatPtab(\sigma)$ by \cref{eq: due to Mason prop 2}.
Therefore, the desired result is obtained from  the above discussion.
\end{proof}

However, unlike the Robinson-Schensted-Knuth algorithm, 
the shapes of $\hatPtab(\sigma)$ and $\hatQtab(\sigma)$
are not necessarily the same. In fact, it holds 
that 
$$\lambda(\sh(\hatPtab(\sigma)))=\lambda(\sh(\hatQtab(\sigma))) =\sh(\Ptab(\sigma))$$
(for further details, see 
\cite[Section 3.3, Corollary 4.4, Corollary 4.5]{06Mason}).

\subsection{The corresponding Greene's theorem}\label{subsec: the corresponding Greene's theorem}

We begin by introducing Greene's theorem, which provides a characterization of how the shapes of insertion tableaux in the Robinson-Schensted-Knuth algorithm are determined.

Let $\sigma \in \SG_n$ be fixed throughout this subsection. 
In 1961, Schensted \cite{61Sch} discovered that the size of the first  row  of the insertion tableau $\Ptab(\sigma)$ is equal to the length of a longest  increasing  subsequence of $\sigma$.
In 1974, Greene \cite{74Gre} extended this result by fully characterizing the shape of $\Ptab(\sigma)$ in terms of  \emph{$k$-increasing subsequences}.

For each positive integer \(k\), we introduce the notion of \(k\)-increasing subsequences. Consider a sequence \(w = w_1 w_2 \ldots w_n\) of positive integers. 
By abuse of notation, we use \(m \in w\) to denote that \(m\) is an entry of \(w\) and \(m \notin w\) otherwise.
Let $\ell(w) := n$ and refer to it as the \emph{length} of $w$.
By convention, we denote $\emptyset$ as the unique sequence of length $0$.
We say that $w$ is \emph{increasing} if $\ell(w) = 0$, or $\ell(w) > 1$ and 
$w_1 < w_2 < \cdots < w_n$.
For $1 \leq k \leq n$, 
a subsequence $w'$ of $w$ is called \emph{$k$-increasing} (with a hyphen) if $w'$ contains no decreasing subsequence of length $k+1$,
Note that $w'$ is $k$-increasing if and only if it can be expressed as a disjoint union of $k$ increasing subsequences of $w$.
Let $\fraki_k(w)$ denote the length of a longest $k$-increasing subsequence of $w$.

\begin{example}
Let $w = 9 4 3 1 5 6 8 2 7$.
The subsequence $ 4156827$ of $w$ is $2$-increasing since it is a union of $2$ increasing subsequences  $4568$ and $127$ of $w$.
Also, one can easily check that it is a longest $2$-increasing subsequence of $w$ and thus $\fraki_2(w) = 7$.
\end{example}

From now on, we write each permutation 
in one-line notation, viewing it as a sequence of positive integers. The following result is due to Greene.

\begin{theorem}{\rm (\cite{74Gre})}\label{thm: Greene k-dec}
Let $\sigma \in \SG_n$ and $\lambda = \sh(\sfP(\sigma))$.
For $1 \leq k \leq n$, we have 
\[
\fraki_k(\sigma) = \lambda_1 + \lambda_2 + \cdots + \lambda_k. 
\]
\end{theorem}

The purpose of this subsection is to present an analogue of
\cref{thm: Greene k-dec} for $\sh(\hatPtab(\sigma))$.
For a sequence \(w = w_1 w_2 \ldots w_n\) of distinct positive integers and  $1 \leq k \leq n$, 
let $\rmInc_k(w)$ be the set of all $k$-tuples $u = (u^{(i)})_{1 \leq i \leq k}$ of increasing subsequences of $w$ such that 
\begin{itemize}
    \item if $i \neq j$, then $u^{(i)}$ and $u^{(j)}$ have no common entries and 
   \item $\sum_{1 \leq i \leq k} \ell(u^{(i)}) = \fraki_k(w)$.
\end{itemize}
For \(u = (u^{(i)})_{1 \leq i \leq k} \in \rmInc_k(w)\), we denote by the boldfaced \(\mathbf{u}\) the subsequence of \(w\) that consists of all entries in the subsequence \(u^{(i)}\) for all $i$'s.
For example, if $w = 9 4 3 1 5 6 8 2 7 \in \SG_9$ and 
$u = (4568, 127) \in \rmInc_2(w)$,   
then $\mathbf{u} = 4156827$.
One sees that $\mathbf{u}$ is a longest $k$-increasing subsequence of $w$ and  $\{\mathbf{u} \mid u \in \rmInc_k(w)\}$ is  the set of all longest $k$-increasing subsequences of $w$.

We introduce a total order \(\le_{\rmset}\) on the set of all finite subsets of \(\Z_{>0}\) by defining \(A \le_{\rmset} B\)  if one of the following conditions holds:
\begin{itemize}
    \item \(|A| < |B|\), 
    \item \(A = B\), 
    \item \(|A| = |B|\) and there exists \(1 \leq i \leq |A|\) such that \(a_j = b_j\) for all \(1 \leq j < i\) and \(a_i <  b_i\), where \(A = \{a_1 < a_2 < \cdots < a_{|A|}\}\) and \(B = \{b_1 < b_2 < \cdots < b_{|B|}\}\).
\end{itemize}

\begin{definition}\label{def: initial entry set}
Let $\sigma \in \SG_n$ and  $1 \leq k \leq n$.

\begin{enumerate}
\item 
    For $u = (u^{(i)})_{1 \leq i \leq k}  \in \rmInc_k(\sigma)$, we define the {\it initial entries set of $u$} by 
\[
\sfIES(u) := \{ u^{(i)}_1  \mid 1 \leq i \leq k\; \text{with $\ell(u^{(i)}) > 0$}\}.
\]
\item 
With respect to the order $\leq_{\rmset}$, 
we define  
$$\sfmIES_k(\sigma):=\max\{\sfIES(u) \mid u \in \rmInc_k(\sigma)\}$$
\end{enumerate}
\end{definition}
By convention, we set $\sfmIES_0(\sigma):= \emptyset$.
From the definition, it follows that \( |\sfmIES_k(\sigma)| = k \).

\begin{example}
Let $\sigma = 637$. Then   
$\rmInc_1(\sigma) = \{(67), (37)\}$, $\rmInc_2(\sigma) = \{(67, 3), (37, 6)\}$, and 
$\rmInc_3(\sigma) = \{(67, 3, \emptyset), (37, 6, \emptyset), (7,6,3)\}$. For these, we have   
\begin{align*}
&\sfIES((67))=\{6\}, \quad \sfIES((37)) = \{3\},\\ 
&\sfIES(((67, 3))=\{3,6\}, \quad \sfIES((37, 6))=\{3,6\}, \text{ and }\\
&\sfIES((67,3, \emptyset))=\{3, 6\}, \quad \sfIES((37, 6,\emptyset))=\{3, 6\}, \quad \sfIES((3,6,7))=\{3,6,7\}.
\end{align*}
It follows that 
$\sfmIES_1(\sigma) = \{6\}$, $\sfmIES_2(\sigma) = \{3, 6\}$, and $\sfmIES_3(\sigma) = \{3, 6, 7\}$.
\end{example}

We begin by proving that \(\sfmIES_k(\cdot)\) is constant on the Knuth equivalence classes. Let us briefly recall the Knuth equivalence relation.
For \(\sigma \in \SG_n\) and \(1 < i < n\), we write \(\sigma \overset{1}{\cong} \sigma s_i\) if
\[
\sigma(i+1) < \sigma(i-1) < \sigma(i) \quad \text{or} \quad 
\sigma(i) < \sigma(i-1) < \sigma(i+1).
\]
We write \(\sigma \overset{2}{\cong} \sigma s_{i-1}\) if
\[
\sigma(i) < \sigma(i+1) < \sigma(i-1) \quad \text{or} \quad 
\sigma(i-1) < \sigma(i+1) < \sigma(i).
\]
The Knuth equivalence is an equivalence relation \(\Keq\) on \(\SG_n\) defined by \(\sigma \Keq \rho\) if and only if there exist \(\gamma_1, \gamma_2, \ldots, \gamma_k \in \SG_n\) such that
\[
\sigma = \gamma_1 \overset{a_1}{\cong} \gamma_2 \overset{a_2}{\cong} \cdots \overset{a_{k-1}}{\cong} \gamma_k = \rho,
\]
where \(a_1, a_2, \ldots, a_{k-1} \in \{1, 2\}\). 
Note that \(\sigma \Keq \rho\) if and only if \(\Ptab(\sigma) = \Ptab(\rho)\). For more information, see \cite{05BB, 97Ful, 91Sag, 99Stanley}.

\begin{lemma}\label{lem: invariant property on knuth equiv class}
Let $\sigma, \rho \in \SG_n$.
If  $\sigma$ and $\rho$ are Knuth equivalent, 
then 
$\sfmIES_k(\sigma) = \sfmIES_k(\rho)$  for all $1 \leq k \leq n$.
\end{lemma}

\begin{proof}
Assume that \(\sigma\) and \(\rho\) are Knuth equivalent and let \(1 \leq k \leq n\). 
To  prove 
the assertion, it suffices to consider the situation where \(\sigma \overset{a}{\cong} \rho\) with \(a \in \{1, 2\}\) and \(\ell(\sigma) > \ell(\rho)\). 
 The following cases are possible:

{\it Case 1}. $\sigma \overset{1}{\cong} \rho$, that is, 
\(\sigma = \ldots yzx \ldots\) and \(\rho = \ldots yxz \ldots\), where \(x < y < z\). 

{\it Case 2}. $\sigma \overset{2}{\cong} \rho$, that is, \(\sigma = \ldots zxy \ldots\) and \(\rho = \ldots xzy \ldots\), where \(x < y < z\).

\noindent
We  deal with  only {\it Case 1}, as {\it Case 2} can be handled in a similar manner. 
 
Since \(\rmInc_k(\sigma) \subseteq \rmInc_k(\rho)\), it immediately follows that \(\sfmIES_k(\sigma) \leq_{\rmset} \sfmIES_k(\rho)\).
Therefore, it remains to show that \(\sfmIES_k(\sigma) \geq_{\rmset} \sfmIES_k(\rho)\). 
Choose \( u = (u^{(i)})_{1 \leq i \leq k} \in \rmInc_k(\rho) \) such that \( \sfIES(u) = \sfmIES_k(\rho) \).
If there is no index \( 1 \leq t \leq k \) with \( x, z \in u^{(t)} \), then \( u \in \rmInc_k(\sigma) \), and thus \( \sfmIES_k(\sigma) \geq_{\rmset} \sfmIES_k(\rho) \). 
Otherwise,  let $1 \leq t \leq k$ be the index such that $x, z \in u^{(t)}$.
Write $u^{(t)} = A \; xz \; B$.
There are two subcases.
\begin{itemize}
    \item 
Suppose \(y \in u^{(s)}\) for some \(1 \leq s \leq k\). In this case, \(s \neq t\). Write \(u^{(s)} = C \, y \, D\) and define \(v = (v^{(i)})_{1 \leq i \leq k}\) as follows:
\[
v^{(i)} := \begin{cases}
 A \, x \, D   & \text{if } i = t, \\
 C \, yz \, B   & \text{if } i = s, \\
 u^{(i)}   & \text{otherwise.}
\end{cases}
\]
By \cref{thm: Greene k-dec}, one sees that \(v \in \rmInc_k(\sigma)\); in addition, \(\sfIES(v) = \sfIES(u) = \sfmIES_k(\rho)\). Therefore, \(\sfmIES_k(\sigma) \geq_{\rmset} \sfmIES_k(\rho)\), as required.

\item Suppose \(y \notin u^{(s)}\) for any \(1 \leq s \leq k\). 
Let \(v = (v^{(i)})_{1 \leq i \leq k}\) be defined as:
\[
v^{(i)} := \begin{cases}
 A \, yz \, B   & \text{if } i = t, \\
 u^{(i)}   & \text{otherwise.}
\end{cases}
\]
By \cref{thm: Greene k-dec}, one sees that \(v \in \rmInc_k(\sigma)\); in addition, \(\sfIES(v) \geq_{\rmset} \sfIES(u) = \sfmIES(\rho)\). 
Therefore, we obtain that \(\sfmIES_k(\sigma) \geq_{\rmset} \sfmIES_k(\rho)\).
\end{itemize}
\end{proof}

\begin{lemma}\label{lem: first col determine}
For $\sigma \in \SG_n$,
$\sfmIES_{l}(\sigma)$ is equal to the set of all entries in the first column of $\hatPtab(\sigma)$, 
 where $l$ denotes the length of  $\sh(\hatPtab(\sigma))$.
\end{lemma}
\begin{proof}
First, consider the case where \(\sigma = \sfrow(T)w_0\) for some \(T \in \SYT(\lambda)\), with \(\lambda\) being a partition of \(n\). Let \(l = \ell(\lambda)\). 
In this setting, \(\Ptab(\sigma) = T\)  and the set of all entries in the first column of \(\hatPtab(\sigma)\) is precisely the set of entries in the first column of \(T\).
Thus, to  prove the assertion, it suffices to verify that \(\sfmIES_l(\sigma) = \{T(1,1), T(2,1), \ldots, T(l,1)\}\). 
This follows from the observation that 
\[T(i, 1) \in \sfIES(u) \text{ for all \(u \in \rmInc_l(\sigma)\) and \(1 \leq i \leq l\)}. \]
Otherwise, there exist \(u \in \rmInc_l(\sigma)\) and \(1 \leq i \leq l\) such that \(T(i, 1) \notin \sfIES(u)\). Since \(T(j, k) > T(i, 1)\) for all \(i < j \leq l\) and \(1 \leq k \leq \lambda_j\), it follows that \(T(i, 1) \notin \mathbf{u}\). 
Consequently, \(\mathbf{u} \neq \sigma\), which contradicts the assumption that  \(u \in \rmInc_l(\sigma)\).

Let \(\sigma\) be an arbitrary permutation in \(\SG_n\) and let \(T := \Ptab(\sigma)\). 
Then, 
\[
\hatPtab(\sigma) = \hatPtab(\sfrow(T)w_0) \quad \text{and} \quad \sfmIES_l(\sigma) = \sfmIES_l(\sfrow(T)w_0),
\]  
where the first equality follows from \cref{eq: due to Mason prop 2}, while the second is a consequence of \cref{lem: invariant property on knuth equiv class}. 
Now, the desired assertion can be obtained by combining these equalities with the preceding discussion.
\end{proof}

For a partition $\lambda$, we denote 
 the skew partition whose Young diagram is obtained by rotating $\tyd(\lambda)$ by $180^\circ$ by $\lambda^\circ$.

\begin{lemma}\label{lem: restriction to be longest}
Let \(\lambda = (\lambda_1, \lambda_2, \ldots, \lambda_l)\) be a partition of \(n\) and let \(T \in \SYT(\lambda^\circ)\).  
For $1\le k\le l$ with \(\lambda_k = \lambda_1\), the following hold: 
\begin{enumerate}[label = {\rm (\arabic*)}]
    \item For \(u = (u^{(i)})_{1 \leq i \leq k} \in \rmInc_k(\sfcol(T))\), \(\sfIES(u)\) consists of the entries in the first row of \(T\).
\item   
$\sfmIES_k(\sfcol(T)) = \{T(1, j) \mid l-k+1 \leq j \leq l\}.$
\end{enumerate}
\end{lemma}
\begin{proof}
It is well known in the literature that the shape of \(\Ptab(\sfcol(T))\) is \(\lambda\) (for instance, see \cite{97Ful}).

(1) By \cref{thm: Greene k-dec}, an increasing subsequence of \(\sfcol(T)\) has length at most \(\lambda_1\) and the length of a longest \(k\)-increasing subsequence of \(\sfcol(T)\) is \(k\lambda_1\). Combining these properties, we deduce that the length of \(u^{(i)}\) is equal to \(\lambda_1\) for every \(1 \leq i \leq k\).
On the other hand, if an increasing subsequence of \(\sfcol(T)\) has length \(\lambda_1\), then it starts with an entry in the first row of \(T\). Therefore, we obtain the desired result.

(2)  Let \(X\) denote the set of all entries in the first row of \(T\).
By (1), we have 
\[
\{\sfIES(u) \mid u \in \rmInc(\sfcol(T))  \} \subseteq \{S \subseteq X \mid |S| = k\}.
\]
Combining this inclusion with 
$$\{T(1, j) \mid l-k+1 \leq j \leq l\} = \max\{S \subseteq X \mid |S| = k \} \text{ (with respect to $\geq_{\rmset}$)},$$
it suffices to show that there exists $u \in \rmInc_k(\sfcol(T))$ such that
\[
\sfIES(u) = \{T(1, j) \mid l-k+1 \leq j \leq l\}.
\]
The existence of  $u$ can be shown by considering a $k$-tuple $v = (v^{(i)})_{1 \leq i \leq k}$ of subsequences of $\sfcol(T)$, where  \(v^{(i)}\) is  the increasing subsequence of \(\sfcol(T)\) consisting of all entries in the \((l-k+i)\)th column of \(T\). 
\end{proof}

We are now ready to present the main result of this section.

\begin{theorem}\label{thm: an analoge of Greene}
Let \(\sigma \in \SG_n\) and let \(\lambda = (\lambda_1, \lambda_2, \dots, \lambda_l) = \sh(\Ptab(\sigma))\).
\begin{enumerate}[label = {\rm (\arabic*)}]
    \item 
    $\emptyset = \sfmIES_0(\sigma) \subsetneq \sfmIES_1(\sigma) \subsetneq \cdots \subsetneq
\sfmIES_{l}(\sigma)$.

    \item 
Let \(\sh(\hatPtab(\sigma)) = (\alpha_1, \alpha_2, \dots, \alpha_l)\) and let \(\sfmIES_l(\sigma) = \{x_1 < x_2 < \cdots < x_l\}\). Then, for each \(1 \leq k \leq l\), \(\alpha_k = \lambda_{i_k}\), where \(i_k\) is the smallest index \(1 \leq t \leq l\) such that \(x_k \in \sfmIES_{t}(\sigma)\).
\end{enumerate}
\end{theorem}

\begin{proof}
Before proving the assertions, we recall that there exists \(T \in \SYT(\lambda^\circ)\) such that \(\Ptab(\sfcol(T)) = \Ptab(\sigma)\). For \(1 \leq k \leq l\), let \(T^{(k)}\) denote the subfilling of \(T\) consisting of the first \(k\) columns from the right. Clearly, \(T^{(k)}\) is a Young tableau of shape \((\lambda_1, \lambda_2, \dots, \lambda_k)^\circ\).

(1) 
By \cref{lem: invariant property on knuth equiv class}, it suffices to prove the assertion for \(\sigma = \sfcol(T)\). To begin with, we will show that for every \( 1 \leq k \leq l \),
\begin{align}\label{eq: I_k on T and T[1:k]}
\sfmIES_k(\sfcol(T)) = \sfmIES_k(\sfcol(T^{(k)})).
\end{align}
Assume first that $\lambda_k = \lambda_1$. 
Now, applying \cref{lem: restriction to be longest}(2) to both \(\sfcol(T)\) and \(\sfcol(T')\) shows that 
$$\sfmIES_k(\sfcol(T)) = \{ T(1, j) \mid l-k+1 \leq j \leq l \} = \sfmIES_k(\sfcol(T^{(k)})).$$

From now on, assume that $\lambda_k < \lambda_1$. 
Since $\rmInc_k(\sfcol(T)) \supseteq \rmInc_k(\sfcol(T^{(k)}))$, it follows that 
$$\sfmIES_k(\sfcol(T)) \geq_{\rmset} \sfmIES_k(\sfcol(T^{(k)})).$$
Therefore, it suffices to verify that \(\sfmIES_k(\sfcol(T)) \leq_{\rmset} \sfmIES_k(\sfcol(T^{(k)}))\).
This proceeds in three steps as follows:

{\it Step 1}. Let $u = (u^{(i)})_{1 \leq i \leq k} \in \rmInc_k(\sfcol(T))$ be given.
We show that 
each $u^{(i)}$ contains at least one element from the ($m+1$)st row in $T$, where $m=\lambda_1 - \lambda_k$.
Let $X$ be the set of entries in the  first $m$ rows of $T$, that is,
\begin{align*}
X = \{T(i, j) 
\mid (i, j) \in \tyd(\lambda^\circ) \text{ with } 1 \leq i \leq m\}.
\end{align*}
Let $T'$ be the subfilling of $T$ obtained by removing all elements in $X$.
Consider the $k$-tuple ${\hat u} = ({\hat u}^{(i)})_{1 \leq i \leq k}$, where ${\hat u}^{(i)}$ is the sequence obtained from $u^{(i)}$ by removing all elements in $X$.
By its definition,
$\mathbf{\hat{u}}$ is a $k$-increasing subsequence of $\sfcol(T')$ and 
\begin{equation*}
\ell(\mathbf{\hat{u}}) \geq \ell(\mathbf{u}) - |X| = k \lambda_k = \fraki_k(\sfcol(T')).
\end{equation*}
Combining this inequality with 
\cref{thm: Greene k-dec} yields that $\ell(\mathbf{\hat{u}}) = \fraki_k(\sfcol(T'))$, that is, $\hat{u} \in \rmInc_k(\sfcol(T'))$.
Therefore,
by \cref{lem: restriction to be longest}(1), 
\begin{align*}
\sfIES(\hat{u}) \subseteq \{T(m+1, j) \mid (m+1, j) \in \tyd(\lambda^\circ) \}.
\end{align*}
As a consequence,  we have
\[
u^{(i)} = A^{(i)} \; T(m+1, c_i) \; B^{(i)} \quad (1 \leq i \leq k),
\]
where $T(m+1, c_i)$ is the initial entry of $\hat{u}^{(i)}$. 
By suitably rearranging the indices, we may assume that \( c_1 < c_2 < \cdots < c_k \).

{\it Step 2}. 
We show that 
\begin{align}\label{eq: prop of Ai}
    X = \{x \in A^{(i)} \mid 1 \leq i \leq k \}.
\end{align}
To prove this equality, it suffices to show that for each \( 1 \leq i \leq k \), \( A^{(i)} \) consists of elements that belong to \( X \), while \( B^{(i)} \) contains no elements from \( X \).
If the cell $C \in \tyd(\lambda^\circ)$ is strictly right and weakly above of $(m+1, c_i)$, then $T(C) > T(m+1, c_i)$, and thus $T(C) \notin A^{(i)}$.
This shows that $A^{(i)}$ consists only of elements in $X$.
On the other hand, if the cell $C \in \tyd(\lambda^\circ)$ is strictly left and strictly below of $(m+1, c_i)$, then $T(C) < T(m+1, c_i)$, and thus $T(C) \notin B^{(i)}$.
This shows that $B^{(i)}$ does not contain any element in $X$.

{\it Step 3}. 
Define a \( k \)-tuple \( v = (v^{(i)})_{1 \leq i \leq k} \) of sequences by
\[
v^{(i)} :=  A^{(i)} \,T(m+1, l-k+i)\, T(m+2, l-k+i)\, \dots \,T(\lambda_1, l-k+i).
\]
Since \(c_i \leq l-k+i\), we have  
\[
T(m+1, l-k+i) \geq T(m+1, c_i),
\]  
which implies that \(v^{(i)}\) is an increasing sequence.
And, $\mathbf{v}$ is equal to \( \sfcol(T^{(k)}) \) by \cref{eq: prop of Ai}. 
Consequently, \( v \in \rmInc_k(\sfcol(T^{(k)})) \). 
Moreover, it is easy to see that \( \sfIES(u) \leq_{\rmset} \sfIES(v) \).  
Since \( u \in \rmInc_k(\sfcol(T)) \) was chosen arbitrarily, we conclude that  
\[
\sfmIES_k(\sfcol(T)) \leq_{\rmset} \sfmIES_k(\sfcol(T^{(k)})).
\]

Now we are ready to prove the assertion.
For a given positive integer $2 \leq k \leq l$,
every entry in the first column of $\hatPtab(\sfcol(T^{(k-1 )}))$ is also in the first column of  $\hatPtab(\sfcol(T^{(k )}))$.
Combining this with \cref{lem: first col determine},  
we have
\begin{align*}
\sfmIES_{k-1}(\sfcol(T^{(k-1)})) \subseteq \sfmIES_k(\sfcol(T^{(k)})).
\end{align*}
Finally, by \cref{eq: I_k on T and T[1:k]}, we conclude that 
\[
\sfmIES_{k-1}(\sfcol(T)) \subseteq \sfmIES_k(\sfcol(T)).
\]

(2) 
By (1), the set of entries in the first column of $\hatPtab(\sfcol(T))$ is given by $\{y_1, y_2, \ldots, y_l\}$, where \(y_k\) denotes a unique element in \(\sfmIES_k(\sfcol(T)) \setminus \sfmIES_{k-1}(\sfcol(T)\)).
For the assertion, it suffices to show that the number of cells in the row starting with $y_k$ is $\lambda_k$ for $1 \leq k \leq l$.
To do this, let $\hat{P}_{k}:= \hatPtab(\sfcol(T^{(k)}))$ for all $1 \leq k \leq l$.
The shape of $\hat{P}_1$ is $(\lambda_1)$, and thus, the number of cells starting with $y_1$ in $\hat{P}_1$ is equal to $\lambda_1$.
Let us fix $2 \leq k \leq l$.
By \cref{eq: I_k on T and T[1:k]}, the set of entries in the first column of $\hat{P}_k$ is $\{y_1, y_2, \ldots, y_k \}$ and $\hat{P}_{k-1}$ is $\{y_1, y_2, \ldots, y_{k-1} \}$.
One can easily see that for all $1 \leq i \leq k-1$, 
the number of cells in the row starting with $y_i$ in $\hat{P}_{k}$  is greater than or equal to that of cells in the row starting with $y_i$ in $\hat{P}_{k-1}$.
Combining this with 
$$
\lambda(\sh(\hat{P}_{k-1})) = (\lambda_1,  \ldots, \lambda_{k-1}) \text{ and }\lambda(\sh(\hat{P}_{k})) = (\lambda_1, \lambda_2, \ldots, \lambda_k),
$$ 
we deduce that 
\begin{itemize}
    \item for $1 \leq i < k$, the number of cells in the row starting with $y_i$ in $\hat{P}_k$ is  equal to that of cells in the row starting with $y_i$ in $\hat{P}_{k-1}$, and 
    \item the number of cells in the row starting with $y_i$ in $\hat{P}_k$ is $\lambda_k$.
\end{itemize}
Consequently, for $1 \leq k \leq l$, the number of cells in the row starting with \(y_k\) in \(\hat{P}_l = \hatPtab(\sfcol(T))\) is equal to that of cells in the row starting with \(y_k\) in \(\hat{P}_k\).
\end{proof}

\begin{example}
Let $\sigma = 52783146 \in \SG_8$.
Then,
\[
\Ptab(\sigma) = 
\begin{array}{c}\begin{ytableau}
    5\\
    2 & 7 & 8\\
    1 & 3 & 4 & 6
\end{ytableau}
\end{array}\,\,.
\]
This shows that $\sh (\Ptab(\sigma))=(4,3,1)$.
Furthermore, 
$\sfmIES_1(\sigma) = \{2 \}$, $\sfmIES_2(\sigma) = \{ 2, 5\}$, \text{and} $\sfmIES_3(\sigma) = \{ 1,2,5\}$.
Under the notation in \cref{thm: an analoge of Greene}(2), this computation shows that \( i_1 = 3 \), \( i_2 = 1 \), and \( i_3 = 2 \).
By \cref{thm: an analoge of Greene}(2), we deduce that 
\(\sh(\hatPtab(\sigma)) = (1,4,3)\). Indeed, 
\[
\hatPtab(\sigma) = 
\begin{array}{c}
\begin{ytableau}
5 & 7 & 8\\
2 & 3& 4& 6\\
1
\end{ytableau}
\end{array}.
\]
\end{example}

\section{Distinguished filtrations of  $\calV_\alpha$ and $X_\alpha$}
\label{Distinguished filtrations of V and X}

We begin by presenting the background of the problem under consideration.
Let \( M \) be a finitely generated \( H_n(0) \)-module and let
\[
\ch([M]) = \sum_{\alpha \models n} c_\alpha F_\alpha \quad (c_\alpha \in \mathbb{Z}_{\ge 0}).
\] 
Given a  composition series 
\[
0 =: M_0 \subsetneq M_1 \subsetneq M_2 \subsetneq \cdots \subsetneq M_l := M
\]
of \( M \),
the following hold:
\begin{itemize}
    \item For each \( i \), \( \ch([M_i/M_{i-1}]) = F_\alpha \) for some \( \alpha \models n \) with \( c_\alpha > 0 \).
    \item For each \( \alpha \models n \), the number of composition factors \( M_i/M_{i-1} \) (with \( 1 \le i \le l \)) satisfying \( \ch([M_i/M_{i-1}]) = F_\alpha \) is equal to \( c_\alpha \).
\end{itemize}
Now, suppose that $\ch([M])$  expands positively in $\mathcal{B}$, 
where \(\mathcal{B} = \{\mathcal{B}_\alpha \mid \alpha \in \calI\}\) is a linearly independent subset of \(\Qsym_n\) such that \(\mathcal{B}_\alpha\) is \(F\)-positive for all \(\alpha \in \calI\), with \(\calI\) being an index set. 
One may naturally ask whether a similar phenomenon occurs in this context. 
To address this question, Kim, Lee, and Oh introduced the notion of distinguished filtrations of an \( H_n(0) \)-module. 

\begin{definition}{\rm (\cite[Definition 6.5]{23KLO})}\label{def: dist filt}
Let $\mathcal{B} = \{\mathcal{B}_\alpha \mid \alpha \in \calI\}$ be a linearly independent subset of $\Qsym_n$ 
with the property that  $\mathcal{B}_\alpha$ is $F$-positive for all $\alpha \in \calI$,
where $\calI$ is an index set. 
Given an  $H_n(0)$-module $M$, 
a \emph{distinguished filtration of $M$ with respect to $\mathcal{B}$} is an $H_n(0)$-submodule series  of $M$ 
\[
0 =: M_0 \subsetneq M_1 \subsetneq M_2 \subsetneq \cdots \subsetneq M_l := M
\]
such that for all $1 \leq k \leq l$, $\ch([M_k / M_{k-1}]) = \mathcal{B}_\alpha$ for some $\alpha \in \calI$.
\end{definition}

As seen in \cite[Example 6.6]{23KLO}, 
a distinguished filtration of $M$ with respect to $\calB$ may not exist even if $\ch([M])$ expands positively in $\calB$.
This is because the category  $\Hnmod$ is  neither semisimple nor representation-finite when $n > 3$ (\cite{11DY, 02DHT}).  

Let us review the results presented in~\cite[Section 6.2]{23KLO}. 
The authors address the problem in the case where \( M \) is the weak Bruhat interval module \( B(I) \) associated with a dual plactic-closed left weak Bruhat interval \( I\) and \(\mathcal{B}\) is the Schur basis \(\{s_{\lambda} \mid \lambda \vdash n\}\). 
It is well known that $\ch([\sfB(I)])$ is a skew Schur function.
Under this assumption, it was proved in~\cite[Theorem 6.7]{23KLO} that \(\sfB(I)\) admits a distinguished filtration with respect to the Schur basis. 
In the proof, the combinatorics related to the Robinson-Schensted-Knuth algorithm plays a crucial role.

Now, consider the basis \(\mathcal{S} := \{\mathscr{S}_{\alpha} \mid \alpha \models n\}\) of the quasisymmetric Schur functions and the basis \(\hcalS := \{\hscrS_{\alpha} \mid \alpha \models n\}\) of the Young quasisymmetric Schur functions, rather than the Schur basis.
It is known that every Schur function \(s_{\lambda}\)  expands positively in both $\calS$ and $\hcalS$.
Consequently, one might anticipate that the modules examined in \cite{23KLO} possess a distinguished filtration with respect to these bases. 
However, as seen in \cref{no filtration example}, this expectation does not hold in general.

In this paper, the focus is on the cases where \( M = \mathcal{V}_\alpha \) or \( X_\alpha \) and $\mathcal{B}=\hcalS$ for the following reasons:
\begin{itemize}
    \item According to \cite[Theorem 1.1]{18AHM}, \(\ch([\mathcal{V}_\alpha])\) expands positively in  $\hcalS$, even though it is not a symmetric function.
    \item According to \cite[Theorem 4.10]{24MN},  \(\ch([X_\alpha])\)  expands positively in  $\hcalS$ when \(\alpha\) is a composition of $n$ obtained by shuffling a partition and $(1^k)$ for some $k \geq 0$.
    \item These cases can be handled in a combinatorial manner  by virtue of  \cref{alg: modified mason algotithm}.
    \item By \cite[Theorem 3.5]{22CKNO}, \(X_\alpha\) is a quotient module of \(\calV_\alpha\), revealing an important relationship between these two structures.
\end{itemize}

\subsection{Distinguished filtrations of  $\calV_\alpha$}
\label{Distinguished filtrations of V}
We introduce a new relation on \(\SG_n\) derived from \cref{alg: modified mason algotithm}.

\begin{definition}
Define a relation $\simeq_M$ on $\SG_n$ by 
\[
\sigma \simeq_M \rho \quad \text{if  $\sh(\hatPtab(\sigma)) = \sh(\hatPtab(\rho))$ and $\hatQtab(\sigma) = \hatQtab(\rho)$.}
\]
\end{definition}
It is straightforward to verify that \(\simeq_M\) is an equivalence relation on \(\SG_n\). 
Moreover, it follows from \cref{eq: due to Mason prop 1} and \cref{eq: due to Mason prop 2} that
\[
\sigma \simeq_M \rho \quad \text{if and only if} \quad \sh(\hatPtab(\sigma)) = \sh(\hatPtab(\rho)) \text{ and } \Qtab(\sigma) = \Qtab(\rho).
\]
Recall that  the dual Knuth equivalence relation $\dKeq$ on $\SG_n$ is defined by
\[
 \sigma \dKeq \rho \quad\text{ if $\sigma^{-1} \Keq \rho^{-1}$, equivalently, $\Qtab(\sigma) = \Qtab(\rho)$.}
\]
This indicates that \(\simeq_M\) is a refinement of the dual Knuth equivalence relation \(\dKeq\). To elaborate further, let \(\sigma_0 \in \SG_n\) and let \(\lambda\) denote the shape of \(\Ptab(\sigma_0)\). The dual Knuth equivalence class \(C\) containing \(\sigma_0\) can be expressed as the union of equivalence classes under \(\simeq_M\) as follows:
\[
C = \bigsqcup_{
\substack{
\beta \models n \\\lambda(\beta) = \lambda}}
\{\sigma \in \SG_n \mid \text{$\sh(\hatPtab(\sigma)) = \beta$ and $\hatQtab(\sigma) = \hatQtab(\sigma_0)$}\}.
\]
\cref{prop: left descent and Young descent} implies that if a subset \(S\) of \(\SG_n\) is closed under \(\simeq_M\), then the quasisymmetric function
\[
\sum_{\sigma \in S} F_{\comp(\Des_L(\sigma))}
\]
 expands positively in $\hcalS$. 
In particular, when $S$ is an equivalence class under $\simeq_M$, this function is equal to $\hscrS_\beta$, where $\beta$ is the shape of $\hatPtab(\sigma)$ for any $\sigma \in S$.

Let $\alpha \models n$.
Recall that \( \calV_\alpha \cong \sfB(\sfrow(\calT_\alpha), \sfrow(\calT'_\alpha)) \) (see \cref{interval module structures of V and X}). 
We will  show that \( [\sfrow(\calT_\alpha), \sfrow(\calT'_{\alpha})]_L \cdot w_0 \) is closed under \( \simeq_M \). 
To do so, we first review the results of Allen, Hallam, and Mason \cite{18AHM}. 
Therein, they proved that  the dual immaculate quasisymmetric function \( \SG_\alpha^* \)  expands positively in \( \hcalS \). 
To be precise, they first map a word \( w = w_1 w_2 \ldots w_n \) to a pair \( (\hatPtab(w), \hat{Q}(w)) \) of fillings of the same shape.
Here, \( \hat{Q}(w) \) is obtained by filling the new cell created during the insertion procedure \( w_k \to \hatPtab(w[1:k-1]) \) with \( k \), for each \( 1 \leq k \leq n \). 
They then showed that for \( \sigma \in \SG_n \), \( \hat{Q}(\sigma) \) is a {\it dual immaculate recording tableau} (DIRT) with a {\it row strip shape} \( \alpha^\rmr \) if and only if \( \sigma \in [\sfrow(\calT_\alpha), \sfrow(\calT'_\alpha)]_L \cdot w_0 \) 
(for undefined terms such as dual immaculate recording tableaux and row strip shapes, see \cite[Definitions 3.9 and 3.4]{18AHM}).
Using this property, they established the bijection
\begin{align}\label{eq: 2018 AHM map}
\begin{split}
\underline{\rm AHM}: 
\SIT(\alpha) &\ra \bigsqcup_{\beta \models n} \SYCT(\beta) \times {\rm DIRT}(\beta, \alpha^\rmr)\\
\calT &\mapsto  (\hatPtab(\sfrow(\calT)w_0), \hat{Q}(\sfrow(\calT)w_0)).
\end{split}
\end{align}
Second, we observe that for $\sigma, \rho \in \SG_n$, 
\begin{align}\label{eq: relationship between Q' and hatQ}
    \text{if $\hat{Q}(\sigma) = \hat{Q}(\rho)$, then
$\hatQtab(\sigma) = \hatQtab(\rho)$.
}
\end{align}
Indeed, this can be derived by combining the following facts for a word \(w = w_1 w_2 \ldots w_n\):
\begin{itemize}
    \item \(\hat{Q}(w)\) 
records  the row and column indices  of the new cell created during the insertion procedure \(w_k \ra \hatPtab(w[1:k-1])\), for all \(1 \leq k \leq n\),
    \item \(\hatQtab(w)\) records only the column index of the new cell created during the insertion procedure \(w_k \ra \hatPtab(w[1:k-1])\), for all \(1 \leq k \leq n\).
\end{itemize}

The following lemma is derived by combining \cref{eq: 2018 AHM map} and \cref{eq: relationship between Q' and hatQ} with the $\C$-linear isomorphism in \cref{interval module structures of V and X}(1).
\begin{lemma}\label{prop: SIT and SET is closed under simeqM}
For a composition $\alpha$ of $n$,
$[\sfrow(\calT_\alpha), \sfrow(\calT'_{\alpha})]_L \cdot w_0$ is closed under $\simeq_M$.
\end{lemma}

\begin{remark} 
(1) \cref{prop: SIT and SET is closed under simeqM} does not imply that \( [\sfrow(\calT_\alpha), \sfrow(\calT'_{\alpha})]_L \) is closed under \( \simeq_M \). More generally, even if a subset \( S \subseteq \SG_n \) is closed under \( \simeq_M \), it does not necessarily follow that the sets \( S \cdot w_0 \), \( w_0 \cdot S \), and \( w_0 \cdot S \cdot w_0 \) are also closed under \( \simeq_M \). 
For instance, consider \( S = \{3124, 4123\} \subseteq \SG_4 \); in this case, none of $S \cdot w_0$, $w_0 \cdot S$, and $w_0 \cdot S \cdot w_0$ is closed under $\simeq_M$. 
On the other hand, if \( S \) is closed under the dual Knuth equivalence relation \( \dKeq \), then the sets \( S \cdot w_0 \), \( w_0 \cdot S \), and \( w_0 \cdot S \cdot w_0 \) remain closed under \( \dKeq \). 
This property is an intriguing distinction between \( \simeq_M \) and \( \dKeq \).

(2) For $\sigma, \rho \in \SG_n$, the equality \(\hatQtab(\sigma) = \hatQtab(\rho)\) does not necessarily imply that \(\hat{Q}(\sigma) = \hat{Q}(\rho)\).
For instance, let $\sigma = 621543$ and $\rho = 531426$.
Then
\[
\hatQtab(\sigma) = \hatQtab(\rho)=
\begin{array}{c}
\begin{ytableau}
    6 \\
    3 & 4\\
    2 & 5\\
    1 
\end{ytableau}
\end{array}, 
\text{
but } \,\,
\hat{Q}(\sigma) =
\begin{array}{c}
\begin{ytableau}
    3\\
    6\\
    2 & 4\\
    1 & 5
\end{ytableau}
\end{array} 
\ne 
\hat{Q}(\rho) =
\begin{array}{c}
\begin{ytableau}
    6\\
    3\\
    2 & 4\\
    1 & 5
\end{ytableau}
\end{array}.
\]
\end{remark}

For any \( \calT \in \SIT(\alpha) \), the set of entries in the first column of \( \hatPtab(\sfrow(\calT)w_0) \) is equal to the set of entries in the first column of \( \calT \). 
Combined with \cref{thm: an analoge of Greene}, this observation directly leads to the following lemma.

\begin{lemma}\label{lem: first col of SIT}
   Let $\alpha$ be a composition of $n$ and $\calT \in \SIT(\alpha)$.
   \begin{enumerate}[label = {\rm (\arabic*)}]
       \item The length of the shape of  $\hatPtab(\sfrow(\calT) w_0)$  is equal to $\ell(\alpha)$.
       \item  For \( 1 \leq k \leq \ell(\alpha) \),  
$$\sfmIES_{k}(\sfrow(\calT)w_0) \subseteq \{\calT(i, 1) \mid 1 \leq i \leq \ell(\alpha) \}.$$  
In particular, the equality holds when \( k = \ell(\alpha) \). 
\end{enumerate}
\end{lemma}

Let \(\leq_{\rm lex}\) denote the lexicographic order on the set of all compositions of \(n\).
The following lemma is essential in constructing our distinguished filtrations. 

\begin{lemma}\label{lem: shape and SIT word}
Let \(\alpha\) be a composition of \(n\) and let \(T, S \in \SIT(\alpha)\). If \(\sfrow(T) \dKeq \sfrow(S)\) and \(\sfrow(T) \preceq_L \sfrow(S)\), then 
\[
\sh(\hatPtab(\sfrow(T) w_0)) \geq_{\rm lex} \sh(\hatPtab(\sfrow(S) w_0)).
\]
\end{lemma}
\begin{proof}
We may assume that \(\sfrow(S) = s_i \sfrow(T)\) for some \(1 \leq i \leq n-1\). 
Then \(i\) is strictly below \(i+1\) in \(T\) and 
one of the following conditions holds:
\begin{itemize}
    \item Neither $i$ nor $i+1$ are not in the first column of $T$, or
    \item $i+1$ is in the first column of $T$, while $i$ is not.
\end{itemize} 
For simplicity, let \(\lambda := \sh(\Ptab(\sfrow(T)w_0))\) and \(l := \ell(\lambda)\).
Let \( 1 \leq k \leq l \). Note that \( i+1 \) is to the left of \( i \) in \(\sfrow(T) w_0\) when written in one-line notation.
From this, it follows that  
 $s_i \cdot u \in \rmInc_k(\sfrow(S)w_0)$ for any $u \in \rmInc_k(\sfrow(T)w_0)$.
Here, $s_i \cdot u$ denotes the $k$-tuple of sequences obtained from $u$ by swapping $i$ and $i+1$. 
Therefore, 
\begin{align}\label{eq: index inequality comp}
    s_i \cdot \sfmIES_k(\sfrow(T)w_0) \leq_{\rmset} \sfmIES_k(\sfrow(S)w_0).
\end{align}
By \cref{lem: first col of SIT}, there exist subsets $I_k, J_k \subseteq [l]$ with $|I_k| = |J_k| = k$ such that  
\begin{align*}
\sfmIES_k(\sfrow(T)w_0) = \{T(r, 1) \mid r \in I_k \} \quad \text{and} \quad \sfmIES_k(\sfrow(S)w_0) = \{S(r, 1) \mid r \in J_k \}.
\end{align*}
Combining these equalities with \cref{eq: index inequality comp} yields that
\begin{align}\label{eq: row index compare}
I_k \leq_{\rmset} J_k.
\end{align}

Let \(\alpha := \sh(\hatPtab(\sfrow(T)w_0))\) and \(\beta := \sh(\hatPtab(\sfrow(S)w_0))\). 
We claim that \(\alpha \geq_{\rm lex} \beta\). 
Suppose to the contrary that \(\alpha <_{\rm lex} \beta\). 
Then, there exists a unique index \(1 \leq t_0 \leq l\) such that \(\alpha_t = \beta_t\) for all \(1 \leq t < t_0\) and \(\alpha_{t_0} < \beta_{t_0}\). 
Thus, we can select an index \(1 \leq k_0 < l\) such that \(\lambda_{k_0} = \beta_{t_0}\) and \(\lambda_{k_0} > \lambda_{k_0 + 1}\).
Since $\lambda(\alpha) = \lambda(\beta) =\lambda$, by \cref{thm: an analoge of Greene},
we have \(I_{k_0} = \{1 \leq r \leq l \mid \alpha_r \geq \lambda_{k_0}\}\) and \(J_{k_0} = \{1 \leq r \leq l \mid \beta_r \geq \lambda_{k_0}\}\). 
Consequently,
\[
I_{k_0} = A \sqcup B \quad \text{and} \quad J_{k_0} = A' \sqcup \{t_0\} \sqcup B',
\]
for some \(A, A' \subseteq [1, t_0-1]\) and \(B, B' \subseteq [t_0 + 1, l]\).
Note that 
$A = A'$ since $\alpha_t = \beta_t$ for all $1 \leq t < t_0$. 
Furthermore, since $|I_{k_0}| = |J_{k_0}|$, it follows that $I_{k_0} >_{\rmset} J_{k_0}$. 
This contradicts the inequality in \cref{eq: row index compare}.
Therefore, our assertion follows.
\end{proof}

Recall that in the proof of \cite[Theorem 6.7]{23KLO}, the following order relation plays a key role: for \(\sigma, \rho \in \SG_n\) with \(\sigma \preceq_L \rho\), it holds that
\begin{align}\label{eq: left bruhat order and shape}
\Qtab(\sigma) = \Qtab(\rho) \quad \text{or} \quad \sh(\Qtab(\sigma)) \triangleright \sh(\Qtab(\rho)),
\end{align}
where \(\triangleright\) denotes the dominance order on the set of partitions of \(n\). This result can be derived by applying Ta\c{s}kin's result \cite[Proposition 3.2.5]{06Taskin} to the weak order\footnote{This order was originally defined in \cite[2.5.1]{04Mel}, where it is called the \emph{induced Duflo order}.} on \(\SYT_n\) given in \cite[Definition 3.1.3]{06Taskin}.
This relation also plays a crucial role in the proof of the subsequent theorem.

\begin{theorem}\label{thm: dist filt for Va}
For a composition $\alpha$ of $n$, $\calV_\alpha$ has a distinguished filtration with respect to $\{\hscrS_\beta \mid \beta \models n \}$.
\end{theorem}
\begin{proof}
By \cref{interval module structures of V and X}(1), the assertion is equivalent to stating that \( \sfB(\sfrow(\calT_\alpha), \sfrow(\calT'_\alpha)) \) has a distinguished filtration with respect to \( \{\hscrS_\beta \mid \beta \models n \} \).
For simplicity, set $I := [\sfrow(\calT_\alpha), \sfrow(\calT'_\alpha)]_L$.
Let $\calL := \{\hatQtab(\sigma w_0) \mid \sigma \in I\}$ and choose an arbitrary total order  $<_{\calL}$  on $\calL$ such that 
\[
T <_{\calL} S \quad \text{if $\lambda(\sh(T)) \triangleright \lambda(\sh(S))$}.
\]
For $Q \in \calL$, set
\[
A_Q := \{\sh(\hatPtab(\sigma w_0)) \mid  \text{$\sigma \in I$ and $\hatQtab(\sigma w_0) = Q$}\}.
\]
Let 
\[
R:= \bigsqcup_{Q \in \calL} \{Q\} \times  A_Q,
\]
and define a total order $\ll$ on $R$ by 
\[
(T, \beta) \ll (S, \gamma) \quad \text{if either $T <_{\calL} S$ or ($T = S$ and $\beta <_{\rm lex} \gamma$).}
\]
Enumerate the elements in \( R \) in increasing order with respect to the total order \( \ll \) so that
\begin{equation}\label{eq: def of R two parameters}
R = \{(Q_1, \gamma_1) \ll (Q_2, \gamma_2) \ll \cdots \ll (Q_l, \gamma_l)\}.
\end{equation}
Now, for each \(1 \leq k \leq l\), define
\begin{equation}\label{eq: def of set B(k)}
B_k := \{\sigma \in I \mid \hatQtab(\sigma w_0) = Q_i \;\text{and}\; \sh(\hatPtab(\sigma w_0)) = \gamma_i \;\text{for some} \; 1 \leq i \leq k\},
\end{equation}
with \(B_0 = \emptyset\). In particular, $B_l = I$.
We prove the assertion in two steps:

{\it Step 1}.  
For each \(1 \leq k \leq l\), we claim that \(M_k := \mathbb{C} B_k\) is an \(H_n(0)\)-submodule of \(\sfB(I)\). 
To  prove  this, it suffices to verify that  
\begin{align*}
\pi_i \cdot \sigma \subseteq B_k \cup \{0\} \quad \text{for all $\sigma \in B_k$ and $1 \leq i \leq n-1$.}
\end{align*}
Let $\sigma \in B_k$ and $1 \leq i \leq n-1$.
If $\pi_i \cdot \sigma$ is $\sigma$ or $0$, we are done.
Suppose that $\pi_i \cdot \sigma = s_i \sigma$, that is, $s_i \sigma \in I$ and $\ell(s_i\sigma) = \ell(\sigma) + 1$.
Let \( s \) (respectively, \( t \)) be the smallest positive integer $1 \leq k \leq l$ such that \( \sigma w_0 \in B_k \) (respectively, \( s_i \sigma w_0 \in B_k \)).
To prove the claim, it suffices to show that $s \geq t$.
\begin{itemize}
\item 
Suppose that the rearrangements of \(\gamma_s\) and \(\gamma_t\) are the same. This assumption is equivalent to saying that the shapes of \(\Ptab(\sigma w_0)\) and \(\Ptab(s_i \sigma w_0)\) are the same.
Note $\sigma w_0 \succeq_L s_i \sigma w_0$. 
Therefore, by \cref{eq: left bruhat order and shape}, we deduce that  $\Qtab(\sigma w_0) = \Qtab(s_i \sigma w_0)$, which is equivalent that $\hatQtab(\sigma w_0) = \hatQtab(s_i \sigma w_0)$.
Also, by \cref{lem: shape and SIT word}, we have  $\gamma_s \geq_{\rm lex} \gamma_t$.
Therefore, $s \geq t$.
\item 
Suppose that the rearrangements of $\gamma_s$ and $\gamma_t$ are different.
This assumption is equivalent to saying that the shapes of   $\Ptab(\sigma w_0)$ and $\Ptab(s_i \sigma w_0)$ are different.
Note that $\sigma w_0 \succeq_L s_i \sigma w_0$.
Therefore, by \cref{eq: left bruhat order and shape}, we deduce that 
\[
\lambda(\sh(Q_s)) = \sh(\Qtab(\sigma w_0)) \triangleleft \sh(\Qtab(s_i\sigma w_0)) = \lambda(\sh(Q_t)).
\]
As a result, we have \(Q_s >_{\calL} Q_t\), which implies that \(s > t\).
\end{itemize}

{\it Step 2}. 
For each \( 1 \leq k \leq l \), we claim that 
$\ch([M_k / M_{k-1}])= \hscrS_{\gamma_k}$.
Since $M_{k-1}$ and $M_{k}$ are submodules of \( B(I) \), it follows that  
\[
\ch([M_k/M_{k-1}]) = \ch([M_k]) - \ch([M_{k-1}]) = \sum_{\sigma \in B_k \setminus B_{k-1}} F_{\comp(\Des_L(\sigma))^\rmc}
\]  
(see \cref{eq: character of WBI}).
On the other hand, in view of \cref{prop: SIT and SET is closed under simeqM}, we see that the restriction
\begin{align*}
\underline{\rm AHM}\circ ({\Theta_V})^{-1}|_{B_k} \colon B_k \to \bigsqcup_{1 \leq i \leq k} \SYCT(\gamma_i) \times \{Q_i\}, \quad \sigma \mapsto (\hatPtab(\sigma w_0), \hatQtab(\sigma w_0))
\end{align*}
gives  a bijection and 
\begin{align}\label{eq: reestricted bijection two}  
\underline{\rm AHM}\circ ({\Theta_V})^{-1}(B_k \setminus B_{k-1}) = \SYCT(\gamma_k) \times \{Q_k\}
\end{align}
(for the definition of $\Theta_V$, see \cref{interval module structures of V and X}(1)).
Consequently, we have  
\begin{align*}  
\ch([M_k / M_{k-1}]) = \sum_{\sigma \in B_k \setminus B_{k-1}} F_{\comp(\Des_L(\sigma w_0))} = \hscrS_{\gamma_k},  
\end{align*}  
where the second is derived from \cref{eq: reestricted bijection two} and \cref{prop: left descent and Young descent}.

By {\it Step 1} and {\it Step 2}, we conclude that
\[
0 \subsetneq M_1 \subsetneq M_2 \subsetneq \cdots \subsetneq M_l = \sfB(I)
\]
forms a distinguished filtration with respect to  \( \{\hscrS_\beta \mid \beta \models n \} \).
\end{proof}

\begin{example}
Let $\alpha = (2,2,2)$.
Then, $\calV_\alpha \cong \sfB(I)$, where $I := [214365, 615243]_L$.
Note that $\{\hatQtab(\gamma w_0) \mid \gamma \in I \}$ is given by
\[
\left\{
Q_1 := \begin{array}{c} 
\begin{ytableau}
6\\
5\\
1 & 2& 3& 4
\end{ytableau}
\end{array}, \;
Q_2 := \begin{array}{c} 
\begin{ytableau}
5 & 6\\
3 \\
1 & 2& 4
\end{ytableau}
\end{array}, \;
Q_3 := \begin{array}{c} 
\begin{ytableau}
5\\
3 & 4 & 6\\
1 & 2
\end{ytableau}
\end{array}, \;
Q_4 := \begin{array}{c} 
\begin{ytableau}
5 & 6 \\
3 & 4 \\
1 & 2
\end{ytableau}
\end{array}
\right\}.
\]
Following the method presented in the proof of ~\cref{thm: dist filt for Va}, we will construct a distinguished filtration of $\sfB(I)$ with respect to $\{\hscrS_{\alpha} \mid \alpha \models  6\}$.
For $1 \leq i \leq 4$, let
\[
A_i := \{\sh(\hatPtab(\sigma w_0)) \mid \sigma \in I \text{ with } \hatQtab(\sigma w_0) = Q_i \}.
\]
We see that
\[
 A_1 = \{(1,1,4) \}, \;
 A_2 = \{(1,2,3), (2,1,3)\}, \;
 A_3 = \{(1,2,3), (1,3,2) \}, \;
A_4 = \{(2,2,2) \}.
\]
For $1 \leq i \leq 4$ and $\alpha \in A_i$, let
\[
B'_{i; \alpha} : = \{\sigma \in [214365, 615243]_L \mid \sh(\hatPtab(\sigma w_0)) = \alpha \text{ and } \hatQtab(\sigma w_0) = Q_i  \}
\]
and 
\[
B_{i;\alpha} := \left(\bigsqcup_{
\substack{1 \leq t < i \\ \beta \in A_t}} B'_{t; \beta}\right) 
\; \bigsqcup \; \left(\bigsqcup_{\substack{\beta \in A_i \\
\text{with $\beta \leq_{\rmlex} \alpha$}}} B'_{i; \beta}\right)\,\,.
\]
Then, 
\[
0 = \C B_{1;(1,1,4)}  \subsetneq \C B_{2;(3,2,1)} \subsetneq \C B_{2;(2,3,1)} \subsetneq \C B_{3;(3,2,1)} \subsetneq \C B_{3;(2,1,3)} \subsetneq \C B_{4;(2,2,2)} = \sfB(I)
\]
is a distinguished  filtration of $\sfB(I)$ with respect to $\{\hscrS_\beta \mid \beta \models 6 \}$.

\begin{figure}[t]
\[
\def \vp {1.3}
\def \hp {2.5}
\begin{tikzpicture}
\node[left] at (-0.3*\hp, -6*\vp) {\textcolor{red}{\small $B'_{1;(1,1,4)}$}};
\draw[red, dashed, thick] (-0.3*\hp, -5.7*\vp) -- (0.3*\hp, -5.7*\vp) -- (0.3*\hp, -6.3*\vp) -- (-0.3*\hp, -6.3*\vp) -- (-0.3*\hp, -5.7*\vp); 

\node[right] at (1.8*\hp, -5*\vp) {\textcolor{blue}{\small $B'_{3;(1,2,3)}$}};
\draw[blue, dashed, thick] (0.7*\hp, -5.3*\vp) -- (1.3*\hp, -5.3*\vp) -- (2.3*\hp, -4.3*\vp) -- (2.3*\hp, -3.7*\vp) -- (-0.3*\hp, -3.7*\vp) -- (-0.3*\hp, -4.3*\vp) -- (0.7*\hp, -5.3*\vp);

\node[left] at (-2.3*\hp, -4*\vp) {\textcolor{teal}{\small $B'_{2;(1,2,3)}$}};
\draw[teal, dashed, thick]
(-0.7*\hp, -5.3*\vp) -- 
(-0.7*\hp, -4.8*\vp) --
(-1.7*\hp, -4*\vp) --
(-1.7*\hp, -2.7*\vp) --
(-2.3*\hp, -2.7*\vp) --
(-2.3*\hp, -4.3*\vp)--
(-1.3*\hp, -5.3*\vp) -- (-0.7*\hp, -5.3*\vp);

\node[left] at (-2.3*\hp, -2*\vp) {\textcolor{brown}{\small $B'_{2;(2,1,3)}$}};
\draw[brown, dashed, thick]
(-1.7*\hp, -2.3*\vp) -- (-2.3*\hp, -2.3*\vp) -- (-2.3*\hp, -1.7*\vp) -- (-1.7*\hp, -1.7*\vp) -- (-1.7*\hp, -2.3*\vp);

\node[right] at (2.2*\hp, -1.5*\vp) {\textcolor{violet}{\small $B'_{3;(1,3,2)}$}};
\draw[violet, dashed, thick] (2.3*\hp, -3.3*\vp) -- (1.7*\hp, -3.3*\vp) -- (1.7*\hp, -1.7*\vp) -- (2.3*\hp, -1.7*\vp) -- (2.3*\hp, -3.3*\vp);

\node[left] at (-0.7*\hp, 0*\vp) {\textcolor{magenta}{\small $B'_{4;(2,2,2)}$}};
\draw[magenta, dashed, thick] 
(0.3*\hp, -3.3*\vp) -- (0.3*\hp, -2*\vp)-- (1.3*\hp, -1.3*\vp) -- (1.3*\hp, -0.7*\vp)--
(0.3*\hp, 0.3*\vp)-- (-0.3*\hp, 0.3*\vp) -- (-1.3*\hp, -0.7*\vp)-- (-1.3*\hp, -1.3*\vp) -- (-0.3*\hp, -2*\vp) -- (-0.3*\hp, -3.3*\vp) -- (0.3*\hp, -3.3*\vp);

\node[] at (0*\hp, 0*\vp) {$214365$};

\node[] at (-1*\hp, -1*\vp) {$215364$};
\node[] at (1*\hp, -1*\vp) {$314265$};

\node[] at (-2*\hp, -2*\vp) {$216354$};
\node[] at (0*\hp, -2*\vp) {$315264$};
\node[] at (2*\hp, -2*\vp) {$413265$};

\node[] at (-2*\hp, -3*\vp) {$316254$};
\node[] at (-0*\hp, -3*\vp) {$415263$};
\node[] at (2*\hp, -3*\vp) {$513264$};

\node[] at (-2*\hp, -4*\vp) {$416253$};
\node[] at (0*\hp, -4*\vp) {$514263$};
\node[] at (2*\hp, -4*\vp) {$613254$};

\node[] at (-1*\hp, -5*\vp) {$516243$};
\node[] at (1*\hp, -5*\vp) {$614253$};

\node[] at (0*\hp, -6*\vp) {$615243$};

\node at (0*\hp + 0.2*\hp, 0*\vp) {} edge [out=40,in=320, loop] ();
\node[right] at (0.4*\hp, 0*\vp) {\footnotesize $\pi_1, \pi_3, \pi_5$};

\node at (-1*\hp + 0.2*\hp, -1*\vp) {} edge [out=40,in=320, loop] ();
\node[right] at (-1*\hp + 0.4*\hp, -1*\vp) {\footnotesize $\pi_1, \pi_4$};
\node at (1*\hp + 0.2*\hp, -1*\vp) {} edge [out=40,in=320, loop] ();
\node[right] at (1*\hp + 0.4*\hp, -1*\vp) {\footnotesize $\pi_2, \pi_5$};

\node at (-2*\hp + 0.2*\hp, -2*\vp) {} edge [out=40,in=320, loop] ();
\node[right] at (-2*\hp + 0.4*\hp, -2*\vp) {\footnotesize $\pi_1, \pi_4, \pi_5$};
\node at (0*\hp + 0.2*\hp, -2*\vp) {} edge [out=40,in=320, loop] ();
\node[right] at (0*\hp + 0.4*\hp, -2*\vp) {\footnotesize $\pi_2, \pi_4$};
\node at (2*\hp + 0.2*\hp, -2*\vp) {} edge [out=40,in=320, loop] ();
\node[right] at (2*\hp + 0.4*\hp, -2*\vp) {\footnotesize $\pi_2, \pi_3, \pi_5$};

\node at (-2*\hp + 0.2*\hp, -3*\vp) {} edge [out=40,in=320, loop] ();
\node[right] at (-2*\hp + 0.4*\hp, -3*\vp) {\footnotesize $\pi_2, \pi_4, \pi_5$};
\node at (0*\hp + 0.2*\hp, -3*\vp) {} edge [out=40,in=320, loop] ();
\node[right] at (0*\hp + 0.4*\hp, -3*\vp) {\footnotesize $\pi_3$};
\node at (2*\hp + 0.2*\hp, -3*\vp) {} edge [out=40,in=320, loop] ();
\node[right] at (2*\hp + 0.4*\hp, -3*\vp) {\footnotesize $\pi_1, \pi_2$};

\node at (-2*\hp + 0.2*\hp, -4*\vp) {} edge [out=40,in=320, loop] ();
\node[right] at (-2*\hp + 0.4*\hp, -4*\vp) {\footnotesize $\pi_3, \pi_5$};
\node at (0*\hp + 0.2*\hp, -4*\vp) {} edge [out=40,in=320, loop] ();
\node[right] at (0*\hp + 0.4*\hp, -4*\vp) {\footnotesize $\pi_3, \pi_4$};
\node at (2*\hp + 0.2*\hp, -4*\vp) {} edge [out=40,in=320, loop] ();
\node[right] at (2*\hp + 0.4*\hp, -4*\vp) {\footnotesize $\pi_2, \pi_4, \pi_5$};

\node at (-1*\hp + 0.2*\hp, -5*\vp) {} edge [out=40,in=320, loop] ();
\node[right] at (-1*\hp + 0.4*\hp, -5*\vp) {\footnotesize $\pi_3, \pi_4$};
\node at (1*\hp + 0.2*\hp, -5*\vp) {} edge [out=40,in=320, loop] ();
\node[right] at (1*\hp + 0.4*\hp, -5*\vp) {\footnotesize $\pi_3, \pi_5$};

\node at (0*\hp + 0.2*\hp, -6*\vp) {} edge [out=40,in=320, loop] ();
\node[right] at (0*\hp + 0.4*\hp, -6*\vp) {\footnotesize $\pi_3, \pi_4, \pi_5$};

\draw[->] (-0.2*\hp, -0.3*\vp) -- (-0.8*\hp, -0.7*\vp);
\node at (-0.6*\hp, -0.4*\vp) {\footnotesize $\pi_4$};
\draw[->] (0.2*\hp, -0.3*\vp) -- (0.8*\hp, -0.7*\vp);
\node at (0.6*\hp, -0.4*\vp) {\footnotesize $\pi_2$};

\draw[->] (-1.2*\hp, -1.3*\vp) -- (-1.8*\hp, -1.7*\vp);
\node at (-1.6*\hp, -1.4*\vp) {\footnotesize $\pi_5$};
\draw[->] (-0.8*\hp, -1.3*\vp) -- (-0.2*\hp, -1.7*\vp);
\node at (-0.4*\hp, -1.4*\vp) {\footnotesize $\pi_2$};

\draw[->] (0.8*\hp, -1.3*\vp) -- (0.2*\hp, -1.7*\vp);
\node at (0.4*\hp, -1.4*\vp) {\footnotesize $\pi_4$};

\draw[->] (1.2*\hp, -1.25*\vp) -- (1.8*\hp, -1.7*\vp);
\node at (1.6*\hp, -1.4*\vp) {\footnotesize $\pi_3$};
\draw[->] (-2*\hp, -2.25*\vp) -- (-2*\hp, -2.7*\vp);
\node at (-2.1*\hp, -2.45*\vp) {\footnotesize $\pi_2$};

\draw[->] (-0.2*\hp, -2.25*\vp) -- (-1.8*\hp, -2.7*\vp);
\node at (-1.25*\hp, -2.45*\vp) {\footnotesize $\pi_5$};
\draw[->] (0*\hp, -2.25*\vp) -- (0*\hp, -2.7*\vp);
\node at (-0.1*\hp, -2.45*\vp) {\footnotesize $\pi_3$};

\draw[->] (2*\hp, -2.25*\vp) -- (2*\hp, -2.7*\vp);
\node at (2.1*\hp, -2.45*\vp) {\footnotesize $\pi_4$};

\draw[->] (-2*\hp, -3.25*\vp) -- (-2*\hp, -3.7*\vp);
\node at (-2.1*\hp, -3.45*\vp) {\footnotesize $\pi_3$};

\draw[->] (0*\hp, -3.25*\vp) -- (0*\hp, -3.7*\vp);
\node at (0.1*\hp, -3.45*\vp) {\footnotesize $\pi_4$};
\draw[->] (-0.2*\hp, -3.25*\vp) -- (-1.8*\hp, -3.7*\vp);
\node at (-1.25*\hp, -3.45*\vp) {\footnotesize $\pi_5$};

\draw[->] (2*\hp, -3.25*\vp) -- (2*\hp, -3.7*\vp);
\node at (2.1*\hp, -3.45*\vp) {\footnotesize $\pi_5$};
\draw[->] (1.8*\hp, -3.25*\vp) -- (0.2*\hp, -3.7*\vp);
\node at (0.8*\hp, -3.45*\vp) {\footnotesize $\pi_3$};

\draw[->] (-1.8*\hp, -4.25*\vp) -- (-1.2*\hp, -4.7*\vp);
\node at (-1.6 * \hp, -4.55*\vp) {\footnotesize $\pi_4$};

\draw[->] (0.2*\hp, -4.25*\vp) -- (0.8*\hp, -4.7*\vp);
\node at (0.4 * \hp, -4.55*\vp) {\footnotesize $\pi_5$};

\draw[->] (1.8*\hp, -4.25*\vp) -- (1.2*\hp, -4.7*\vp);
\node at (1.62 * \hp, -4.55*\vp) {\footnotesize $\pi_3$};

\draw[->] (-0.8*\hp, -5.25*\vp) -- (-0.2*\hp, -5.7*\vp);
\node at (-0.6 * \hp, -5.55*\vp) {\footnotesize $\pi_5$};

\draw[->] (0.8*\hp, -5.25*\vp) -- (0.2*\hp, -5.7*\vp);
\node at (0.6 * \hp, -5.55*\vp) {\footnotesize $\pi_4$};
\end{tikzpicture}
\]
\caption{The $H_6(0)$-action on the basis $[214365, 615243
]_L$ for $\sfB(214365, 615243) \cong \calV_{(2,2,2)}$ and the sets $B'_{i;\alpha}$'s}
\label{fig: figure for example}
\end{figure}
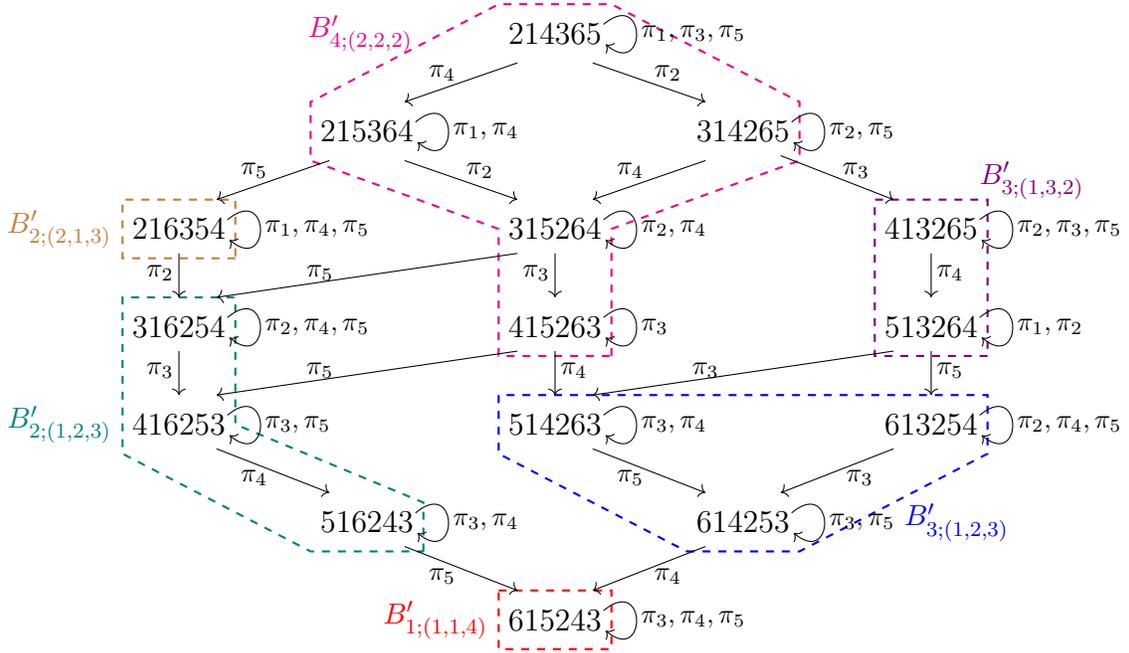
\end{example}

\subsection{Distinguished filtrations of  $X_\alpha$}
\label{Distinguished filtrations of X}

In this subsection, \(\alpha\) is assumed to be a composition of \(n\)
that is a shuffle of a partition and \((1^k)\) for some $k \geq 0$.
Under this assumption, Marcum and Niese \cite{24MN} showed that restricting the domain of the bijection \(\underline{\rm AHM}\) given in \cref{eq: 2018 AHM map} to \(\SET(\alpha)\) induces the bijection
\[
\underline{\rm AHM}|_{\SET(\alpha)}:  \SET(\alpha) \to \bigsqcup_{\beta \models n} \SYCT(\beta) \times {\rm DIRT}^*(\beta, \alpha^\rmr),
\]
where \({\rm DIRT}^*(\beta, \alpha^\rmr)\) is the set of DIRTs in \({\rm DIRT}(\beta, \alpha^\rmr)\) that have no {\it exceptions} (for the precise definition of \({\rm DIRT}^*(\beta, \alpha^\rmr)\), see \cite[Definition 4.7]{24MN}).
Using this, they \cite[Theorem 3.5]{24MN} proved that $\calE_\alpha$  expands positively in \(\hcalS\).

Combined with the $\C$-linear isomorphism in \cref{interval module structures of V and X}(2), the bijection $\underline{\rm AHM}|_{\SET(\alpha)}$ and \cref{eq: relationship between Q' and hatQ} yield the following lemma.

\begin{lemma}\label{prop: SET is closed under simeqM}
If \(\alpha\) is a composition of \(n\) obtained by shuffling a partition and \((1^k)\) for some \(k \geq 0\), then
$[\sfrow(\sfT_\alpha), \sfrow(\sfT'_\alpha)]_L \cdot w_0$ is closed under $\simeq_M$.
\end{lemma}

Moreover, since \(\SET(\alpha) \subseteq \SIT(\alpha)\), \cref{lem: first col of SIT} and \cref{lem: shape and SIT word} remain valid. 
Combining these lemmas with \cref{prop: SET is closed under simeqM}, we can derive the following theorem using the same approach as in the proof of \cref{thm: dist filt for Va}.
It should be noted that, in this case, \(R\) in \eqref{eq: def of R two parameters} can be described more simply. 
Specifically, it consists only of \(\gamma_i\)'s rather than \((Q_i, \gamma_i)\)'s.

\begin{theorem}\label{thm: dist filt for Xa}
If \(\alpha\) is a composition of \(n\) obtained by shuffling a partition and \((1^k)\) for some \(k \geq 0\), then
$X_\alpha$ has a distinguished filtration with respect to $\{\hscrS_\beta \mid \beta \models n \}$.
\end{theorem}
\begin{proof}
By \cref{interval module structures of V and X}(2), the assertion is equivalent to stating that \( \sfB(\sfrow(\sfT_\alpha), \sfrow(\sfT'_\alpha)) \) has a distinguished filtration with respect to \( \{\hscrS_\beta \mid \beta \models n \} \).
For simplicity, set $I := [\sfrow(\sfT_\alpha), \sfrow(\sfT'_\alpha)]_L$
and $\lambda := \lambda(\alpha)$.

First, we show that 
\begin{equation}\label{eq: same Qtab in Xal}
\hatQtab(\sigma w_0) = \sfT_{\alpha^\rmr} \,\text{ for all $\sigma \in I$.}
\end{equation}
For this purpose, we begin by observing that 
\begin{equation}\label{eq: shape of SET word}
    \sh(\Ptab(\sigma w_0)) =  \lambda \,\text{ for all $\sigma \in I$}.
\end{equation}
Let $\sigma \in I$. 
Then $\sigma = \sfrow(T)$ for some $T \in \SET(\alpha)$. 
One can see that for $1 \leq k \leq n$,  
\[
\fraki_k(\sigma w_0) = \lambda_1 + \cdots+ \lambda_k. \]
Indeed, the inequality $\ge$ follows from the fact that every row of $T$ increases from left to right, while the opposite inequality $\le$ follows from the fact that every increasing subsequence of $\sigma w_0$ contains at most one entry in each column of $T$.
Combining this equality with \cref{thm: Greene k-dec}, we see that $\sh(\Ptab(\sigma w_0)) = \lambda$, as required. 
Since $I \cdot w_0$ is a left weak Bruhat interval, by \cref{eq: shape of SET word} and \cref{eq: left bruhat order and shape}, we have 
\[
\Qtab(\sigma w_0) = \Qtab(\sfrow(T) w_0) \,\text{ for any $\sigma \in I$}.
\]
Now, due to \cref{eq: due to Mason prop 1} and \cref{eq: due to Mason prop 2}, \cref{eq: same Qtab in Xal} follows from 
$\hatQtab(\sfrow(\sfT_\alpha) w_0) = \sfT_{\alpha^\rmr}$.

Second, we define $B_k$ for each \(1 \leq k \leq l\).
Let $R = \{\sh(\hatPtab(\sigma w_0)) \mid \sigma \in I\}$.
Enumerate the elements in \( R \) in increasing order according to the lexicographic order \( \leq_{\rmlex} \) so that
\[
R = \{\gamma_1 <_{\rmlex}\gamma_2 <_{\rmlex} \cdots <_{\rmlex} \gamma_l  \}.
\]
Now, for each \(1 \leq k \leq l\), define
\begin{equation*}
B_k := \{\sigma \in I \mid  \sh(\hatPtab(\sigma w_0)) = \gamma_i \;\text{for some} \; 1 \leq i \leq k\},
\end{equation*}
with \(B_0 = \emptyset\). 
In particular, $B_l = I$.
In view of \cref{eq: same Qtab in Xal}, we see that  
\[
B_k = \{\sigma \in I \mid  \sh(\hatPtab(\sigma w_0)) = \gamma_i \text{ and } \hatQtab(\sigma w_0) = \sfT_{\alpha^\rmr} \;\text{for some} \; 1 \leq i \leq k\}.
\]

Finally, following the same manner as the proof of \cref{thm: dist filt for Va}, 
we can see that for \(1 \leq k \leq l\), \(M_k := \mathbb{C} B_k\) is an \(H_n(0)\)-submodule of \(\sfB(I)\) and that
\[
0 \subsetneq M_1 \subsetneq M_2 \subsetneq \cdots \subsetneq M_l = \sfB(I)
\]
forms a distinguished filtration with respect to  \( \{\hscrS_\beta \mid \beta \models n \} \). 
\end{proof}

\begin{example}
Let $\alpha= (3, 1, 2)$. 
Then, $I := [\sfrow(\sfT_\alpha), \sfrow(\sfT'_\alpha)]_L = [321465, 641253]_L$.
And, for every $\sigma \in I$, we have  
\[
\hatQtab(\sigma w_0)=\begin{array}{c}
\begin{ytableau}
4 & 5& 6\\
3\\
1 & 2
\end{ytableau}
\end{array}.
\]
Let
\[
R := \{\gamma_1 := (1,2,3) <_{\rmlex} \gamma_2 := (1, 3, 2)  <_{\rmlex} \gamma_3 := (2,1,3) <_{\rmlex} \gamma_4 := (3,1, 2) \}.
\]
Then, by \cref{thm: dist filt for Xa},
\[
0 \subsetneq \C B_{1}  \subsetneq \C B_{2} \subsetneq \C B_{3} \subsetneq \C B_{4} = \sfB(I)
\]
is the desired  distinguished  filtration of $\sfB(I)$ with respect to $\{\hscrS_\beta \mid \beta \models 6 \}$.

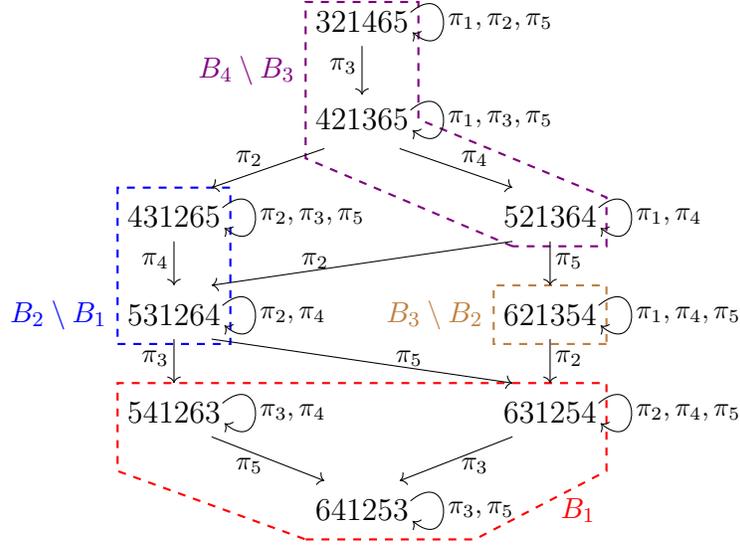
\begin{figure}[t]
\[
\def \vp {1.3}
\def \hp {2.5}
\begin{tikzpicture}

\node[right] at (0*\hp, -5*\vp) {\textcolor{red}{\small $B_1$}};
\draw[red, dashed, thick] (-1.3*\hp, -5.3*\vp) -- (-0.4*\hp, -5.3*\vp) -- ( 0.3*\hp,- 4.6*\vp) -- (0.3* \hp, -3.7 * \vp) -- (-2.3*\hp, -3.7*\vp) -- ( -2.3*\hp,- 4.6*\vp) -- (-1.3*\hp, -5.3*\vp);

\node[left] at (-2.3 * \hp, -3*\vp) {\textcolor{blue}{\small $B_2 \setminus B_1$}};
\draw[blue, dashed, thick] (-2.3*\hp, -3.3*\vp) -- (-1.7*\hp, -3.3*\vp) -- (-1.7*\hp, -1.7*\vp) -- (-2.3*\hp, -1.7*\vp) -- (-2.3*\hp, -3.3*\vp);

\node[left] at (-0.3 * \hp, -3*\vp) {\textcolor{brown}{\small $B_3 \setminus B_2$}};
\draw[brown, dashed, thick] (0.3*\hp, -3.3*\vp) -- (-0.3*\hp, -3.3*\vp) -- (-0.3*\hp, -2.7*\vp) -- (0.3*\hp, -2.7*\vp) -- (0.3*\hp, -3.3*\vp);

\node[left] at (-1.3
*\hp, -0.5*\vp) {\textcolor{violet}{\small $B_4 \setminus B_3$}};
\draw[violet, dashed, thick] (-0.2* \hp, -2.3 * \vp) -- (-1.3*\hp, -1.4*\vp) --(-1.3*\hp, 0.2*\vp) --(-0.7*\hp, 0.2*\vp) --(-0.7*\hp, -1*\vp)-- (0.3* \hp, -1.8 * \vp) -- (0.3* \hp, -2.3 * \vp) -- (-0.2* \hp, -2.3 * \vp);

\node[] at (-1*\hp, 0*\vp) {$321465$};

\node[] at (-1*\hp, -1*\vp) {$421365$};

\node[] at (-2*\hp, -2*\vp) {$431265$};
\node[] at (0*\hp, -2*\vp) {$521364$};

\node[] at (-2*\hp, -3*\vp) {$531264$};
\node[] at (-0*\hp, -3*\vp) {$621354$};

\node[] at (-2*\hp, -4*\vp) {$541263$};
\node[] at (0*\hp, -4*\vp) {$631254$};

\node[] at (-1*\hp, -5*\vp) {$641253$};

\node at (-1*\hp + 0.2*\hp, 0*\vp) {} edge [out=40,in=320, loop] ();
\node[right] at (0.4*\hp-1*\hp, 0*\vp) {\footnotesize $\pi_1, \pi_2, \pi_5$};

\node at (-1*\hp + 0.2*\hp, -1*\vp) {} edge [out=40,in=320, loop] ();
\node[right] at (-1*\hp + 0.4*\hp, -1*\vp) {\footnotesize $\pi_1, \pi_3, \pi_5$};

\node at (-2*\hp + 0.2*\hp, -2*\vp) {} edge [out=40,in=320, loop] ();
\node[right] at (-2*\hp + 0.4*\hp, -2*\vp) {\footnotesize $\pi_2, \pi_3, \pi_5$};
\node at (0*\hp + 0.2*\hp, -2*\vp) {} edge [out=40,in=320, loop] ();
\node[right] at (0*\hp + 0.4*\hp, -2*\vp) {\footnotesize $\pi_1, \pi_4$};

\node at (-2*\hp + 0.2*\hp, -3*\vp) {} edge [out=40,in=320, loop] ();
\node[right] at (-2*\hp + 0.4*\hp, -3*\vp) {\footnotesize $\pi_2, \pi_4$};
\node at (0*\hp + 0.2*\hp, -3*\vp) {} edge [out=40,in=320, loop] ();
\node[right] at (0*\hp + 0.4*\hp, -3*\vp) {\footnotesize $\pi_1, \pi_4, \pi_5$};

\node at (-2*\hp + 0.2*\hp, -4*\vp) {} edge [out=40,in=320, loop] ();
\node[right] at (-2*\hp + 0.4*\hp, -4*\vp) {\footnotesize $\pi_3, \pi_4$};
\node at (0*\hp + 0.2*\hp, -4*\vp) {} edge [out=40,in=320, loop] ();
\node[right] at (0*\hp + 0.4*\hp, -4*\vp) {\footnotesize $\pi_2, \pi_4, \pi_5$};

\node at (-1*\hp + 0.2*\hp, -5*\vp) {} edge [out=40,in=320, loop] ();
\node[right] at (-1*\hp + 0.4*\hp, -5*\vp) {\footnotesize $\pi_3, \pi_5$};

\draw[->] (-1*\hp, -0.25*\vp) -- (-1*\hp, -0.75*\vp);
\node at (-1.1*\hp, -0.45*\vp) {\footnotesize $\pi_3$};

\draw[->] (-1.2*\hp, -1.3*\vp) -- (-1.8*\hp, -1.7*\vp);
\node at (-1.6*\hp, -1.4*\vp) {\footnotesize $\pi_2$};
\draw[->] (-0.8*\hp, -1.3*\vp) -- (-0.2*\hp, -1.7*\vp);
\node at (-0.4*\hp, -1.4*\vp) {\footnotesize $\pi_4$};

\draw[->] (-2*\hp, -2.25*\vp) -- (-2*\hp, -2.7*\vp);
\node at (-2.1*\hp, -2.45*\vp) {\footnotesize $\pi_4$};

\draw[->] (-0.2*\hp, -2.25*\vp) -- (-1.8*\hp, -2.7*\vp);
\node at (-1.25*\hp, -2.45*\vp) {\footnotesize $\pi_2$};
\draw[->] (0*\hp, -2.25*\vp) -- (0*\hp, -2.7*\vp);
\node at (0.1*\hp, -2.45*\vp) {\footnotesize $\pi_5$};

\draw[->] (-2*\hp, -3.25*\vp) -- (-2*\hp, -3.7*\vp);
\node at (-2.1*\hp, -3.45*\vp) {\footnotesize $\pi_3$};

\draw[->] (-1.8*\hp, -3.25*\vp) -- (-0.2*\hp, -3.7*\vp);
\node at (-0.75*\hp, -3.45*\vp) {\footnotesize $\pi_5$};

\draw[->] (0*\hp, -3.25*\vp) -- (0*\hp, -3.7*\vp);
\node at (0.1*\hp, -3.45*\vp) {\footnotesize $\pi_2$};

\draw[->] (-1.8*\hp, -4.25*\vp) -- (-1.2*\hp, -4.7*\vp);
\node at (-1.6 * \hp, -4.55*\vp) {\footnotesize $\pi_5$};

\draw[->] (-0.2*\hp, -4.25*\vp) -- (-0.8*\hp, -4.7*\vp);
\node at (-0.4 * \hp, -4.55*\vp) {\footnotesize $\pi_3$};

\end{tikzpicture}
\]

\caption{The $H_6(0)$-action on the basis $[215436, 641254]_L$ for $\sfB(215436, 641254) \cong X_{(3,1,2)}$ and the sets $B_i \setminus B_{i-1}$'s}
\label{fig: figure for example X}
\end{figure}
\end{example}

\section{Indecomposable $0$-Hecke modules 
for Young quasisymmetric Schur functions and quasisymmetric Schur functions}
\label{Indecomposable modules for (Young) quasisymmetric Schur functions}

Given an important \(F\)-positive quasisymmetric function \(f\), constructing an indecomposable $H_n(0)$-module whose image under the quasisymmetric characteristic is \(f\) is a significant open problem. However, this problem remains unsolved for \(\scrS_{\alpha}\) and \(\hscrS_{\alpha}\). For example, \( \ch([\bfS_\alpha])=  \scrS_{\alpha}\), but $\bfS_\alpha$ is not, in general, indecomposable (see \cref{subsec: quasi Schur ftn}). 
The objective of this section is to construct and investigate an indecomposable \(H_n(0)\)-module \(\bfY_\alpha\) such that \(\ch([\bfY_\alpha])=\hscrS_\alpha\).
By taking the \(\upphi\)-twist of this module, we obtain an indecomposable \(H_n(0)\)-module whose image under the quasisymmetric characteristic is \(\scrS_\alpha\).

\subsection{Indecomposable $H_n(0)$-modules $\bfY_\alpha$ and $\upphi[\bfY_\alpha]$}\label{subsec: indecomp mod for YSal and Sal}
Let $\alpha \models n$  and let 
\begin{equation*}
0 = M_0 \subsetneq M_1 \subsetneq \cdots \subsetneq M_{l-1}  \subsetneq M_l = \sfB(\sfrow(\calT_\alpha), \sfrow(\calT'_\alpha))  \,\,(\cong \calV_\alpha)  
\end{equation*}
be a distinguished filtration with respect to $\{\hscrS_\alpha \mid \alpha \models n\}$,  constructed in the proof of \cref{thm: dist filt for Va}. 
A remarkable property of this filtration is that the submodule \( M_{l-1} \) is uniquely determined, independent of the choice of filtration.
Let us explain how this property is derived. 
Since $\sfrow(\calT_\alpha)$ is a generator of  $\sfB(\sfrow(\calT_\alpha), \sfrow(\calT'_\alpha))$, the minimal index $1 \leq i \leq l$ such that $\sfrow(\calT_\alpha) \in B_i$ is equal to $l$ (for the definition of $B_i$, see \cref{eq: def of set B(k)}).
Combining this with the fact that 
\[
\text{$\hatPtab(\sfrow(\calT_\alpha) w_0) = \calT_\alpha$ \,and\, 
$\hatQtab(\sfrow(\calT_\alpha) w_0) = \calT_{\alpha^\rmr}$,} 
\]
 we have  
$$B_l \setminus B_{l-1}  = \{\sigma \in [\sfrow(\calT_\alpha), \sfrow(\calT'_\alpha)]_L \mid \sh(\hatPtab(\sfrow(\sigma w_0))) = \alpha \text{ and } \hatQtab(\sfrow(\sigma w_0)) = \calT_{\alpha^\rmr}  \}.
$$
This shows that the submodule \(M_{l-1}\) is given by the \(\C\)-span of  
\begin{equation}\label{eq: basis for Mlm1}
\{\sigma \in [\sfrow(\calT_\alpha), \sfrow(\calT'_\alpha)]_L \mid \sh(\hatPtab(\sigma w_0)) \neq \alpha \; \text{or} \; \hatQtab(\sigma w_0) \neq \calT_{\alpha^\rmr}\}
\end{equation}
and it depends only on $\alpha$, not on the choice of filtration.

Now, define the $H_n(0)$-module $\bfY_\alpha$ by
$$\bfY_\alpha:= M_l/M_{l-1}.$$
We have a surjective $H_n(0)$-module homomorphism
\[
\delta: \calV_\alpha \stackrel{\Theta_V}{\to} \sfB(\sfrow(\calT_\alpha), \sfrow(\calT'_\alpha)) \stackrel{\rm pr}{\twoheadrightarrow} \bfY_\alpha,
\]
where ${\rm pr}: \sfB(\sfrow(\calT_\alpha), \sfrow(\calT'_\alpha)) \ra \bfY_\alpha$ is the natural projection.
By \cref{eq: basis for Mlm1}, $\bfY_\alpha$ can be viewed as the $H_n(0)$-module whose underlying space  is the \(\C\)-span of 
\begin{equation}\label{eq: def of K alpha}
\calK_\alpha := \{\sigma \in \SG_n \mid \sh(\hatPtab(\sigma w_0)) = \alpha \text{ and } \hatQtab(\sigma w_0) = \calT_{\alpha^\rmr}\},
\end{equation}
equipped with the \(H_n(0)\)-action defined as follows: for \(\sigma \in \calK_\alpha\) and \(1 \leq i \leq n-1\),
\[
\pi_i \cdot \sigma = 
\begin{cases}
\sigma & \text{if } i \in \Des_L(\sigma), \\
s_i \sigma & \text{if } i \notin \Des_L(\sigma) \text{ and } s_i \sigma \in \calK_\alpha, \\
0 & \text{otherwise.}
\end{cases}
\]
Since  $|{\rm DIRT}(\alpha, \alpha^\rmr)| = 1$ by \cite[Theorem 1.1]{18AHM},
the basis $\calK_\alpha$ can be rewritten as 
\begin{align}\label{eq: simple descript for Ga}
 \{\sfrow(\calT) \mid \calT \in \SIT(\alpha)  \text{ and }  \sh(\hatPtab(\sfrow(\calT) w_0)) = \alpha 
\}.
\end{align}

Recall that every automorphism of \(H_n(0)\) induces 
an equivalence on the category of finitely generated \(H_n(0)\)-modules.
Let \(\upphi: H_n(0) \to H_n(0)\) be the automorphism defined by \(\pi_i \mapsto \pi_{n-i}\).  
For an  \(H_n(0)\)-module \(M\), let \(\upphi[M]\) denote the \(\upphi\)-twist of \(M\).  
For more details, see \cite[Section 3.4]{22JKLO}.

\begin{theorem} \label{thm: new modules for quasisymmetric Schur and Young quasi Schur}
Let $\alpha$ be a composition of $n$.
\begin{enumerate}[label = {\rm (\arabic*)}]
    \item \(\bfY_\alpha \) is an indecomposable $H_n(0)$-module 
whose image under the quasisymmetric characteristic is 
$\hscrS_\alpha$.
\item 
The \(\upphi\)-twist \(\upphi[\bfY_\alpha]\) of \(\bfY_\alpha\) is an indecomposable \(H_n(0)\)-module whose image under the quasisymmetric characteristic is \(\scrS_\alpha\).
\end{enumerate}
\end{theorem}

\begin{proof}
(1) 
Note that \(\calK_\alpha \cdot w_0\) is an equivalence class under \(\simeq_M\) containing $\sfrow(\calT_\alpha) w_0$ and \(\ch([\bfY_\alpha])\) is given by
\[
\sum_{\sigma \in \calK_\alpha} F_{\comp(\Des_L(\sigma))^\rmc} \,\, \left(= \sum_{\sigma \in \calK_\alpha \cdot w_0} F_{\comp(\Des_L(\sigma))} \right).
\]
This shows that $\ch([\bfY_\alpha]) = \hscrS_\alpha$.

Next, we show that \(\bfY_\alpha\) is indecomposable. 
It was shown in \cite[Theorem 3.2]{22CKNO} that there exists a surjective \(H_n(0)\)-module homomorphism \(\Phi: \overline{\bfP}_{\alpha^{\rmc}} \to \calV_\alpha\). 
Here, \(\overline{\bfP}_{\alpha^{\rm c}}\) denotes the projective indecomposable \(H_n(0)\)-module indexed by \(\alpha^{\rm c}\) as defined in \cite[(2.2)]{22CKNO}. 
Then, we have a surjective \(H_n(0)\)-module homomorphism \(\delta \circ \Phi: \overline{\bfP}_{\alpha^{\rmc}} \to \bfY_\alpha\). 
The desired result follows from the general fact that every quotient of a projective indecomposable \(H_n(0)\)-module is indecomposable.

(2) 
Note that the automorphism \(\upphi\) induces an equivalence  on the category of finitely  generated \(H_n(0)\)-modules. 
Hence, \(\upphi[\bfY_\alpha]\) is indecomposable since \(\bfY_\alpha\) is indecomposable by (1). 
Furthermore, we  have \(\ch([\upphi[\bfY_\alpha]]) = \scrS_\alpha\) since \(\upphi[\bfF_\beta] \cong \bfF_{\beta^{\rm r}}\) (for instance, see \cite[Table 2]{22JKLO}).
\end{proof}

\subsection{A surjection series containing $\bfY_\alpha$}\label{subsec: A surjection series containing Yal}
Let $\alpha \models n$.
It was shown in \cite[Corollary 4.6]{22CKNO} that there exists a series of surjective \( H_n(0) \)-module homomorphisms given by
\begin{align*}
\overline{\bfP}_{\alpha^{\rmc}} \stackrel{\Phi}{\twoheadrightarrow} \calV_\alpha \stackrel{\Gamma}{\twoheadrightarrow} X_\alpha \stackrel{\tilde{\eta}}{\twoheadrightarrow} \hbfS_{\alpha, C},
\end{align*}
where 
\begin{itemize}
\item 
\(\Phi\) is the surjective \(H_n(0)\)-module homomorphism described in \cite[(3.2)]{22CKNO},  
\item \(\Gamma \) is the surjective \(H_n(0)\)-module homomorphism such that for $\calT \in \SIT(\alpha)$,
\begin{equation}\label{eq: def of the map Gamma}
\Gamma(\calT) = \begin{cases}
    \calT & \text{if $\calT \in \SET(\alpha)$,}\\
    0 & \text{otherwise}
\end{cases}    
\end{equation}
as given in \cite[(3.4)]{22CKNO},  
\item
\(\tilde{\eta}\) is the surjective \(H_n(0)\)-homomorphism such that
for $T \in \SET(\alpha)$,
\begin{equation}\label{eq: def of the map eta}
\tilde{\eta}(T) = \begin{cases}
    T & \text{if $T \in \SYCT(\alpha, C)$,}\\
    0 & \text{otherwise},
\end{cases}
\end{equation}
as detailed in \cite[Section 3.3 and Section 4.2]{22CKNO}. 
\end{itemize}

\begin{remark}
(1) In \cite{22CKNO}, the map \(\tilde{\eta}\) is defined using a different formulation. Nevertheless, it can be readily verified that this definition is equivalent to the one given in \eqref{eq: def of the map eta}.

(2) Using \cref{eq: def of the map eta}, one can easily verify that the essential epimorphism $\eta: \overline{\bfP}_{\alpha^{\rmc}} \twoheadrightarrow \hbfS_{\alpha, C}$ in \cite[Theorem 5.3]{22CKNO} satisfies $$\eta = \tilde{\eta} \circ \Gamma \circ \Phi.$$
\end{remark}

In this subsection, we show that there is a surjective $H_n(0)$-module homomorphism 
$\Upsilon: \bfY_\alpha \ra \hbfS_{\alpha, C}$
such that   
the diagram 
\[
\begin{tikzpicture}
\def\hp{2.6} 
\def\vp{1.5}  
\node at (0, 0) {$\overline{\bfP}_{\alpha^{\rmc}}$};
\node at (1*\hp, 0) {$\calV_\alpha$};
\node at (2*\hp, 0) {$X_\alpha$};
\node at (3*\hp, 0) {$\hbfS_{\alpha, C}$};
\node at (2.45*\hp, -1*\vp) {$\bfY_\alpha$};

\draw[->>] (0.2*\hp, 0) --(0.8*\hp, 0);
\node[above] at (0.5*\hp, 0) {\scriptsize $\Phi$};
\draw[->>] (1.2*\hp, 0) --(1.8*\hp, 0);
\node[above] at (1.5*\hp, 0) {\scriptsize $\Gamma$};
\draw[->>] (2.2*\hp, 0) --(2.8*\hp, 0);
\node[above] at (2.5*\hp, 0) {\scriptsize $\tilde{\eta}$};

\draw[->>] (1.2*\hp, -0.2*\vp) -- (2.3*\hp, -1*\vp);
\node[left] at (1.65*\hp, -0.6*\vp) {\scriptsize $\delta$};

\draw[->>] (2.6*\hp, -1*\vp) -- (2.9*\hp, -0.2*\vp); 
\node[right] at (2.75*\hp, -0.6*\vp) {\scriptsize $\Upsilon$};

\node at (2.4*\hp, -0.5*\vp) {\scalebox{1.5}{$\circlearrowleft$}};
\end{tikzpicture}
\]
commutes.
In the case where \(\alpha\) is  a shuffle of a partition and \((1^k)\)
for some \(k \geq 0\), we also show that there exists
a surjective $H_n(0)$-module homomorphism  $\tilde{\delta}: X_\alpha \ra \bfY_\alpha$ such that the diagram 
\[
\begin{tikzpicture}
\def\hp{2.6} 
\def\vp{1.5}  
\node at (0, 0) {$\overline{\bfP}_{\alpha^{\rmc}}$};
\node at (1*\hp, 0) {$\calV_\alpha$};
\node at (2*\hp, 0) {$X_\alpha$};
\node at (3*\hp, 0) {$\hbfS_{\alpha, C}$};
\node at (2.45*\hp, -1*\vp) {$\bfY_\alpha$};

\draw[->>] (0.2*\hp, 0) --(0.8*\hp, 0);
\node[above] at (0.5*\hp, 0) {\scriptsize $\Phi$};
\draw[->>] (1.2*\hp, 0) --(1.8*\hp, 0);
\node[above] at (1.5*\hp, 0) {\scriptsize $\Gamma$};
\draw[->>] (2.2*\hp, 0) --(2.8*\hp, 0);
\node[above] at (2.5*\hp, 0) {\scriptsize $\tilde{\eta}$};

\draw[->>] (1.2*\hp, -0.2*\vp) -- (2.3*\hp, -1*\vp);
\node[left] at (1.65*\hp, -0.6*\vp) {\scriptsize $\delta$};

\draw[->>] (2.6*\hp, -1*\vp) -- (2.9*\hp, -0.2*\vp); 
\node[right] at (2.75*\hp, -0.6*\vp) {\scriptsize $\Upsilon$};

\node at (2.55*\hp, -0.5*\vp) {\scalebox{1.4}{$\circlearrowleft$}};
\node at (1.95*\hp, -0.5*\vp) {\scalebox{1.4}{$\circlearrowright$}};

\draw[->>] (2.1*\hp, -0.2*\vp) --(2.4*\hp, -0.8*\vp) ;
\node[left] at (2.28*\hp, -0.6*\vp) {\scriptsize $\tilde{\delta}$}; 
\end{tikzpicture}
\]
commutes.

\begin{proposition}\label{prop: existence of Gamma map}
For a composition $\alpha$ of $n$, there is a unique $H_n(0)$-module homomorphism $\Upsilon: \bfY_\alpha \ra \hbfS_{\alpha, C}$ such that $\Upsilon \circ \delta=\tilde{\eta} \circ \Gamma $.
\end{proposition}

\begin{proof}
To prove the assertion, it suffices to show that 
$\ker(\tilde{\eta} \circ \Gamma) \supseteq \ker(\delta)$.
To this end, we first describe  $\ker(\tilde{\eta} \circ \Gamma)$ and $\ker(\delta)$ explicitly.
By \cref{eq: def of the map Gamma} and \cref{eq: def of the map eta}, $\ker(\tilde{\eta}\circ \Gamma)$ is given by 
\[
\ker(\tilde{\eta}\circ \Gamma) = \C\{\calT \in \SIT(\alpha) \mid \calT \notin \SYCT(\alpha; C) \}.
\]
On the other hand, by the definition of $\bfY_\alpha$, we have 
$\ker(\delta) = {\Theta_V}^{-1}(M_{l-1})$.
Combining this equality with \cref{eq: simple descript for Ga}, we obtain  
\[
\ker(\delta) =  \C\{\calT \in \SIT(\alpha) \mid \sh(\hatPtab(\sfrow(\calT)w_0)) \neq \alpha\}.
\]
To establish the desired inclusion, we will use mathematical induction on \( n = |\alpha| \). Specifically, we aim to show that if \(\tau \in \SYCT(\alpha; C)\), then \(\hatPtab(\sfrow(\tau)w_0) = \tau\). 
If $\ell(\alpha) = 1$, there is nothing to prove.
So, assume that $l:=\ell(\alpha) > 1$.
Let $\sigma := \sfrow(\tau)w_0$ and $\sigma' := \sigma[1:n-1]$.
Let $\tau'$ be the filling obtained from $\tau$ by removing the cell containing $\tau(1, \alpha_1)\,(= \sigma(n))$.
Then, $\tau'$ is a Young composition tableau filled with distinct entries and its shape is 
\[
\beta = \begin{cases}
   (\alpha_2, \ldots, \alpha_l) & \text{if $\alpha_1 = 1$,}\\
   (\alpha_1-1, \alpha_2, \ldots, \alpha_l) & \text{if $\alpha_1 > 1$.}
\end{cases}
\]
By the induction hypothesis, we have 
$\hatPtab(\sigma') = \tau'$.
Thus, it remains to show that 
\begin{align}\label{text: suff assertion}
\text{$\tau(1, \alpha_1)$ is placed in the cell $(1, \alpha_1)$ in $\hatPtab(\sigma)$.}
\end{align}
If $\alpha_1 = 1$, then \cref{text: suff assertion} follows immediately since $\tau(1,1) = 1$.
From now on, suppose that $\alpha_1 > 1$.
Then \cref{text: suff assertion} follows by showing that for \((i, j) \in \tcd(\beta)\),  
\[
\text{if \(i \geq 2\) and \(j \geq \alpha_1 - 1\), then \(\tau(i, j) > \tau(1, \alpha_1)\).}
\]
In the case where $j \geq \alpha_1$, since $\tau \in \SYCT(\alpha; C)$, it follows that $\tau(i, j) \geq \tau(i, \alpha_1) > \tau(1, \alpha_1)$.
Now, consider the case where \( j = \alpha_1 - 1 \). Suppose, for the sake of contradiction, that \( \tau(i, \alpha_1 - 1) < \tau(1, \alpha_1) \). 
By Young triple rule, this inequality implies that \( (i, \alpha_1) \in \tcd(\alpha) \) and \( \tau(i, \alpha_1) < \tau(1, \alpha_1) \). However, this contradicts the assumption that \( \tau \in \SYCT(\alpha; C) \).
\end{proof}

\begin{remark}
In \cite[Section 7]{15TW}, Tewari and van Willigenburg introduced the notion of \emph{simple compositions} to characterize when $\bfS_\alpha$ is indecomposable.
A composition \(\alpha\) is called \emph{simple} if, for any \(1 \leq i < j \leq \ell(\alpha)\) with \(\alpha_i \geq \alpha_j \geq 2\), there exists an integer \(k\) such that \(i < k < j\) and \(\alpha_k = \alpha_j - 1\). 
Since \(\hbfS_{\alpha} = \upphi[\bfS_{\alpha^\rmr}]\) by \cite[(4.2)]{22CKNO}, \cite[Theorem 7.6]{15TW} says that 
\(\hbfS_{\alpha}\) is indecomposable, equivalently, \(\hbfS_{\alpha} = \hbfS_{\alpha, C}\) if and only if \(\alpha^\rmr\) is simple. 
Combining this with the equality \(\ch([\bfY_\alpha]) = \hscrS_\alpha = \ch([\hbfS_\alpha])\), it follows from \cref{prop: existence of Gamma map} that \(\Upsilon: \bfY_\alpha \to \hbfS_{\alpha, C}\) is an isomorphism if and only if \(\alpha^\rmr\) is a simple composition. 
\end{remark}

Next, we introduce the second main result of this subsection.
\begin{proposition}\label{prop: exitence of tildel map}
If \(\alpha\) is a composition of \(n\) obtained by shuffling a partition and \((1^k)\) for some \(k \geq 0\), then there exists a unique \(H_n(0)\)-module homomorphism \(\tilde{\delta}: X_\alpha \to \bfY_\alpha\) such that \(\delta = \tilde{\delta} \circ \Gamma\).
\end{proposition}
\begin{proof}
As in the proof of \cref{prop: existence of Gamma map},
it suffices to show that $\ker(\delta) \supseteq \ker(\Gamma)$.
Note that 
\begin{align*}
&\ker(\delta) = \C \{\calT \in \SIT(\alpha) \mid \sfrow(\calT) \notin \calK_\alpha\} \text{ and }\\
&\ker(\Gamma) = \C \{\calT \in \SIT(\alpha) \mid \calT \notin \SET(\alpha)\}
\end{align*}
(for the definition of $\calK_\alpha$, see \eqref{eq: def of the map Gamma}).
Note that  $\calK_\alpha \cdot w_0$ is an equivalence class under $\simeq_M$ and it contains $\sfrow(\calT_\alpha) w_0$.
Combining this with \cref{prop: SET is closed under simeqM}, we have 
\begin{equation*}
    \calK_\alpha \subseteq \{\sfrow(T) \mid T \in \SET(\alpha) \}.
\end{equation*}
Therefore, if $\calT \in \SIT(\alpha) \setminus \SET(\alpha)$, then $\sfrow(\calT) \notin \ker(\delta)$, as required.
\end{proof}

\subsection{Weak Bruhat interval module structure for $\bfY_\alpha$}\label{subsec: Weak Bruhat interval module structure for Yal}
Unless otherwise specified, in this subsection, \(\alpha\) refers to a composition of \(n\) obtained by shuffling a partition and \((1^k)\) for some \(k \geq 0\). 
The purpose of this subsection is to prove that \(\calK_\alpha\) consists of the reading words of SE-decreasing standard extended tableaux and that \(\calK_\alpha\)  forms a left weak Bruhat interval.
Consequently, \(\bfY_\alpha\) is endowed with the structure of a weak Bruhat interval module.

\begin{definition}\label{def: SE-decreasing}
Let $\alpha$ be a composition of $n$.
We say that a filling $T$ of $\tcd(\alpha)$
is {\it southeast decreasing (simply, SE-decreasing)}  if  
for all \(1 \leq i < j \leq \ell(\alpha)\) with $\alpha_i > \alpha_j$, 
\begin{align*}
T(i, k+1) < T(j, k) \quad\text{for all $1 \leq k  \leq \alpha_j$.}
\end{align*}   
\end{definition}

\begin{lemma}\label{lem: property2 of top of Va}
Let \(\alpha\) be a composition of \(n\)
and \(\calT \in \SIT(\alpha)\). 
If \(\sh(\hatPtab(\sfrow(\calT)w_0)) = \alpha\), then 
$\calT$ is SE-decreasing.  
\end{lemma}
\begin{proof}
For simplicity, let \(\sigma := \sfrow(\calT)w_0\), \(\lambda := \lambda(\alpha)\), and \(l := \ell(\alpha)\). 
Assume, for the sake of contradiction, that there exist \(1 \leq i < j \leq l\) with \(\alpha_i > \alpha_j\) such that \(\calT(i, t+1) > \calT(j, t)\) for some \(1 \leq t \leq \alpha_j\).
Let
\[
R := \{1 \leq r \leq l \mid \alpha_r \geq \alpha_i\}, \quad k := |R|,
\]
and let \(u = (u^{(r)})_{r \in R}\) denote the \(k\)-tuple of increasing subsequences of \(\sigma\) defined as
\[
u^{(r)} := 
\begin{cases}
   \calT(j, 1)\,\calT(j, 2)\,\ldots\,\calT(j, t)\,\calT(i, t+1)\,\ldots\,\calT(i, \alpha_i) & \text{if \(r = i\),} \\
   \calT(r, 1)\,\calT(r, 2)\,\ldots\,\calT(r, \alpha_r) & \text{otherwise.}
\end{cases}
\]
Since $j \notin R$, the length of $\mathbf{u}$ is $\sum_{r \in R} \alpha_r$ $(= \sum_{1 \leq r \leq k} \lambda_r)$.
Combining this with  $\sh(\Ptab(\sigma)) = \lambda$, we derive
\(u \in \rmInc_k(\sigma)\) from
\cref{thm: Greene k-dec}.
On the other hand, by  
\cref{thm: an analoge of Greene} and  \cref{lem: first col of SIT}, the assumption  $\sh(\hatPtab(\sigma))  = \alpha$ implies that
\[
\sfmIES_k(\sigma) = \{\calT(r, 1) \mid r \in R\}.
\]
Since \(\calT(i, 1) < \calT(j, 1)\) with $i\in R$ and $j\notin R$, it follows that 
\(\sfmIES_k(\sigma) <_{\rmset} \sfIES(u)\), 
leading to a contradiction.
\end{proof}

Let 
$$\nSYCT(\alpha):=\{T\in \SET(\alpha)\mid \text{$T$ is SE-decreasing}\}.$$
The following theorem presents the first main result of this subsection.

\begin{theorem}\label{prop: description for Gamma wrt nSYCT}
If \(\alpha\) is a composition of \(n\) obtained by shuffling a partition and \((1^k)\) for some \(k \geq 0\),
then  
\[
\calK_\alpha = \{\sfrow(T) \mid T \in \nSYCT(\alpha) \}.
\]
\end{theorem}

\begin{proof}
By \cref{prop: SET is closed under simeqM} and \cref{lem: property2 of top of Va}, 
we have $\calK_\alpha \subseteq  \{\sfrow(T) \mid T \in \nSYCT(\alpha) \}$.
Thus, it remains to prove that $\calK_\alpha \supseteq  \{\sfrow(T) \mid T \in \nSYCT(\alpha) \}$.
By \cref{eq: simple descript for Ga}, this reduces to showing that $\sh(\hatPtab(\sfrow(T) w_0) = \alpha$ for all $T \in \nSYCT(\alpha)$.

Let \(T \in \nSYCT(\alpha)\). 
For simplicity, let 
$\sigma := \sfrow(T)w_0$,
\(\lambda := \lambda(\alpha)\), and \(l := \ell(\lambda)\).
Note that every increasing subsequence of $\sigma$ contains at most one  element in each column of $T$.
More precisely, every increasing subsequence of length \(k\) can be expressed as 
\begin{align}\label{eq: observation on inc subseq}
T(i_1, j_1) T(i_2, j_2) \ldots T(i_k, j_k),
\end{align}
where \(l \geq i_1 \geq i_2 \geq \cdots \geq i_k \geq 1\) and \(1 \leq j_1 < j_2 < \cdots < j_k \leq l\).  
From this, it follows that 
\begin{align*}
    \fraki_k(\sigma) = \lambda_1 + \lambda_2 + \cdots + \lambda_k
\end{align*}
for every \(1 \leq k \leq l\).  
Therefore, by \cref{thm: Greene k-dec}, \(\Ptab(\sigma)\) has shape \(\lambda\).  
Combining this with \cref{lem: first col of SIT} and  \cref{thm: an analoge of Greene}, we deduce that 
\(\sh(\hatPtab(\sigma)) = \alpha\) if and only if  
\begin{equation}\label{eq: inclusion for IES k}
\sfmIES_k(\sigma) \subseteq \{T(i, 1) \mid 1 \leq i \leq l \text{ and } \alpha_i \geq \lambda_k\} \; \text{ for all \(1 \leq k \leq l\)}.
\end{equation}
Let us fix $1 \leq k \leq l$ and $u = (u^{(i)})_{1 \leq i \leq k} \in \rmInc_k(\sigma)$.
To prove \eqref{eq: inclusion for IES k}, we have only to see that 
\[
\sfIES(u) \subseteq \{T(i,1) \mid 1 \leq i \leq l \text{ with $\alpha_i \geq \lambda_k$} \}.
\]

{\it Case 1:} $\lambda_1 = \lambda_k$. 
Let $1 \leq i \leq k$.
Since $\fraki_1(\sigma) = \lambda_1$, the length of $u^{(i)}$ is $\lambda_1$. 
By \cref{eq: observation on inc subseq}, we can write $u^{(i)}$ as 
\[
u^{(i)} = T(i_1, 1) T(i_2, 2) \ldots T(i_{\lambda_1}, \lambda_1),
\]
where $l \geq i_1 \geq i_2 \geq \cdots \geq i_{\lambda_1} \geq 1$.
It suffices to show that $\alpha_{i_1} = \lambda_1$.
Suppose, for the sake of contradiction, that \(\alpha_{i_1} \neq \lambda_1\), i.e., \(\alpha_{i_1} < \lambda_1\). 
Since  \(\alpha_{i_{\lambda_1}} = \lambda_1\),   we can choose the smallest \(1 \leq t \leq \lambda_1\) such that \(\alpha_{i_t} = \lambda_1\).
Then, $1 < t \leq l$.
By the choice of $t$, $i_{t- 1} > i_{t}$ and $\alpha_{i_{t - 1}} < \alpha_{i_{t}} = \lambda_1$.
Since $T$ is SE-decreasing, this implies that $T(i_{t - 1}, t - 1) > T(i_{t}, t)$, 
which contradicts that $u^{(i)}$ is an increasing subsequence.

{\it Case 2:} $\lambda_1 > \lambda_k$.
Let 
$$X := \{T(i, j) \mid 1 \leq i \leq  l \text{ with $\alpha_i > \lambda_k$} \text{ and } \lambda_k < j \leq \alpha_i  \}.$$ 
Let $T'$ be the subfilling obtained from $T$ by removing all elements in $X$ and $\sigma'$ the word obtained from $\sigma$ by removing all elements in $X$. 
Note that $T'$ is an extended tableau of shape $\beta = (\beta_1, \ldots, \beta_l)$, where 
$$\beta_i = \begin{cases} \lambda_k & \text{ if } \alpha_i > \lambda_k,\\  \alpha_i & \text{ otherwise}\end{cases}$$
(for the definition of  extended tableaux, see \cite[Definition 6.16]{19AS}).
Define $u' = (u'^{(i)})_{1 \leq i \leq k}$ as the $k$-tuple of increasing 
subsequences of $\sigma'$, where each $u'^{(i)}$ is obtained from $u^{(i)}$ by removing all entries in $X$.
In view of \cref{eq: observation on inc subseq} and \cref{thm: Greene k-dec}, we have $\sfIES(u) \cap X = \emptyset$, and therefore $\sfIES(u') = \sfIES(u)$.
Thus, to prove the assertion, it suffices to show that {\it Case 1} can be applied to \(\sigma'\) and \(u'\).
Observe that 
\begin{itemize}
\item $\beta$ is a composition that is a  shuffle of a partition and $(1^{k'})$ for some $k' \geq 0$,
\item $T'$ is SE-decreasing, and 
\item \(\sigma'\) is the word obtained by reversing the entries of \(\sfrow(T')\).
\end{itemize}
Therefore, we have only to show that $u' \in \rmInc_k(\sigma')$.
By the construction, it immediately follows that $\mathbf{u'}$ is a $k$-increasing subsequence of $\sigma'$ and 
$$\ell(\mathbf{u'}) \geq \lambda_1 + \lambda_2 + \cdots +\lambda_k -|X| = k \lambda_k.$$
Also, since the shape of $\Ptab(\sigma')$ is 
$(\underbrace{\lambda_k, \ldots, \lambda_k}_{k\text{\;times}}, \lambda_{k+1}, \ldots, \lambda_l)$, \cref{thm: Greene k-dec} says that $\ell(\mathbf{u'}) \leq k \lambda_k$, as required.
\end{proof}

Next, we show that $\calK_\alpha$ forms a left weak Bruhat interval. 
We begin by introducing  sequences of cells in \(\tcd(\alpha)\). 
These sequences are constructed using the following procedure.

\begin{procedure}\label{proc: algo for defining sink}
Let $\alpha = (\alpha_1, \ldots, \alpha_l)$ be a composition of $n$ that is a shuffle of a partition and $(1^k)$ for some $k\ge 0$. 
Let \(i\) be the smallest index \(1 \leq t \leq l\) such that \(\alpha_t\) is the largest part of \(\alpha\).
Set the sequences $c^{(1)}$ and $c^{(2)}$ as \(c^{(1)} = ((i, \alpha_i))\) and \(c^{(2)} = \emptyset\).

\begin{enumerate}
\item[\textbf{P1.}] 
If \(\alpha_i \geq 3\), proceed to \textbf{P2}.  
Otherwise, proceed to \textbf{P3}.

\item[\textbf{P2.}]
If $\alpha_j < \alpha_i - 1$ for all $i < j \leq l$, terminate the procedure.
Otherwise, choose the smallest  \(i < j \leq l\) such that  
\(\alpha_j \geq \alpha_i - 1\); in this case, $\alpha_t = 1$ for all $i < t < j$.
Update \(c^{(1)}\) to be the concatenation of $((j, \alpha_j))$ and  $c^{(1)}$, and
update $i$ to $j$.
Then, proceed to \textbf{P1}.

\item[\textbf{P3.}]
Choose the largest \(i \leq j \leq l\) such that \(\alpha_t = \alpha_i\) for all \(i \leq t \leq j\).  
Update \(c^{(1)}\) to be the concatenation of
$((j, \alpha_j), \ldots,(i+2, \alpha_{i+2}), (i+1, \alpha_{i+1}))$ and  \(c^{(1)}\).
Let $\beta := (\alpha_{j+1}, \ldots, \alpha_{l})$.
\begin{enumerate}[]
    \item[\textbf{P3a.}] If $\beta = \varnothing$,  terminate the procedure.
    Otherwise, proceed to $\textbf{P3b}$.
    \item[\textbf{P3b.}]
Choose the smallest \(1 \leq t \leq \ell(\beta)\) such that \((\beta_t, \ldots, \beta_{\ell(\beta)})\) forms a partition.  
Update \(c^{(2)}\) to be  the concatenation of $c^{(2)}$ and $((\ell(\beta), \beta_{\ell(\beta)}), \ldots, (t, \beta_t))$
and update \(\beta\) to be  $(\beta_1, \ldots, \beta_{t-1}, \beta_t - 1, \ldots, \beta_{\ell(\beta)} - 1)$.  
Then, proceed to \(\textbf{P3a}\).
\end{enumerate}
\end{enumerate}
\end{procedure}
We denote the two resulting sequences \(c^{(1)}\) and \(c^{(2)}\) obtained by applying \cref{proc: algo for defining sink} to \(\alpha\) as \(\sfseq_1(\alpha)\) and \(\sfseq_2(\alpha)\), respectively.
For instance, if $\alpha = (1, 4, 1,4,3,2,2,2, 1,2,1,1,2,1)$, we have
\begin{align*}
    & \sfseq_1(\alpha) = ((8,2), (7,2), (6,2), (5, 3), (4,4), (2,4)) \quad \text{and}\\
    & \sfseq_2(\alpha) = ((14,1), (13,2), (13,1),(12, 1), (11, 1), (10, 2),(10, 1), (9,1)).
\end{align*}
See \cref{fig: example for seq_1 and seq_2}, where the cells in \( \sfseq_1(\alpha) \) are highlighted in red and the cells in \( \sfseq_2(\alpha) \) are highlighted in blue.
\begin{figure}
$\begin{array}{c}
\scalebox{0.5}{
\begin{ytableau}
*(blue!25)  \\
*(blue!25) & *(blue!25)  \\
*(blue!25)  \\
*(blue!25)  \\
*(blue!25) & *(blue!25) \\
*(blue!25) \\
{\color{white}1} & *(red!25) \\
{\color{white}1} & *(red!25) \\
{\color{white}1} & *(red!25) \\
{\color{white}1} & {\color{white}1} & *(red!25) \\
{\color{white}1} & {\color{white}1} & {\color{white}1} & *(red!25)\\
    {\color{white}1}  \\
    {\color{white}1} & {\color{white}1} & {\color{white}1} & *(red!25)\\
    {\color{white}1}
\end{ytableau}}
\end{array}$
\caption{}
\label{fig: example for seq_1 and seq_2}
\end{figure}

Let $r_0$ be the positive integer such that $(r_0, \alpha_{r_0})$ is the initial cell of the sequence $\sfseq_1(\alpha)$.
By the 
assumption on $\alpha$ and
the construction of \(\sfseq_1(\alpha)\) and \(\sfseq_2(\alpha)\), it follows that for \((i, j) \in \tcd(\alpha)\),
\begin{flalign}
    &\text{\((i, j) \in \sfseq_1(\alpha)\) if and only if \(1 \leq i \leq r_0\) with
     \(\alpha_i \geq \alpha_{r_0}\)  and \(j = \alpha_i\),}\label{eq: simple obs on seq one}\\
    &\text{\((i, j) \in \sfseq_2(\alpha)\) if and only if 
 $\alpha_{r_0} = 2$ and \(i > r_0\).}\label{eq: simple obs on seq two}
\end{flalign}
Now, define \(D_\alpha\) as the diagram obtained from \(\tcd(\alpha)\) by removing all cells in the sequences \(\sfseq_1(\alpha)\) and \(\sfseq_2(\alpha)\).

\begin{lemma}\label{lem: removing seq from original comp}
If \(\alpha\) is a composition of \(n\) obtained by shuffling a partition and \((1^k)\) for some \(k \geq 0\), then
$D_\alpha$ is a composition diagram.
Moreover, letting $\tilal :=\sh(D_\alpha)$, 
we have 
\begin{enumerate}[label = {\rm (\arabic*)}]
    \item $\tilal$ is a shuffle of a partition and $(1^{k'})$ for some $k' \ge 0$, and 
    \item 
    for all $1 \leq i <j  \leq \ell(\tilal)$, if $\alpha_i > \alpha_j$, then $\tilal_i > \tilal_j$.
\end{enumerate}
\end{lemma}
\begin{proof}
Let $l := \ell(\alpha)$ and $(r_0, \alpha_{r_0})$ be the initial cell of the sequence $\sfseq_1(\alpha)$.
For each $1 \leq i \leq l$, let $\beta_i$ be the number of cells in the $i$th row of $D_\alpha$. Then,  
\[
\beta_i = 
\begin{cases}
\alpha_i - 1 & \text{if $(i, \alpha_i) \in \sfseq_1(\alpha)$,}\\
0  & \text{if $(i, \alpha_i) \in \sfseq_2(\alpha)$,}\\
\alpha_i & \text{otherwise.}
\end{cases}
\]
We prove our assertion by considering the following three cases for \(\alpha_{r_0}\).

{\it Case 1:  $\alpha_{r_0} = 1$.}
In this case,  by  the construction of $\sfseq_1(\alpha)$, the largest part of $\alpha$ is \(1\).  
Therefore, \(\alpha = (1^n)\) and \(r_0 = n\).  
As a result, \(D_\alpha\) is the empty diagram, implying that \(\tilal = \varnothing\).  
Consequently, \(\tilal\) trivially satisfies conditions (1) and (2).  

{\it Case 2:  $\alpha_{r_0} = 2$.}
Combining  \cref{eq: simple obs on seq one} and \cref{eq: simple obs on seq two} yields that for $1 \leq i \leq l$, 
\[
\beta_i = \begin{cases}
    \alpha_i - 1 & \text{if $1 \leq i \leq r_0$ and $\alpha_i \geq 2$,}\\
    1 & \text{if $1 \leq i \leq r_0$ and $\alpha_i =1$,}\\
    0 & \text{if $r_0 < i \leq l$}.
\end{cases}
\]
This shows that $D_\alpha$ is a composition diagram and $\tilal = (\beta_1, \beta_2, \ldots, \beta_{r_0})$. 
Furthermore,  $\tilal$ satisfies conditions (1) and (2), as can be deduced from the assumption on $\alpha$ and the observation that $\{1 \leq i \leq r_0 \mid \alpha_i = 2 \} = [t, r_0]$ for some $1 \leq t \leq r_0$.

{\it Case 3: $\alpha_{r_0} \geq 3$.} 
Combining  \cref{eq: simple obs on seq one} and \cref{eq: simple obs on seq two} yields that for $1 \leq i \leq l$, 
\[
\beta_i = \begin{cases}
    \alpha_i - 1 & \text{if $1 \leq i \leq r_0$ and $\alpha_i \geq 2$,}\\
    1 & \text{if $1 \leq i \leq r_0$ and $\alpha_i =1$,}\\
    \alpha_i & \text{if $r_0 < i \leq l$}.
\end{cases}
\]
This shows that \( D_\alpha \) is a composition diagram and that \(\tilal = (\beta_1, \beta_2, \ldots, \beta_l)\). 
Furthermore, \(\tilal\) satisfies conditions (1) and (2), as can be deduced from the assumption on $\alpha$ and the observation that $\alpha_i < \alpha_{r_0} - 1$ for all $r_0 < i \leq l$.
\end{proof}

Recall that the shape of \(D_\alpha\) is denoted by \(\tilal\) (see \cref{lem: removing seq from original comp}).
We have $\tcd(\alpha) = \tcd(\tilal) \sqcup \{C \mid C \in \sfseq_i(\alpha) \; \text{for some $i = 1, 2$} \}$.
Now, define the filling $\tau'_\alpha$ of $\tcd(\alpha)$ by
\[
\tau'_\alpha(C) = \begin{cases}
  \tau'_{\tilal}(C)  & \text{if $C \in \tcd(\tilal)$,}\\
  n+1-k& \text{if $C$ is the $k$th entry of $\sfseq(\alpha)$},
\end{cases}
\]
where $\sfseq(\alpha)$ is the sequence obtained by concatenating $\sfseq_2(\alpha)$ and $\sfseq_1(\alpha)$.

\begin{lemma}\label{lem: tau'al in nSYCT}
If \(\alpha\) is a composition of \(n\) obtained by shuffling a partition and \((1^k)\) for some \(k \geq 0\), then $\tau'_\alpha \in \nSYCT(\alpha)$.
\end{lemma}
\begin{proof}
We will prove the assertion using mathematical induction on \(n\). 
If \(n = 1\), then the assertion holds trivially.
Assume \(n > 1\). 
If \(\tilal = \varnothing\), then \(\alpha = (1^n)\), and again,  the assertion holds trivially.
Now, suppose \(\tilal \neq \varnothing\). 
Since \(1 \leq |\tilal| < n\), by \cref{lem: removing seq from original comp}, the induction hypothesis ensures that \(\tau'_{\tilal} \in \nSYCT(\tilal)\). 
From this containment one can derive that $\tau'_\alpha$ is a standard extended tableau.
It remains to show that $\tau'_\alpha$ is SE-decreasing.

Let $(r_0, \alpha_{r_0})$ be the initial cell of $\sfseq_1(\alpha)$.
Since $\tilal \neq \varnothing$, we have $\alpha_{r_0} \geq 2$ and $r_0 \leq \ell(\tilal)$.
Let $1 \leq i < j \leq l$ with $\alpha_i > \alpha_j$ and $1 \leq k \leq \alpha_j$.
We prove our assertion by considering the following three cases.

{\it Case 1:} $j > r_0$. 
In this case, it follows from the definition of $\tau'_\alpha$ that $\tau'_\alpha(j, k) > \tau'_\alpha(i, k+1)$.

{\it Case 2:} $j \leq r_0$ with $\alpha_j > 1$ and $k = \alpha_j$.
In this case, we have $\alpha_j \geq \alpha_{r_0}$ by the assumption on $\alpha$.
This implies that $(i, \alpha_i), (j, \alpha_j) \in \sfseq_1(\alpha)$ 
by \cref{eq: simple obs on seq one}.
Therefore, 
$(j, \alpha_j)$ appears before $(i, \alpha_i)$ in $\sfseq_1(\alpha)$, which implies that
\[
\tau'_\alpha(j, \alpha_j) > \tau'_\alpha(i, \alpha_i) \geq \tau'_\alpha(i, \alpha_j + 1).
\]

{\it Case 3:} If neither of the above applies, by \cref{lem: removing seq from original comp}(2),  we have \(\tilal_i > \tilal_j\).
Additionally, it is easy to see that $k \leq \tilal_j$, and thus,  $k+1 \leq \tilal_i$. 
Combining these observations with $\tau'_{\tilal} \in \nSYCT(\tilal)$, we derive that   
\[
\tau'_\alpha(j, k) = \tau'_{\tilal}(j, k) > \tau'_{\tilal}(i, k+1) = \tau'_\alpha(i, k+1).
\]
\end{proof}

We are now prepared to state the second main result of this subsection.
\begin{theorem} \label{second main result of sec 5.3}
    If \(\alpha\) is a composition of \(n\) obtained by shuffling a partition and \((1^k)\) for some \(k \geq 0\), then
    \[
    \calK_\alpha=[\sfrow(\sfT_\alpha), \sfrow(\tau'_\alpha)]_L.
    \]
\end{theorem}

Before proceeding with the proof, let us collect the necessary notations.  
Define the bijection \(\iota_\alpha : \tcd(\alpha) \to [n]\) by  
\[
\iota_\alpha((i, j)) = \sum_{1 \leq t \leq i} \alpha_t - j + 1.
\]
For a filling \(T\) of \(\tcd(\alpha)\), define  
\[
\rmInvC(T) := \{((i_1, j_1), (i_2, j_2)) \in \tcd(\alpha) \times \tcd(\alpha) \mid i_1 < i_2 \text{ and } T(i_1, j_1) > T(j_2, j_2)\}.
\]
If \(T \in \SET(\alpha)\), then  
\begin{align}\label{eq: inversion of SET}
\rmInv_L(\sfrow(T)) = \{(\iota_\alpha(C_1), \iota_\alpha(C_2)) \mid (C_1, C_2) \in \rmInvC(T)\} \sqcup A_\alpha,
\end{align}
where  
\[
A_\alpha = \{(\iota_\alpha((i, j)), \iota_\alpha((i, k))) \mid 1 \leq i \leq \ell(\alpha) \text{ and } 1 \leq k < j \leq \alpha_i\}.
\]

\medskip 
\noindent
{\it Proof of \cref{second main result of sec 5.3}}.  
For simplicity, let $I := [\sfrow(\sfT_\alpha), \sfrow(\tau'_\alpha)]_L$ and $l := \ell(\alpha)$.

We first show that $\calK_\alpha \supseteq I$.
Let \(\sigma \in I\).    
Observe that \(\sfT_\alpha \in \nSYCT(\alpha)\), and \(\tau'_\alpha \in \nSYCT(\alpha)\) by \cref{lem: tau'al in nSYCT}.  
By \cref{prop: description for Gamma wrt nSYCT}, we obtain  
\[
\sh(\hatPtab(\sfrow(\sfT_\alpha) w_0)) = \alpha = \sh(\hatPtab(\sfrow(\tau'_\alpha) w_0)), 
\]  
and thus 
\[
\sh(\Ptab(\sfrow(\sfT_\alpha) w_0)) = \lambda(\alpha) = \sh(\Ptab(\sfrow(\tau'_\alpha) w_0)).
\]
On the other hand, since \(\sfrow(\tau'_\alpha) w_0 \preceq_L \sigma w_0 \preceq_L \sfrow(\sfT_\alpha) w_0\), it follows from \cref{eq: left bruhat order and shape} that  
$\sh(\Ptab(\sigma w_0)) = \lambda(\alpha)$.
Therefore, by \cref{lem: shape and SIT word},
we derive that 
$\sh(\hatPtab(\sigma w_0)) = \alpha$,
and thus, $\sigma \in \calK_\alpha$.

Next, we show that $\calK_\alpha \subseteq I$.
Let $T \in \nSYCT(\alpha)$.
Since it has been already shown that $\sfrow(\sfT_\alpha) \preceq_L \sfrow(T)$, it suffices to show that $\sfrow(T) \preceq_L \sfrow(\tau'_\alpha)$, that is, $\rmInv_L(\sfrow(T)) \subseteq \rmInv_L(\sfrow(\tau'_\alpha))$.
By \cref{eq: inversion of SET}, this is equivalent to 
\begin{equation}\label{inclusion of InvC}
\rmInvC(T) \subseteq \rmInvC(\tau'_\alpha).
\end{equation}

Before proving this inclusion, we partition $\rmInvC(T)$ into three  disjoint subsets.
Let $(r_0, \alpha_{r_0})$ be the cell in the first entry of $\sfseq_1(\alpha)$.
Let $T^{(1)}$ be the subfilling of $T$ consisting of all cells not in $\sfseq_1(\alpha)$ or $\sfseq_2(\alpha)$, and let $T^{(2)}$ be the subfilling of $T$ consisting of the cells in $\sfseq_2(\alpha)$. 
For later use, we note that the standardizations of $T^{(1)}$ and $T^{(2)}$ are SE-decreasing standard extended tableaux.
In particular, when $T$ is equal to $\tau'_\alpha$, it is not difficult to see that the standardizations of $T^{(1)}$ and $T^{(2)}$ are $\tau'_{\tilal}$ and $\tau'_\beta$, respectively,
where $\beta = (\alpha_{r_0 + 1}, \ldots, \alpha_l)$.
We claim that 
\begin{align}\label{eq: partition InvC into three parts}
\rmInvC(T) = \rmInvC(T^{(1)}) \sqcup \rmInvC(T^{(2)}) \sqcup X, 
\end{align}
where 
$$X = \{(C_1, C_2) \in  \rmInvC(T) \mid C_1 \in \sfseq_1(\alpha), C_2 \in \tcd(\tilal) \}.$$
The inclusion \(\supseteq\) is evident. To show the reverse inclusion \(\subseteq\), we will verify that for \((i_1, j_1), (i_2, j_2) \in \tcd(\alpha)\) with \(i_1 < i_2\), the following conditions are satisfied:
\begin{itemize}
\item[(i)] if $(i_2, j_2) \in \sfseq_1(\alpha)$, then $T(i_1, j_1) < T(i_2, j_2)$ and 
\item[(ii)] if $(i_1, j_1) \notin \sfseq_2(\alpha)$ and $(i_2, j_2) \in \sfseq_2(\alpha)$, then $T(i_1, j_1) < T(i_2, j_2)$.
\end{itemize}
For (i), it suffices to consider the case where $\alpha_{i_1} > 1$.
By the assumption on $\alpha$, we see that $\alpha_{i_1} \geq \alpha_{i_2}$.
Since $j_2=\alpha_{i_2}$ and $T \in \nSYCT(\alpha)$, it follows that
$$
T(i_1, j_1) \leq T(i_1, \alpha_{i_1}) < T(i_2, \alpha_{i_2}) = T(i_2, j_2).
$$
Now, it remains to show (ii).  
From the condition $\sfseq_2(\alpha) \neq \emptyset$, we deduce that $\alpha_{r_0} = 2$ and $\alpha_{r_0 + 1} = 1$.
Since $T$ is SE-decreasing, it follows that 
$T(r_0, 2) < T(r_0 +1 , 1)$.
Moreover, we have 
\[
T(i_1, j_1) \leq T(r_0, 2) \text{ and } T(r_0 + 1, 1) \leq T(i_2, j_2),
\]
where the first inequality follows from (i), while the second one follows from $T \in \nSYCT(\alpha)$.
Combining these three inequalities   
verifies (ii).

In what follows, we prove the inclusion given in \cref{inclusion of InvC} using mathematical induction on $n = |\alpha|$.
For $n=0, 1$, this inclusion trivially holds.
Now, assume that $n > 1$ and the inclusion \eqref{inclusion of InvC} holds for all 
compositions of size \(< n\) obtained by shuffling a partition and \((1^k)\) for some \(k \geq 0\).
Let $\alpha$ be a composition of $n$ obtained by shuffling a partition and $(1^k)$ for some $k \geq 0$.
By \cref{lem: removing seq from original comp}, $\tilal$ is a composition obtained by shuffling a partition and $(1^k)$ for some $k \geq 0$.
Since $|\tilal| < n$ and  the standardization of  $T^{(1)}$ is an element of $\nSYCT(\tilal)$,  by  the induction hypothesis, we have 
\begin{align}\label{eq: first ind hypo}
    \rmInvC(T^{(1)}) \subseteq \rmInvC(\tau'_{\tilal}).
\end{align}
In a similar manner, viewing $\tau'_{\beta}$ as the subfilling of $\tau'_\alpha$ consisting of the cells in $\sfseq_2(\alpha)$, 
we obtain 
\begin{align}\label{eq: second ind hypo}
    \rmInvC(T^{(2)}) \subseteq \rmInvC(\tau'_{\beta}).
\end{align}
On the other hand, applying \cref{eq: partition InvC into three parts} to the case where $T$ is equal to $\tau'_\alpha$,  
it follows from the definition of $\tau'_\alpha$ that 
\[
\rmInvC(\tau'_\alpha) = \rmInvC(\tau'_{\tilal}) \sqcup \rmInvC(\tau'_\beta) \sqcup Y, 
\]
where 
\begin{align*}
Y = \{((i_1, j_1), (i_2, j_2)) \mid (i_1, j_1) \in \sfseq_1(\alpha), \text{ } (i_2, j_2) \in \tcd(\tilal), \text{ and $i_1 < i_2$} \}.
\end{align*}
Finally, by combining \cref{eq: first ind hypo} and \cref{eq: second ind hypo} with the inclusion  \(X \subseteq Y\), we obtain the desired result.  
\qed

\medskip
\cref{second main result of sec 5.3} tells us that 
\(\bfY_\alpha\) is endowed with the structure of a weak Bruhat interval module:
\begin{align*}
\bfY_\alpha \cong \sfB(\sfrow(\sfT_\alpha), \sfrow(\tau'_\alpha)).
\end{align*}
Combining this with \cite[Table 1]{22JKLO}  implies that the \(\upphi\)-twist \(\upphi[\bfY_\alpha]\) of \(\bfY_\alpha\) also has a weak Bruhat interval module structure:
\[
\upphi[\bfY_\alpha] \cong \sfB(\sfrow(\sfT_\alpha)^{w_0}, \sfrow(\tau'_\alpha)^{w_0}).
\]

\begin{example}
Let $\alpha = (3,1,2)$. Then  
\[
\tau'_\alpha = 
\begin{array}{c}
\begin{ytableau}
    4 & 6\\
    3\\
    1 & 2 & 5
\end{ytableau}
\end{array}.
\]
On the other hand, considering \cref{fig: figure for example X},
we have $\bfY_\alpha = \C B_3/ \C B_2$, and thus, we have
\[
\bfY_\alpha \cong \sfB(321465, 521364) = \sfB(\sfrow(\sfT_\alpha), \sfrow(\tau'_\alpha)).
\]
\end{example}

\begin{remark}\label{rem: WBIM structure of Yal}
For a general composition \(\alpha\), \(\bfY_\alpha\) does not necessarily have a weak Bruhat interval module structure.
For instance, let $\alpha = (2,3,1)$.
In this case, the $H_6(0)$-action on $\calK_\alpha$ is illustrated  in \cref{fig: Ya not WBIM str}.
\begin{figure}[t]
\[
\def \vp {1.3}
\def \hp {2.5}
\begin{tikzpicture}
\node[] at (0*\hp, 0*\vp) {$215436$};

\node[] at (-1*\hp, -1*\vp) {$315426$};
\node[] at (1*\hp, -1*\vp) {$216435$};

\node[] at (-2*\hp, -2*\vp) {$415326$};
\node[] at (0*\hp, -2*\vp) {$316425$};

\node at (0*\hp + 0.2*\hp, 0*\vp) {} edge [out=40,in=320, loop] ();
\node[right] at (0.4*\hp, 0*\vp) {\footnotesize $\pi_1, \pi_3, \pi_4$};

\node at (-1*\hp + 0.2*\hp, -1*\vp) {} edge [out=40,in=320, loop] ();
\node[right] at (-1*\hp + 0.4*\hp, -1*\vp) {\footnotesize $\pi_2, \pi_4$};
\node at (1*\hp + 0.2*\hp, -1*\vp) {} edge [out=40,in=320, loop] ();
\node[right] at (1*\hp + 0.4*\hp, -1*\vp) {\footnotesize $\pi_1, \pi_3, \pi_5$};

\node at (-2*\hp + 0.2*\hp, -2*\vp) {} edge [out=40,in=320, loop] ();
\node[right] at (-2*\hp + 0.4*\hp, -2*\vp) {\footnotesize $\pi_2, \pi_3$};
\node at (0*\hp + 0.2*\hp, -2*\vp) {} edge [out=40,in=320, loop] ();
\node[right] at (0*\hp + 0.4*\hp, -2*\vp) {\footnotesize $\pi_2, \pi_5$};

\draw[->] (-0.2*\hp, -0.3*\vp) -- (-0.8*\hp, -0.7*\vp);
\node at (-0.6*\hp, -0.4*\vp) {\footnotesize $\pi_2$};
\draw[->] (0.2*\hp, -0.3*\vp) -- (0.8*\hp, -0.7*\vp);
\node at (0.6*\hp, -0.4*\vp) {\footnotesize $\pi_5$};

\draw[->] (-1.2*\hp, -1.3*\vp) -- (-1.8*\hp, -1.7*\vp);
\node at (-1.6*\hp, -1.4*\vp) {\footnotesize $\pi_3$};
\draw[->] (-0.8*\hp, -1.3*\vp) -- (-0.2*\hp, -1.7*\vp);
\node at (-0.4*\hp, -1.4*\vp) {\footnotesize $\pi_5$};

\draw[->] (0.8*\hp, -1.3*\vp) -- (0.2*\hp, -1.7*\vp);
\node at (0.4*\hp, -1.4*\vp) {\footnotesize $\pi_2$};
\end{tikzpicture}
\]
\caption{The $H_6(0)$-action on $\calK_{(2,3,1)}$}
\label{fig: Ya not WBIM str}
\end{figure}
We claim that there is no left weak Bruhat interval $I$ in $\SG_6$ such that $\bfY_\alpha \cong \sfB(I)$.
Suppose, for the sake of contradiction, that such an interval \(I = [\sigma, \rho]_L\) exists.
Let $f: \bfY_\alpha \ra \sfB(\sigma, \rho)$ be an $H_6(0)$-module  isomorphism.
Note that there exists a unique \(\gamma \in \calK_\alpha\) such that \(\Des_L(\gamma) \supseteq \{1, 3, 4\}\) and this unique element is \(215436\).
Therefore, there exists a unique $\gamma' \in [\sigma, \rho]_L$ such that $\Des_L(\gamma') \supseteq \{1,3,4\}$ and $f(215436) = \gamma'$.
Since $215436$ is a generator for $\bfY_\alpha$, $\gamma'$ must be $\sigma$.
This contradicts our assumption that $I$ is a left weak Bruhat itnerval.
\end{remark}

\section{Remarks concerned with future research}
\label{Further avenues}

\noindent
(1) As seen in \cref{rem: WBIM structure of Yal}, for a general composition $\alpha$, $\calK_\alpha$ is not necessarily a left weak Bruhat interval.
It would be nice to characterize when  $\calK_\alpha$ is a left weak Bruhat interval, more generally, when   $\bfY_\alpha$ has a weak Bruhat interval module structure.

\noindent
(2)
Let $I$ be a left weak Bruhat interval in $\SG_n$.
It was shown in~\cite[Section 6.2]{23KLO} that if \( I \) is a dual plactic-closed,
then \( B(I) \) admits a distinguished filtration with respect to the Schur basis. 
A natural next step is to characterize the conditions under which the module \( B(I) \) admits a distinguished filtration with respect to $\{\scrS_\alpha \mid \alpha \models n\}$ or $\{\hscrS_\alpha \mid \alpha \models n\}$ (see \cref{no filtration example}).  

\noindent
(3) 
When \( I \) is a dual plactic-closed left weak Bruhat interval in \( \SG_n \), 
let 
\[
0 = M_0 \subsetneq M_1 \subsetneq \cdots \subsetneq M_l = \sfB(I)
\]  
be a distinguished filtration of \( \sfB(I) \) with respect to the Schur basis, as constructed in the proof of \cite[Theorem 6.7]{23KLO}.
Since $\ch([M_i/M_{i-1}])$ is a Schur function, it is natural to  ask whether \( M_i/M_{i-1} \) admits a distinguished filtration with respect to $\{\scrS_\alpha \mid \alpha \models n\}$ or $\{\hscrS_\alpha \mid \alpha \models n\}$.
The validity of this question has been checked for values of $n$ up to $9$ with the aid of the computer program 
\textsc{SageMath}.

\appendix
\section{Examples of \( H_n(0) \)-modules that have no distinguished filtrations with respect to \( \{\scrS_\alpha \mid \alpha \models n\} \) or \( \{\hscrS_\alpha \mid \alpha \models n\} \).}\label{no filtration example}

In this appendix, we present an example to show that when \( P \) is a regular Schur labeled poset with the underlying set \( [n] \), the associated \( H_n(0) \)-module \( \sfM_P \), as defined in \cite[Definition 2.8]{23KLO}, does not necessarily admit a distinguished filtration with respect to \( \{\scrS_\alpha \mid \alpha \models n\} \) or \( \{\hscrS_\alpha \mid \alpha \models n\} \).

Consider the $H_8(0)$-modules $X_{(5,2,1)}$ and its $\upphi$-twist $\upphi[X_{(5,2,1)}]$.
In view of \cite[Appendix]{23KLO}, one sees that each of these modules is isomorphic to the $H_8(0)$-module associated with a regular Schur labeled skew shape poset with underlying set $[8]$.

{\it Claim 1.} $X_{(5,2,1)}$ has no distinguished filtrations 
with respect to $\{\scrS_\alpha \mid \alpha \models 8 \}$.

We begin by observing that 
$$\ch([X_{(5,2,1)}])=s_{(5,2,1)}=\sum_{\substack{\alpha \models 8\\ \lambda(\alpha)=(5,2,1)}}\scrS_{\alpha}$$
and $X_{(5,2,1)} \cong \sfB(54321768, 87641523)$ as $H_8
(0)$-modules.
Suppose on the contrary that $\sfB(54321768, 87641523)$ has a distinguished filtration 
\[
0 =: M_0 \subsetneq M_1 \subsetneq \cdots \subsetneq M_6 := \sfB(54321768, 87641523)
\]
with respect to $\{\scrS_\alpha \mid \alpha \models 8 \}$.
We note that the sink of $\sfB(54321768, 87641523)$ is isomorphic to $\bfF_{(1,1,2,4)}$ and thus $[F_{(1,1,2,4)}]\ch([M_i]) > 0$ for $1\le i \le 6$.
Combining this with the fact that 
\[
[F_{(1,1,2,4)}]s_{(5,2,1)} = 1 \quad \text{and} \quad [F_{(1,1,2,4)}]\scrS_{(1,2,5)} = 1,
\]
we see that $\ch([M_1])=\scrS_{(1,2,5)}$ and $\dim M_1=16$.
Given a composition series of $M_1$
\[
0 =: N_0 \subsetneq N_1 \subsetneq \cdots \subsetneq N_{16} := M_1,
\]
let $\{v_1, v_2, \ldots, v_{16}\}$ be a basis for $M_1$ such that 
such that $\{v_1, v_2, \ldots, v_{j}\}$ is a basis for $N_j$ for $1\le j \le 16$.
Since $[F_{(2,2,4)}]\scrS_{(1,2,5)} = 1$, there exists a unique $1 \leq j \leq 16$ such that 
\begin{align}\label{eq: unique index j comp factor}
N_{j}/N_{j-1} \cong \bfF_{(2,2,4)}.
\end{align}
For each \(\sigma \in [54321768, 87641523]_L\), observe that \(\pi_{w_0(\{1,3,5,6,7\})} \cdot \sigma\) is one of \(0\), \(87421635\), or \(87621435\).
This implies that 
$$\pi_{w_0(\{1,3,5,6,7\})}v_j \in \C\{87421635, 87621435\}.$$
Let 
$$\pi_{w_0(\{1,3,5,6,7\})}v_j = c \,87421635 + d\, 87621435.$$
By \cref{eq: unique index j comp factor}, it holds that 
$\pi_{w_0(\{1,3,5,6,7\})}v_j - v_j \in N_{j-1}$ and $\pi_{w_0(\{1,3,5,6,7\})}v_j \neq 0$.
It follows that 
$$N_j = \C \{v_1,v_2, \ldots, v_{j-1}, c\,87421635 + d\,87621435\}.$$  
If $c = 0$ or $d =0$, then $L := \sfB(87621534, 87641523)$ is a submodule of $M_1$.
However, this cannot occur since 
$$[F_{(2,1,5)}]\scrS_{(1,2,5)} = 0 \quad \text{ and } \quad [F_{(2,1,5)}]\ch([L]) = 1.$$
Next, suppose that $c, d \neq 0$.
Since 
$$\pi_2 \cdot (c\,87421635 + d\,87621435) = c\,87431625 + d\,87631425$$
and 
$$\pi_3\pi_2 \cdot (c87421635 + d87621435) = c87431625,$$ 
it follows that $87631425 \in M_1$.
Therefore, $L' := \sfB(87631425, 87641523)$ is a submodule of $M_1$.
However, this cannot occur since
$$[F_{(1,2,1,4)}]\ch([L']) = 1 \quad \text{ and } \quad  [F_{(1,2,1,4)}]\scrS_{(1,2,5)} = 0.$$
For a pictorial description of \(L\) and \(L'\), along with the \(H_8(0)\)-action on \(\sfB(54321768, 87641523)\), see \cref{fig: appendix example}. 
Consequently, \(\sfB(54321768, 87641523) (\cong X_{(5,2,1)})\) does not admit a distinguished filtration with respect to \(\{\scrS_\alpha \mid \alpha \models 8\}\).

{\it Claim 2.} $\upphi[X_{(5,2,1)}]$
has no distinguished filtrations with respect to \(\{\hscrS_\alpha \mid \alpha \models 8\}\).

This follows from {\it Claim 1} since \(\upphi\) induces an equivalence on the category of finitely  generated \(H_n(0)\)-modules and \(\upphi[\bfF_\beta] \cong \bfF_{\beta^{\rm r}}\).

\begin{figure}
\centering
\begin{adjustbox}{addcode={\begin{minipage}{\width}}{\caption{The $H_8(0)$-action on the basis $[54321768, 87641523]_L$ for $\sfB(54321768, 87641523) \cong X_{(5,2,1)}$ and bases for $L$ and $L'$
      }
\label{fig: appendix example}
\end{minipage}},rotate=90,center}
\includegraphics[scale = 0.58, angle = 0]{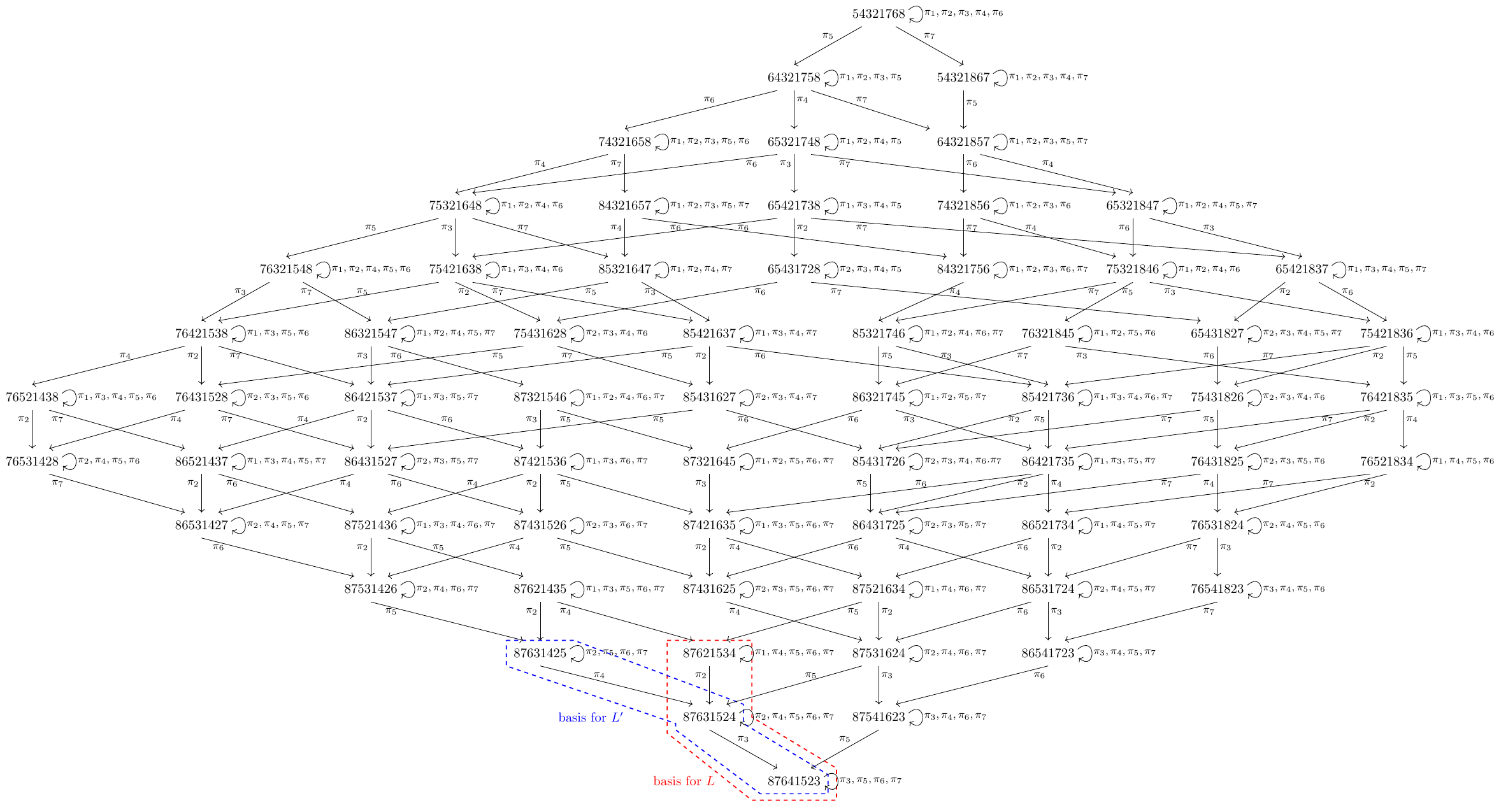}
\end{adjustbox}  

\end{figure}

\newpage
\paragraph{{\bf Acknowledgments.}}
The first author was supported by NRF grant funded by Basic Science Research Program through NRF funded by the Ministry of Education (No. RS-2023-00271282),  NRF grant funded by the Korea government(MSIT) (No. RS-2024-00342349), and the Sogang University Research Grant of 2024(No. 202412001.01).
The second author was supported by NRF grant funded by the Korea government(MSIT) (No. RS-2024-00342349).
\bibliographystyle{abbrv}

\end{document}